\pdfoutput=1
\RequirePackage{ifpdf}
\ifpdf 
\documentclass[pdftex]{sigma}
\else
\documentclass{sigma}
\fi

\numberwithin{equation}{section}

\newtheorem{Theorem}{Theorem}[section]
\newtheorem{Corollary}[Theorem]{Corollary}
\newtheorem{Lemma}[Theorem]{Lemma}
\newtheorem{Proposition}[Theorem]{Proposition}
 { \theoremstyle{definition}
\newtheorem{Example}[Theorem]{Example}
\newtheorem{Remark}[Theorem]{Remark} }

\usepackage{bbm}
\newcommand{\SP}[2]{\left\langle#1,#2\right\rangle} 
\newcommand{\set}[1]{{\mbox{\boldmath{$#1$}}}}
\newcommand{\seq}[1]{\underline{\mbox{\boldmath{$#1$}}}}
\newcommand\DS{\displaystyle}
\newcommand\CC{\mathbb{C}}
\newcommand\RR{\mathbb{R}}

\newcommand\NN{\mathbb{N}}
\newcommand\TT{\mathbb{T}}

\newcommand\DD{\mathbb{D}}
\newcommand\EE{\mathbb{E}}

\renewcommand\P{\mathcal{P}}

\renewcommand\L{\mathcal{L}}

\newcommand\diag{\mathrm{diag}}
\renewcommand\d{\mathrm{d}}
\renewcommand\epsilon{\varepsilon}
\newcommand\G{\mathcal{G}}
\newcommand\R{\mathcal{R}}
\newcommand\M{\mathcal{M}}
\newcommand\A{\mathcal{A}}
\newcommand\C{\mathcal{C}}
\newcommand\I{\mathcal{I}}
\newcommand\T{\mathcal{T}}
\newcommand\U{\mathcal{U}}
\newcommand\V{\mathcal{V}}
\newcommand\N{\mathcal{N}}

\newcommand\cS{\mathcal{S}}
\newcommand\cH{\mathcal{H}}
\newcommand\supp{\mathrm{supp}\,}
\newcommand\Span{\mathop{\rm span}\nolimits}

\newcommand{\oa}{\overline{\alpha}}
\newcommand{\ob}{\overline{\beta}}
\newcommand\dB{\dot{B}}
\newcommand\dvarpi{\dot{\varpi}}
\newcommand\dzeta{\dot{\zeta}}
\newcommand\dvarsigma{\dot{\varsigma}}
\newcommand\dpi{\dot{\pi}}

\newcommand\dL{\dot{\mathcal{L}}}

\newcommand\ed{\mathrel{\vcenter{\baselineskip0.5ex \lineskiplimit0pt
 \hbox{\scriptsize.}\hbox{\scriptsize.}}}%
 =}
\newcommand\Dprod{{%
 \setbox0\hbox{$\DS\mathop\prod$}%
 \rlap{\hbox to \wd0{\hss$\cdot$\hss}}\box0
}}
\newcommand\dprod{{%
 \setbox0\hbox{$\mathop\prod$}%
 \rlap{\hbox to \wd0{\hss$\cdot$\hss}}\box0
}}
\let\tilde\widetilde
\let\hat\widehat

\begin{document}

\allowdisplaybreaks

\newcommand{\arXivNumber}{1707.09748}

\renewcommand{\thefootnote}{}

\renewcommand{\PaperNumber}{090}

\FirstPageHeading

\ShortArticleName{Orthogonal Rational Functions on the Unit Circle}
\ArticleName{Orthogonal Rational Functions on the Unit Circle\\ with Prescribed Poles not on the Unit Circle\footnote{This paper is a~contribution to the Special Issue on Orthogonal Polynomials, Special Functions and Applications (OPSFA14). The full collection is available at \href{https://www.emis.de/journals/SIGMA/OPSFA2017.html}{https://www.emis.de/journals/SIGMA/OPSFA2017.html}}}

\Author{Adhemar BULTHEEL~$^\dag$, Ruyman CRUZ-BARROSO~$^\ddag$ and Andreas LASAROW~$^\S$}

\AuthorNameForHeading{A.~Bultheel, R.~Cruz-Barroso and A. Lasarow}

\Address{$^\dag$~Department of Computer Science, KU Leuven, Belgium}
\EmailD{\href{mailto:adhemar.bultheel@cs.kuleuven.be}{adhemar.bultheel@cs.kuleuven.be}}
\URLaddressD{\url{https ://people.cs.kuleuven.be/~adhemar.bultheel/}}

\Address{$^\ddag$~Department of Mathematical Analysis, La Laguna University, Tenerife, Spain}
\EmailD{\href{mailto:rcruzb@ull.es}{rcruzb@ull.es}}

\Address{$^\S$~Fak.~Informatik, Mathematik \& Naturwissenschaften, HTWK Leipzig, Germany}
\EmailD{\href{lasarow@imn.htwk-leipzig.de}{lasarow@imn.htwk-leipzig.de}}

\ArticleDates{Received August 01, 2017, in f\/inal form November 20, 2017; Published online December 03, 2017}

\Abstract{Orthogonal rational functions (ORF) on the unit circle generalize orthogonal polynomials (poles at inf\/inity) and Laurent polynomials (poles at zero and inf\/inity). In this paper we investigate the properties of and the relation between these ORF when the poles are all outside or all inside the unit disk, or when they can be anywhere in the extended complex plane outside the unit circle. Some properties of matrices that are the product of elementary unitary transformations will be proved and some connections with related algorithms for direct and inverse eigenvalue problems will be explained.}

\Classification{30D15; 30E05; 42C05; 44A60}

\renewcommand{\thefootnote}{\arabic{footnote}}
\setcounter{footnote}{0}

\section{Introduction}
Orthogonal rational functions (ORF) on the unit circle are well known as generalizations of orthogonal polynomials on the unit circle (OPUC).
The pole at inf\/inity of the polynomials is replaced by poles ``in the neighborhood'' of inf\/inity, i.e., poles outside the closed unit disk.
The recurrence relations for the ORF generalize the Szeg\H{o} recurrence
relations for the polynomials.

If $\mu$ is the orthogonality measure supported on the unit circle, and
$L^2_\mu$ the corresponding Hilbert space, then the shift operator
$\T_\mu\colon L^2_\mu\to L^2_\mu\colon f(z)\mapsto zf(z)$ restricted to the polynomials
has a representation
with respect to the orthogonal polynomials that is a Hessenberg matrix.
However, if instead of a polynomial basis, one uses a basis of orthogonal
Laurent polynomials by alternating between poles at inf\/inity and poles at the
origin, a full unitary representation of~$\T_\mu$ with respect to this basis
is a f\/ive-diagonal CMV matrix \cite{MISCCMV05}.

The previous ideas have been generalized to the rational case by Vel\'azquez
in \cite{MISCVel07}. He showed that the representation of the shift operator
with respect to the classical ORF
is not a Hessenberg matrix but a matrix M\"obius transform of a
Hessenberg matrix. However, a full unitary representation
can be obtained if the shift is represented with respect to a rational analog
of the Laurent polynomials by alternating between a pole inside and a pole
outside the unit disk. The resulting matrix is a matrix M\"obius transform of
a f\/ive-diagonal matrix.

Orthogonal Laurent polynomials on the real line, a half-line, or an interval were
introduced by Jones et al.\ \cite{PADJNT83,NPJNT84} in the context of moment problems, Pad\'e approximation and quadrature and this was elaborated by many authors.
Gonz\'alez-Vera and his co-workers were in particular involved in extending the
theory where the poles zero and inf\/inity alternate (the so-called balanced situation) to a more general case where in each step
either inf\/inity or zero can be chosen as a pole in any arbitrary order
\cite{ArtBDGO98, MISCDMGVJP06}.
They also identify the resulting orthogonal Laurent polynomials as shifted
versions of the orthogonal polynomials.
Hence the orthogonal Laurent polynomials
satisfy the same recurrence as the classical orthogonal
polynomials after an appropriate shifting and normalization is embedded.

The corresponding case of orthogonal Laurent polynomials on the unit
circle was introduced by Thron in \cite{MISCThr88}
and has been studied more recently in
for example \cite{MISCBRGV06,NPCRGV05}.
Papers traditionally deal with the balanced situation like in
\cite{NPCRGV05} but in \cite{MISCBRGV06} also an
arbitrary ordering was considered. Only in
\cite{MISCBD07} Cruz-Barroso and Delvaux investigated the structure of the matrix representation with respect
to the basis of the resulting orthogonal Laurent polynomials on the circle.
They called it a
``snake-shaped'' matrix which generalizes the f\/ive diagonal matrix.

The purpose of this paper is to generalize these ideas valid for
Laurent polynomials on the circle to the rational case.
That is to choose the poles of the
ORF in an arbitrary order either inside or outside the unit disk. We relate
the resulting ORF with the ORF having all their poles outside or all their
poles inside the disk, and study the corresponding recurrence relations.
With respect to this new orthogonal rational basis, the shift operator will
be represented by a matrix M\"obius transformation of a snake-shaped matrix.

In the papers by Lasarow and coworkers
(e.g., \cite{NPFKL02,NPFKL03c,NPFKL12, NPLas00})
matrix versions of the ORF are considered. In these papers
also an arbitrary choice of the poles is allowed but only with the
restrictive condition that if $\alpha$ is used as a pole, then $1/\oa$ cannot
be used anymore. This means that for example the ``balanced situation'' is
excluded. One of the goals of this paper is to remove this restriction on
the poles.

\looseness=1 In the context of quadrature formulas, an arbitrary sequence of poles
not on the unit circle was also brief\/ly
discussed in \cite{MISCCYGV07}.
The sequence of poles considered there need not be Newtonian, i.e.,
the poles for the ORF of degree $n$ may depend on $n$.
Since our approach will emphasize the role of the recurrence relation for
the ORF, we do need a Newtonian sequence, although some of the results
may be generalizable to the situation of a non-Newtonian sequence of poles.

One of the applications of the theory of ORF is the construction of quadrature
formulas on the unit circle that form rational generalizations of the Szeg\H{o}
quadrature. They are exact in spaces of rational functions having poles inside
and outside the unit disk. The nodes of the quadrature formula are zeros of
para-orthogonal rational functions (PORF) and the weights are all positive
numbers.
These nodes and weights can (like in Gaussian quadrature) be derived from
the eigenvalue decomposition of a unitary truncation of the shift operator
to a f\/inite-dimensional subspace.
One of the results of the paper is that there is no gain in considering
an arbitrary sequence of poles inside and outside the unit disk unless in a balanced situation.
When all the poles are chosen outside the closed unit disk or when some
of them are ref\/lected in the circle, the same
quadrature formula will be obtained.
The computational ef\/fort for the general case will not increase but neither can it reduce
the cost.

In network applications or dif\/ferential equations one often has to work with functions of large sparse
matrices. If $A$ is a matrix and the matrix function $f(A)$ allows the Cauchy representation
$f(A)=\int_\Gamma f(z)(z-A)^{-1}\d\mu(z)$,
where $\Gamma$ is a contour encircling all the eigenvalues of $A$
then numerical quadrature is a possible technique to obtain an approximation for $f(A)$.
If for example~$\Gamma$ is the unit circle, then expressions like $u^*f(A)u$ for some
vector $u$ can be approximated by quadrature formulas discussed
in this paper which will be implemented disguised as Krylov subspace methods (see for example \cite{MISCGM09b,MISCGut13,MISCKS10}).

The purpose of the paper though is not to discuss quadrature in particular.
It is just an example application that does not require much extra introduction of new terminology and notation.
The main purpose however is to give a general framework on which to build for the many
applications of ORFs.
Just like orthogonal polynomials are used in about every branch of mathematics,
ORFs can be used with the extra freedom to exploit the location of the poles.
For example, it can be shown that the ORFs can be used to solve multipoint moment problems
as well as more general rational interpolation problems where locations of the poles
inside and outside the circle are important for the engineering applications like
system identif\/ication, model reduction, f\/iltering, etc. When modelling the transfer function
of a linear system,
poles should be chosen inside as well as outside the disk to guarantee that the transient
as well as the steady state of the system is well modelled. It would lead us too far to also
include the interpolation properties of multipoint Pad\'e
approximation and the related applications in several branches of engineering.
We only provide the basics in this paper so that it can be used in the context of more applied papers.

The interpretation of the recursion for the ORFs as a factorization of a matrix
into elementary unitary transformations illustrates that the spectrum of the
resulting matrix is independent of the order in which the elementary factors are
multiplied. As far as we know, this fact was previously unknown in the linear algebra community,
unless in particular cases like unitary Hessenberg matrices.
As an illustration, we develop some preliminary results in Section~\ref{secAMPD}
in a~linear algebra setting that is slightly more general than the ORF situation.

In the last decades, many papers appeared on inverse eigenvalue problems for unitary
Hessenberg matrices and rational Krylov methods. Some examples are
\cite{MISCBG12,MISCGK13,MISCMVBVB13,MISCMVBVB14,MISCMer16,MISCMV14,MISCBMM13}.
These use elementary operations that are very closely related to
the recurrence that will be discussed in this paper.
However, they are not the same and often miss the f\/lexibility discussed here.
We shall illustrate some of these connections with certain algorithms from the literature in Section~\ref{secAF}.

The outline of the paper is as follows.
In Section~\ref{secnot} we introduce the main notations used in this paper.
The linear spaces and the ORF bases are given in Section~\ref{secLSOB}.
Section~\ref{secCD} brings the Christof\/fel--Darboux relations
and the reproducing kernels which form an essential element to obtain
the recurrence relation given in Section~\ref{secRR} but also for the
PORF in Section~\ref{secPORF} to be used for quadrature formulas
in Section~\ref{secQF}.
The alternative representation of the shift operator is given in
Section~\ref{secHB} and its factorization in elementary $2\times2$ blocks in
the subsequent Section~\ref{secMF}.
We end by drawing some conclusions about the spectrum of the shift operator
and about the computation of rational Szeg\H{o} quadrature formulas in
Section~\ref{secSA}.
The ideas that we present in this paper, especially the factorization of
unitary Hessenberg matrices in elementary unitary factors is also used
in the linear algebra literature mostly in the f\/inite-dimensional situation.
These elementary factors and what can be said about the spectrum of their product
is the subject of Section~\ref{secAMPD}.
These elementary unitary transformations are intensively used in numerical algorithms such as
Arnoldi-based Krylov methods where they are known as core transformations.
Several variants of these rational Krylov methods exist.
The algorithms are quite similar yet dif\/ferent from our ORF recursion as we explain brief\/ly in Section~\ref{secAF} illustrating why we believe the version presented in this paper has superior advantages.

\section{Basic def\/initions and notation}\label{secnot}
We use the following notation.
$\CC$ denotes the complex plane, $\hat{\CC}$ the extended complex plane (one point compactif\/ication),
$\RR$ the real line, $\hat{\RR}$ the closure of $\RR$ in $\hat{\CC}$,
$\TT$ the unit circle,
$\DD$ the open unit disk, $\hat{\DD}=\DD\cup\TT$, and
$\EE=\hat{\CC}\setminus\hat{\DD}$.
For any number $z\in\hat{\CC}$ we def\/ine $z_*=1/\overline{z}$ (and set $1/0=\infty$, $1/\infty=0$)
and for any complex function $f$, we def\/ine $f_*(z)=\overline{f(z_*)}$.

To approximate an integral
\begin{gather*}
I_\mu(f)=\int_{\TT}f(z)\d\mu(z),
\end{gather*}
where $\mu$ is a probability measure on $\TT$ one may use Szeg\H{o} quadrature formulas.
The nodes of this quadrature can be computed by using the Szeg\H{o} polynomials.
Orthogonality in this paper will always be with respect to the inner product
\begin{gather*}
\SP{f}{g}=\int_{\TT} \overline{f(z)}g(z)\d\mu(z).
\end{gather*}
The weights of
the $n$-point quadrature are all positive, the nodes are on $\TT$ and the formula is exact for all
Laurent polynomials $f\in\Span\{z^k\colon |k|\le n-1\}$.

This has been generalized to rational functions with a set of predef\/ined poles.
The corresponding quadrature formulas are then rational Szeg\H{o} quadratures.
This has been discussed in many papers and some of the earlier results
were summarized in the book \cite{BooksBGHN99}.
We brief\/ly recall some of the results that are derived there.
The idea is the following.
Fix a sequence $\seq{\alpha}=(\alpha_k)_{k\in\NN}$ with $\set{\alpha}=\{\alpha_k\}_{k\in\NN}\subset\DD$,
and consider the subspaces of rational functions def\/ined by
\begin{gather*}
\L_0=\CC,\qquad \L_n=\left\{\frac{p_n(z)}{\pi_n(z)}\colon p_n\in\P_n,\ \pi_n(z)=\prod_{k=1}^n(1-\oa_kz)\right\},\qquad n\ge1,
\end{gather*}
where $\P_n$ is the set of polynomials of degree at most $n$.
These rational functions have their poles among the points
in $\set{\alpha}_*=\{\alpha_{j*}=1/\oa_j\colon \alpha_j\in\set{\alpha}\}$.
We denote the corresponding sequence as $\seq{\alpha}_*=(\alpha_{j*})_{j\in\NN}$.
Let $\phi_n\in\L_n\setminus\L_{n-1}$, and $\phi_n\perp\L_{n-1}$ be the
$n$th orthogonal rational basis function (ORF) in a nested sequence.
It is well known that these functions have all their zeros in $\DD$
(see, e.g., \cite[Corollary~3.1.4]{BooksBGHN99}).
However, the quadrature formulas we
have in mind should have their nodes on the circle $\TT$.
Therefore, para-orthogonal rational functions (PORF) are introduced. They are def\/ined by
\begin{gather*}
Q_n(z,\tau)=\phi_n(z)+\tau\phi_n^*(z),\qquad \tau\in\TT,
\end{gather*}
where besides the ORF $\phi_n(z)=\frac{p_n(z)}{\pi_n(z)}$, also the
``reciprocal''
function $\phi_n^*(z)=\frac{p_n^*(z)}{\pi_n(z)}=\frac{z^np_{n*}(z)}{\pi_n(z)}$ is introduced.
These PORF have $n$ simple zeros $\{\xi_{nk}\}_{k=1}^n\subset\TT$ (see, e.g., \cite[Theorem~5.2.1]{BooksBGHN99})
so that they can be used as nodes for the quadrature formulas
\begin{gather*}
I_n(f)=\sum_{k=1}^n w_{nk}f(\xi_{nk})
\end{gather*}
and the weights are all positive, given by $w_{nk}=1/\sum\limits_{j=0}^{n-1}|\phi_j(\xi_{nj})|^2$ (see, e.g., \cite[Theorem~5.4.2]{BooksBGHN99}).
These quadrature formulas are exact for all functions of the form $\{f=g_*h\colon\! g,h\in\L_{n-1}\}=\L_{n-1}\L_{(n-1)*}$ (see, e.g., \cite[Theorem~5.3.4]{BooksBGHN99}).

The purpose of this paper is to generalize the situation where the $\alpha_j$ are all in $\DD$
to the situation where they are anywhere in the extended complex plane outside $\TT$.
This will require the introduction of some new notation.

So consider a sequence $\seq{\alpha}$ with $\set{\alpha}\subset\DD$ and its ref\/lection in the circle
$\seq{\beta}=(\beta_j)_{j\in\NN}$ where $\beta_j=1/\oa_j=\alpha_{j*}\in\EE$.
We now construct a new sequence $\seq{\gamma}=(\gamma_j)_{j\in\NN}$ where each $\gamma_j$ is
either equal to $\alpha_j$ or $\beta_j$.

Partition $\{1,2,\ldots,n\}$ ($n\in\hat{\NN}=\NN\cup\{\infty\}$) into two disjoint index sets: the ones where $\gamma_j=\alpha_j$ and the indices where $\gamma_j=\beta_j$:
\begin{gather*}
\mathbbm{a}_n=\{j\colon \gamma_j=\alpha_j\in\DD,~1\le j\le n\}\qquad \mbox{and}\qquad
\mathbbm{b}_n=\{j\colon \gamma_j=\beta_j\in\EE,~1\le j\le n\},
\end{gather*}
and def\/ine
\begin{gather*}
\set{\alpha}_n=\{\alpha_j\colon j\in\mathbbm{a}_n\}\qquad \mbox{and}\qquad
\set{\beta}_n=\{\beta_j\colon j\in\mathbbm{b}_n\}.
\end{gather*}

It will be useful to prepend the sequence $\seq{\alpha}$ with an extra point $\alpha_0=0$.
That means that $\seq{\beta}$ is preceded by $\beta_0=1/\oa_0=\infty$.
For $\seq{\gamma}$, the initial point can be $\gamma_0=\alpha_0=0$ or $\gamma_0=\beta_0=\infty$.

With each of the sequences $\seq{\alpha}$, $\seq{\beta}$, and $\seq{\gamma}$ we can associate
orthogonal rational functions. They will be closely related as we shall show.
The ORF for the $\seq{\gamma}$ sequence can be derived from the ORF
for the $\seq{\alpha}$ sequence
by multiplying with a Blaschke product just like the orthogonal
Laurent polynomials are essentially shifted versions of the orthogonal polynomials (see, e.g., \cite{MISCBRGV06}).

To def\/ine the denominators of our rational functions, we introduce the
following elementary factors:
\begin{gather*}
\varpi_j^\alpha(z)=1-\oa_jz,
\qquad\!
\varpi_j^\beta(z)=
\begin{cases}
 1-\ob_jz, & \text{if }\beta_j\ne\infty,\\
 -z, & \text{if }\beta_j=\infty,
 \end{cases}
\qquad\!
\varpi_j^\gamma(z)=
\begin{cases}
 \varpi_j^\alpha(z), & \text{if }\gamma_j=\alpha_n,\\
 \varpi_j^\beta(z), & \text{if }\gamma_j=\beta_n.
 \end{cases}
\end{gather*}
Note that if $\alpha_j=0$ and hence $\beta_j=\infty$ then $\varpi_j^\alpha(z)=1$ but $\varpi_j^\beta(z)=-z$.

To separate the $\alpha$ and the $\beta$-factors in a product, we also def\/ine
\begin{gather*}
 \dvarpi_j^\alpha(z)= \begin{cases}
 \varpi_j^\alpha, & \text{if }\gamma_j=\alpha_j,\\
 1, & \text{if }\gamma_j=\beta_j ,
 \end{cases}
\qquad\mbox{and}\qquad
 \dvarpi_j^\beta(z)= \begin{cases}
 \varpi_j^\beta, & \text{if }\gamma_j=\beta_j,\\
 1, & \text{if }\gamma_j=\alpha_j.
 \end{cases}
\end{gather*}
Because the sequence $\seq{\gamma}$ is our main focus, we simplify the notation by removing the super\-script~$\gamma$ when not needed, e.g., $\varpi_j=\varpi_j^\gamma=\dvarpi_j^\alpha\dvarpi_j^\beta$ etc.

We can now def\/ine for $\nu\in\{\alpha,\beta,\gamma\}$
\begin{gather*}
\pi_n^\nu(z)=\prod_{j=1}^n\varpi_j^\nu(z)
\end{gather*}
and the reduced products separating the $\alpha$ and the $\beta$-factors
\begin{gather*}
\dpi_n^\alpha(z)=\prod_{j=1}^n\dvarpi_j^\alpha(z)=\prod_{j\in\mathbbm{a}_n}\varpi_j(z),\qquad
\dpi_n^\beta(z)=\prod_{j=1}^n\dvarpi_j^\beta(z)=\prod_{j\in\mathbbm{b}_n}\varpi_j(z),
\end{gather*}
so that
\begin{gather*}
\pi_n(z)=\prod_{j=1}^n \varpi_j(z)=\dpi_n^\alpha(z)\dpi_n^\beta(z).
\end{gather*}
We assume here and in the rest of the paper that products over $j\in\varnothing$ equal 1.

The Blaschke factors are def\/ined for $\nu\in\{\alpha,\beta,\gamma\}$ as
\begin{alignat*}{4}
&\zeta^\nu_j(z)=\sigma_j^\nu\DS\frac{z-\nu_j}{1-\overline{\nu}_jz},\qquad &&
\sigma^\nu_j=\frac{\overline{\nu}_j}{|\nu_j|},\qquad &&
\text{if }\nu_j\not\in\{0,\infty\},&\\
&\zeta^\nu_j(z)=\sigma^\nu_jz=z,\qquad && \sigma^\nu_j=1,\qquad && \text{if }\nu_j=0,&\\
&\zeta^\nu_j(z)=\sigma^\nu_j/z=1/z,\qquad && \sigma^\nu_j=1,\qquad && \text{if }\nu_j=\infty.&
\end{alignat*}
Thus
\begin{gather*}
\sigma^\nu_j=
	\begin{cases}
		\dfrac{\overline{\nu}_j}{|\nu_j|}, &\text{for}~~\nu_j\not\in\{0,\infty\},\\
		1,&\text{for}~~\nu_j\in\{0,\infty\}.
	\end{cases}
\end{gather*}
Because $\sigma_n^\alpha=\sigma_n^\beta$, we can remove the superscript and just write $\sigma_n$.
If we also use the following notation which maps complex numbers onto $\TT$
\begin{gather*}
\mathfrak{u}(z)= \begin{cases}
\dfrac{\overline{z}}{|z|}\in\TT, & z\in\CC\setminus\{0\},\\
1, & z\in\{0,\infty\},
\end{cases}
\end{gather*}
then $\sigma_j=\mathfrak{u}(\alpha_j)=\mathfrak{u}(\beta_j)=\mathfrak{u}(\gamma_j)$.

Set $(\varpi_j^\nu)^*(z)=\varpi_j^{\nu*}(z)=z\varpi^\nu_{j*}(z)$
(e.g., $(1-\oa_j z)^*=z-\alpha_j$ if $ \nu=\alpha$),
then $\zeta_j^\nu=\sigma_j\frac{\varpi_j^{\nu*}}{\varpi_j^\nu}$.
Later we shall also use $\pi_n^{\nu*}$ to mean $\prod\limits_{j=1}^n \varpi_j^{\nu*}$.
Note that
 $\zeta_j^\alpha=\zeta_{j*}^\beta=1/\zeta_j^\beta$.
Moreover if $\alpha_j=0$ and hence $\beta_j=\infty$, then $\varpi_j^{\alpha*}(z)=z$ and
$\varpi_j^{\beta*}(z)=-1$.

Next def\/ine the f\/inite Blaschke products for $\nu\in\{\alpha,\beta\}$
\begin{gather*}
B_0^\nu=1,\qquad \mbox{and}\qquad B_n^\nu(z)=\prod_{j=1}^n\zeta_j^\nu(z),\qquad
n=1,2,\ldots.
\end{gather*}
{\it It is important to note that here $\nu\ne\gamma$. For the definition of $B_n^\gamma=B_n$ see below.}

Like we have split up the denominators $\pi_n=\dpi_n^\alpha\dpi_n^\beta$ in the $\alpha$-factors and the $\beta$-factors, we also
def\/ine for $n\ge1$
\begin{gather*}
\dzeta_j^\alpha= \begin{cases}\zeta_j^\alpha,&\mbox{if~~}\gamma_j=\alpha_j,\\1,&\mbox{if~~}\gamma_j=\beta_j,\end{cases}\qquad
\dzeta_j^\beta= \begin{cases}\zeta_j^\beta,&\mbox{if~~}\gamma_j=\beta_j,\\1,&\mbox{if~~}\gamma_j=\alpha_j,\end{cases}
\end{gather*}
and
\begin{gather*}
\dB_n^\alpha(z)=\prod_{j=1}^n\dzeta_j^\alpha(z)=\prod_{j\in\mathbbm{a}_n}\zeta_j(z),\qquad\mbox{and}\qquad
\dB_n^\beta(z)=\prod_{j=1}^n\dzeta_j^\beta(z)=\prod_{j\in\mathbbm{b}_n}\zeta_j(z),
\end{gather*}
so that we can def\/ine the f\/inite Blaschke products for the $\seq{\gamma}$ sequence:
\begin{gather*}
B_n(z)= \begin{cases}
\dB_n^\alpha(z), & \mbox{if~~}\gamma_n=\alpha_n,\\
\dB_n^\beta(z), & \mbox{if~~}\gamma_n=\beta_n.
\end{cases}
\end{gather*}

Note that the ref\/lection property of the factors also holds for the products:
$B_n^\alpha=(B_{n}^\beta)_*=1/B_n^\beta$, $B_{n*}= 1/B_n$, and
$(\dB_n^\alpha\dB_n^\beta)_*=1/(\dB_n^\beta\dB_n^\alpha)$.
However,
\begin{gather*}
\dB_n^\alpha
= \prod\limits_{j\in\mathbbm{a}_n}\zeta_{j}^\alpha
= \prod\limits_{j\in\mathbbm{a}_n}\zeta_{j*}^\beta
= \prod\limits_{j\in\mathbbm{a}_n}1/\zeta_{j}^\beta
\ne \prod\limits_{j\in\mathbbm{b}_n}1/\zeta_{j}^\beta=1/\dB_n^\beta.
\end{gather*}

\section{Linear spaces and ORF bases}\label{secLSOB}
We can now introduce our spaces of rational functions for $n\ge 0$:
\begin{gather*}
\L_n^\nu=\Span\{B_0^\nu,B_1^\nu,\ldots,B_n^\nu\},\qquad \nu\in\{\alpha,\beta,\gamma\},\qquad \mbox{and}\\
\dL_n^\nu=\Span\{\dB_0^\nu,\dB_1^\nu,\ldots,\dB_n^\nu\},\qquad \nu\in\{\alpha,\beta\}.
\end{gather*}
The dimension of $\L_n^\nu$ is $n+1$ for $\nu\in\{\alpha,\beta,\gamma\}$, but
note that the dimension of $\dL_n^\nu$ for $\nu\in\{\alpha,\beta\}$ can be less than $n+1$. Indeed some of the $\dB_j^\nu$ may be repeated so that
for example the dimension of $\dL_n^\alpha$ is only $|\mathbbm{a}_n|+1$ with $|\mathbbm{a}_n|$
the cardinality of $\mathbbm{a}_n$ and similarly for $\nu=\beta$.
Hence for $\nu=\gamma$:
\begin{gather*}
\L_n=\Span\{B_0,\ldots,B_n\}=\Span\big\{\dB_0,\dB_1^\alpha,\ldots,\dB_n^\alpha, \dB_1^\beta,\ldots,\dB_n^\beta\big\}=
\dL_n^\alpha+\dL_n^\beta=\dL_n^\alpha\dL_n^\beta.
\end{gather*}

Because for $n\ge1$
\begin{gather*}
\dB_n^\alpha=
\prod_{j\in\mathbbm{a}_n}\zeta_j^\alpha=
\prod_{j\in\mathbbm{a}_n}\frac{1}{\zeta_j^\beta}\qquad\mbox{and}\qquad
\dB_n^\beta=
\prod_{j\in\mathbbm{b}_n}\zeta_j^\beta=
\prod_{j\in\mathbbm{b}_n}\frac{1}{\zeta_j^\alpha},
\end{gather*}
 it should be clear that $B_k^\alpha=\dB_k^\alpha/\dB_k^\beta$ and $B_k^\beta=\dB_k^\beta/\dB_k^\alpha$, hence that
\begin{gather*}
\L_n^\alpha=\Span\left\{\dB_0,\frac{\dB_1^\alpha}{\dB_1^\beta},\ldots,\frac{\dB_n^\alpha}{\dB_n^\beta}\right\}
\qquad \mbox{and}\qquad
\L_n^\beta=\Span\left\{\dB_0,\frac{\dB_1^\beta}{\dB_1^\alpha},\ldots,\frac{\dB_n^\beta}{\dB_n^\alpha}\right\}.
\end{gather*}
Occasionally we shall also need the notation
\begin{gather*}
\dvarsigma_n^\alpha=\prod_{j\in\mathbbm{a}_n}\sigma_j\in\TT,\qquad
\dvarsigma_n^\beta=\prod_{j\in\mathbbm{b}_n}\sigma_j\in\TT,\qquad\mbox{and}\qquad
\varsigma_n=\prod_{j=1}^n \sigma_j\in\TT.
\end{gather*}

\begin{Lemma}\label{lem1}
If $f\in\L_n$ then $f/\dB_n^\beta\in\L_n^\alpha$ and $f/\dB_n^\alpha\in\L_n^\beta$.
In other words $\L_n=\dB_n^\beta\L_n^\alpha=\dB_n^\alpha\L_n^\beta$.
This is true for all $n\ge0$ if we set $\dB_0^\alpha=\dB_0^\beta=1$.
\end{Lemma}
\begin{proof}
This is trivial for $n=0$ since then $\L_n=\CC$.
If $f\in\L_n$, and $n\ge1$ then it is of the form
\begin{gather*}
f(z)=\frac{p_n(z)}{\pi_n(z)}=\frac{p_n(z)}{\dpi_n^\alpha(z)\dpi_n^\beta(z)},\qquad p_n\in\P_n.
\end{gather*}
Therefore
\begin{gather*}
\frac{f(z)}{\dB_n^\beta(z)}=
\overline{\dvarsigma}_n^\beta\frac{p_n(z)\dpi_n^\beta(z)}{\dpi_n^\alpha(z)\dpi_n^\beta(z)\dpi_n^{\beta*}(z)}=
\overline{\dvarsigma}_n^\beta\frac{p_n(z)}{\dpi_n^\alpha(z)\dpi_n^{\beta*}(z)}.
\end{gather*}
Recall that $\varpi_j^{\beta*}=-1$ and $\sigma_j=1$ if $\beta_j=\infty$ (and hence $\alpha_j=0$), we can leave these
factors out and we shall write $\dprod$ for the product instead of $\prod$,
the dot meaning that we leave out
all the factors for which $\alpha_j=1/\overline{\beta}_j=0$.
\begin{gather*}
\frac{\overline{\dvarsigma}_n^\beta}{\dpi_n^{\beta*}(z)}=
\mathop\Dprod_{j\in\mathbbm{b}_n}\frac{\beta_j}{|\beta_j|(z-\beta_j)}=
\mathop\Dprod_{j\in\mathbbm{b}_n
}\frac{|\alpha_j|}{\oa_j(z-1/\oa_j)}=
\mathop\Dprod_{j\in\mathbbm{b}_n}\frac{-|\alpha_j|}{1-\oa_jz},
\end{gather*}
and thus
\begin{gather*}
\frac{f(z)}{\dB_n^\beta(z)}=c_n\frac{p_n(z)}{\prod\limits_{j=1}^n(1-\oa_jz)}\in\L_n^\alpha,\qquad
c_n=\mathop\Dprod_{j\in\mathbbm{b}_n}(-|\alpha_j|)\ne0.
\end{gather*}
The second part is similar.
\end{proof}

\begin{Lemma}\label{lem2}
With our previous definitions we have for $n\ge1$
\begin{gather*}
\dB_n^\beta\L_{n-1}^\alpha=
\Span\left\{B_k^\alpha\dB_n^\beta=\frac{\dB_k^\alpha}{\dB_k^\beta}\dB_n^\beta\colon k=0,\ldots,n-1\right\}\\
\hphantom{\dB_n^\beta\L_{n-1}^\alpha} =
\dzeta_n^\beta\Span\{B_0,B_1,\ldots,B_{n-1}\}=\dzeta_n^\beta\L_{n-1},
\end{gather*}
and similarly
\begin{gather*}
\dB_n^\alpha\L_{n-1}^\beta=
\Span\left\{B_k^\beta\dB_n^\alpha=\frac{\dB_k^\beta}{\dB_k^\alpha}\dB_n^\alpha\colon k=0,\ldots,n-1\right\}\\
\hphantom{\dB_n^\alpha\L_{n-1}^\beta}{}=
\dzeta_n^\alpha\Span\{B_0,B_1,\ldots,B_{n-1}\}=\dzeta_n^\alpha\L_{n-1}.
\end{gather*}
\end{Lemma}
\begin{proof}
By our previous lemma $\dB_n^\beta\L_{n-1}^\alpha=\dzeta_n^\beta\dB_{n-1}^\beta\L_{n-1}^\alpha=\dzeta_n^\beta\L_{n-1}$.
The second relation is proved in a similar way.
\end{proof}

To introduce the sequences of orthogonal rational functions (ORF)
for the dif\/ferent sequen\-ces~$\seq{\nu}$, $\nu\in\{\alpha,\beta,\gamma\}$
recall the inner product that we can write with our $(\ )_*$-notation as
$\SP{f}{g}=\int_{\TT} f_*(z)g(z)\d\mu(z)$ where $\mu$ is assumed to be
a probability measure positive a.e.\ on~$\TT$.

Then the orthogonal rational functions (ORF) with respect to the sequence $\seq{\nu}$ with $\nu\in\{\alpha,\beta,\gamma\}$ are def\/ined by
$\phi_n^\nu\in\L_n^\nu\setminus \L_{n-1}^\nu$ with $\phi_n^\nu\perp\L_{n-1}^\nu$ for $n\ge1$ and we choose $\phi_0^\nu=1$.

\begin{Lemma}\label{lem3}
The function
$\phi_n^\alpha \dB_n^\beta$ belongs to $\L_n$ and it is orthogonal to the $n$-dimensional subspace $\dzeta_n^\beta\L_{n-1}$ for all $n\ge1$.

Similarly, the function
$\phi_n^\beta \dB_n^\alpha$ belongs to $\L_n$ and it is orthogonal to the $n$-dimensional subspace $\dzeta_n^\alpha\L_{n-1}$, $n\ge1$.
\end{Lemma}
\begin{proof}
First note that $\phi_n^\alpha \dB_n^\beta\in\L_n$ by Lemma~\ref{lem1}.

By def\/inition $\phi_n^\alpha\perp \L_{n-1}^\alpha$.
Thus by Lemma~\ref{lem2} and because $\SP{f}{g}=\SP{\dB_n^\nu f}{\dB_n^\nu g}$,
\begin{gather*}
\dB_n^\beta\phi_n^\alpha\perp \dB_n^\beta \L_{n-1}^\alpha=\dzeta_n^\beta\L_{n-1}.
\end{gather*}
The second claim follows by symmetry.
\end{proof}

Note that $\dzeta_n^\beta\L_{n-1}=\L_{n-1}$ if $\gamma_n=\alpha_n$.
Thus, up to normalization, $\phi_n^\alpha \dB_n^\beta$ is the same as $\phi_n$ and similarly, if $\gamma_n=\beta_n$ then $\phi_n$ and $\phi_n^\beta \dB_n^\alpha$ are the same up to normalization.

\begin{Lemma}\label{lem4}
For $n\ge1$ the function $\dB_n^\alpha(\phi_n^\alpha)_*$ belongs to $\L_n$ and it is orthogonal to $\dzeta_n^\alpha\L_{n-1}$.

Similarly, for $n\ge 1$ the function $\dB_n^\beta(\phi_n^\beta)_*$ belongs to $\L_n$ and it is orthogonal to $\dzeta_n^\beta\L_{n-1}$.
\end{Lemma}
\begin{proof}
Since $\phi_n^\alpha\dB_n^\beta\perp \dzeta_n^\beta\L_{n-1}$,
\begin{gather*}
(\phi_n^\alpha\dB_n^\beta)_*\perp \dzeta_{n*}^\beta\L_{(n-1)*},
\end{gather*}
and thus by Lemma~\ref{lem2} and because
\begin{gather*}
\dB_{n-1}^\alpha\dB_{n-1}^\beta\L_{(n-1)*}=\dB_{n-1}^\alpha\dB_{n-1}^\beta\frac{\P_{(n-1)*}}{\dpi_{(n-1)*}^\alpha\dpi_{(n-1)*}^\beta}
=\frac{\P_{n-1}}{\dpi_{n-1}^\alpha\dpi_{n-1}^\beta}=\L_{n-1}
\end{gather*}
it follows that
\begin{gather*}
\dB_n^\alpha\phi_{n*}^\alpha=
\dB_n^\alpha\dB_n^\beta(\phi_n^\alpha\dB_n^\beta)_*\perp\dzeta_n^\alpha\dB_{n-1}^\alpha\dB_{n-1}^\beta\L_{(n-1)*}=\dzeta_n^\alpha\L_{n-1}.
\end{gather*}
The other claim follows by symmetry.
\end{proof}

We now def\/ine the reciprocal ORFs by (recall $f_*(z)=\overline{f(1/\overline{z})}$)
\begin{gather*}
(\phi_n^\nu)^*=B_n^\nu (\phi_n^\nu)_*,\qquad \nu\in\{\alpha,\beta\}.
\end{gather*}
For the ORF in $\L_n$ however we set
\begin{gather*}
\phi_n^* = \dB_n^\alpha\dB_n^\beta (\phi_n)_*.
\end{gather*}
Note that by def\/inition $B_n$ is either $\dB_n^\alpha$ or $\dB_n^\beta$
depending on $\gamma_n$ being $\alpha_n$ or $\beta_n$,
while in the previous def\/inition we do not multiply with $B_n$ but with the product $\dB_n^\alpha\dB_n^\beta$. The reason is that we want the operation $(\ )^*$
to be a map from $\L_n^\nu$ to $\L_n^\nu$ for all $\nu\in\{\alpha,\beta,\gamma\}$.
\begin{Remark}\rm\label{rem3.1}
As the operation $(\ )^*$ is a map from $\L_n^\nu$ to $\L_n^\nu$, it
depends on $n$ and on $\nu$. So to make the notation unambiguous we should in fact use something like $f^{[\nu,n]}$ if $f\in\L_n^\nu$. However, in order not to
overload our notation, we shall stick to the notation $f^*$ since it should
always be clear from the context what the space is to which $f$ will belong.
Note that we also used the same notation to transform polynomials.
This is just a special case of the general def\/inition.
Indeed, a polynomial of degree $n$ belongs to $\L_n^\alpha$
for a sequence $\seq{\alpha}$ where all $\alpha_j=0$, $j=0,1,2,\ldots$
and for this sequence $B_n^\alpha(z)=z^n$.

Note that for a constant $a\in\L_0=\hat{\CC}$ we have $a^*=\overline{a}$.
Although $(\ )^*$ is mostly used for scalar expressions, we shall occasionally use $A^*$ where $A$
is a matrix whose elements are all in~$\L_n$. Then the meaning is that we take the $(\ )^*$
conjugate of each element in its transpose. Thus if $A$ is a~constant matrix, then
$A^*$ has the usual meaning of the adjoint or complex conjugate transpose of the matrix.
We shall need this in Section~\ref{secHB}.
\end{Remark}

\begin{Remark}\rm\label{rem3.2}
It might also be helpful for the further computations to note the following.
If~$p_n$ is a polynomial of degree $n$ with a zero at $\xi$,
then $p_n^*$ will have a zero at $\xi_*=1/\overline{\xi}$.
Hence, if $\nu\in\{\alpha,\beta,\gamma\}$ and $\phi_n^\nu=\frac{p_n^\nu}{\pi_n^\nu}$,
then $\phi_{n*}^{\nu}=\frac{p_{n*}^{\nu}}{\pi_{n*}^\nu} = \frac{p_{n}^{\nu*}}{\pi_{n}^{\nu*}}$.
We know by \cite[Corollary~3.1.4]{BooksBGHN99} that~$\phi_n^\alpha$ has all its zeros in $\DD$, hence
$p_n^\alpha$ does not vanish in $\EE$ and $p_n^{\alpha*}$ does not vanish in $\DD$.
By symmet\-ry~$\phi_n^\beta$ has all its zeros in $\EE$ and $p_n^{\beta*}$ does not vanish in $\EE$.
For the general $\phi_n$, it depends on $\gamma_n$ being~$\alpha_n$ or~$\beta_n$.
However from the relations between $\phi_n$ and $\phi_n^\alpha$ or $\phi_n^\beta$ that will be derived below,
we will be able to guarantee that at least for $z=\nu_n$ we have $\phi_n^{\nu*}(\nu_n)\ne0$ and $p_n^{\nu*}(\nu_n)\ne0$ for all $\nu\in\{\alpha,\beta,\gamma\}$
(see Corollary~\ref{cor1} below).
\end{Remark}

The orthogonality conditions def\/ine $\phi_n$ and $\phi_n^*$ uniquely up to normalization.
So let us now make the ORFs unique by imposing an appropriate normalization.
First assume that from now on the $\phi_n^\nu$ refer to orthonormal functions
in the sense that $\|\phi_n^\nu\|=1$. This makes them unique up to
a unimodular constant. Def\/ining this constant is what we shall do now.

Suppose $\gamma_n=\alpha_n$, then $\phi_n$ and $\phi_n^\alpha\dB_n^\beta$ are both in $\L_n$
and orthogonal to $\L_{n-1}$ (Lemma~\ref{lem3}).
If we assume $\|\phi_n\|=1$ and $\|\phi_n^\alpha\|=1$, hence
$\|\phi_n^\alpha\dB_n^\beta\|=\|\phi_n^\alpha\|=1$, it follows that
there must be some unimodular constant $s_n^\alpha\in\TT$ such that
$\phi_n=s_n^\alpha \phi_n^\alpha\dB_n^\beta$.
Of course, we have by symmetry that for $\gamma_n=\beta_n$, there is some $s_n^\beta\in\TT$ such that $\phi_n=s_n^\beta \phi_n^\beta\dB_n^\alpha$.

To def\/ine the unimodular factors $s_n^\alpha$ and $s_n^\beta$,
we f\/irst f\/ix $\phi_n^\alpha$ and $\phi_n^\beta$ uniquely as follows.

We know that $\phi_n^\alpha$ has all its zeros in $\DD$ and hence
$\phi_n^{\alpha*}$ has all its zeros in $\EE$ so that
$\phi_n^{\alpha*}({\alpha}_n)\ne0$.
Thus we can take $\phi_n^{\alpha*}({\alpha}_n)>0$ as a normalization for $\phi_n^\alpha$.
Similarly for $\phi_n^\beta$ we can normalize by $\phi_n^{\beta*}({\beta}_n)>0$.
In both cases, we have made the leading coef\/f\/icient with respect to the basis
$\{B_j^\nu\}_{j=0}^n$ positive since
$\phi_n^\alpha(z) = \overline{\phi_n^{\alpha*}(\alpha_n)}B_n^\alpha(z)+\psi_{n-1}^\alpha(z)$ with $\psi_{n-1}^\alpha\in\L_{n-1}^\alpha$ and
$\phi_n^\beta(z) = \overline{\phi_n^{\beta*}(\beta_n)}B_n^\beta(z)+\psi_{n-1}^\beta(z)$ with $\psi_{n-1}^\beta\in\L_{n-1}^\beta$.
Before we def\/ine the normalization for the $\seq{\gamma}$ sequence,
we prove the following lemma which is a consequence of the normalization
of the $\phi_n^\alpha$ and the $\phi_n^\beta$.

\begin{Lemma}\label{lem5}
For the orthonormal ORFs, it holds that $\phi_n^\alpha=\phi_{n*}^\beta$ and $(\phi_n^\alpha)^*\dB_n^\beta=\phi_n^\beta\dB_n^\alpha$ and hence also $(\phi_n^\beta)^*\dB_n^\alpha=\phi_n^\alpha\dB_n^\beta$ for all $n\ge0$.
\end{Lemma}
\begin{proof}
For $n=0$, this is trivial since $\phi_0,\phi_0^\alpha,\phi_0^\beta,\dB_0^\alpha$ and $\dB_0^\beta$ are all equal to 1.

We give the proof for $n\ge1$ and $\gamma_n=\alpha_n$ (for $\gamma_n=\beta_n$, the proof is similar).
Since by previous lemmas $\dB_n^\beta(\phi_n^\beta)_*$ and $\phi_n^\alpha\dB_n^\beta$ are both
in $\L_n$ and orthogonal to $\L_{n-1}$, and since $\|\dB_n^\beta(\phi_n^\beta)_*\|=\|\phi_n^\beta\|=1$
and $\|\phi_n^\alpha\dB_n^\beta\|=\|\phi_n^\alpha\|=1$,
there must be some $s_n\in\TT$ such that
\begin{gather*}
s_n\phi_n^\alpha\dB_n^\beta=\phi_{n*}^\beta \dB_n^\beta \qquad \text{or}\qquad s_n\phi_n^\alpha=\phi_{n*}^\beta.
\end{gather*}
Multiply with $B_n^\beta=B_{n*}^\alpha$ and evaluate at ${\beta}_n$ to get
$s_n\phi_n^\alpha({\beta}_n)B_{n*}^\alpha({\beta}_n)=\phi_n^{\beta*}({\beta}_n)>0$.
Thus $s_n$ should arrange for{\samepage
\begin{gather*}
0<s_n\phi_n^\alpha(1/\overline{{\alpha}}_n)B_{n*}^\alpha(1/\overline{{\alpha}}_n)
= s_n\overline{\phi_{n*}^\alpha({\alpha}_n)}\overline{B_n^\alpha({\alpha}_n)}=
s_n\overline{\phi_n^{\alpha*}({\alpha}_n)},
\end{gather*}
and since $\phi_n^{\alpha*}({\alpha}_n)>0$, it follows that $s_n=1$.}

Because $(\phi_n^\alpha)^*=B_n^\alpha\phi_{n*}^\alpha=B_n^\alpha\phi_n^\beta$ and $B_n^\alpha=\dB_n^\alpha/\dB_n^\beta$, also the other claims follow.
\end{proof}

For the normalization of the $\phi_n$, we can do two things: either we make the
normalization of~$\phi_n$ simple and choose for example $\phi_n^*({\gamma}_n)>0$, similar to what we did for $\phi_n^\alpha$ and $\phi_n^\beta$ (but this is somewhat problematic as we shall see below), or we can insist on keeping the relation with~$\phi_n^\alpha$ and $\phi_n^\beta$ simple as in the previous lemma and arrange that $s_n^\alpha=s_n^\beta=1$. We choose for the second option.

Let us assume that $\gamma_n=\alpha_n$.
Denote
\begin{gather*}
\phi_n(z)=\frac{p_n(z)}{\dpi_n^\alpha(z)\dpi_n^\beta(z)}\qquad \text{and}\qquad
\phi_n^\alpha(z)=\frac{p_n^\alpha(z)}{\pi_n^\alpha(z)},
\end{gather*}
with $p_n$ and $p_n^\alpha$ both polynomials in $\P_n$. Then
\begin{gather*}
\phi_n^*(z)=\frac{\varsigma_n\, p_n^*(z)}{\dpi_n^\alpha(z)\dpi_n^\beta(z)}
\qquad\mbox{and}\qquad
\phi_n^{\alpha*}(z)=\frac{\varsigma_n\, p_n^{\alpha*}(z)}{\pi_n^\alpha(z)},\qquad \varsigma_n=\prod_{j=1}^n \sigma_j.
\end{gather*}
We already know that there is some $s_n^\alpha\in\TT$
such that $\phi_n=s_n^\alpha\dB_n^\beta\phi_n^\alpha$.
Take the $(\ )_*$ conjugate and multiply with $\dB_n^\alpha\dB_n^\beta$
to get $\phi_n^*=\overline{s_n^\alpha}\dB_n^\beta\phi_n^{\alpha*}$.

It now takes some simple algebra to reformulate
$\phi_n^*=\overline{s_n^\alpha}\dB_n^\beta\phi_n^{\alpha*}$ as
\begin{gather*}
\phi_n^*(z)=\frac{\varsigma_n\,p_n^*(z)}{\dpi_n^\alpha(z)\dpi_n^\beta(z)}=
\overline{s_n^\alpha}
\frac{\varsigma_n\,p_n^{\alpha*}(z)}{\dpi_n^\alpha(z)\dpi_n^\beta(z)}\mathop\Dprod_{j\in\mathbbm{b}_n}(-|\beta_j|).
\end{gather*}
This implies that $p_n^*(z)=\overline{s_n^\alpha}p_n^{\alpha*}(z)\dprod_{j\in\mathbbm{b}_n}(-|\beta_j|)$
and thus that $p_n^*(z)$ has the same zeros as~$p_n^{\alpha*}(z)$, none of which is in $\DD$.
Thus the numerator of $\phi_n^*$ will not vanish at ${\alpha}_n\in\DD$ but
one of the factors $(1-\ob_j{\alpha}_n)$ from $\dpi_n^\beta(\alpha_n)$ could be zero.
Thus a normalization $\phi_n^*(\alpha_n)>0$ is not an option in general.
We could however make $s_n^\alpha=1$ when we choose $p_n^*(\alpha_n)/p_n^{\alpha*}(\alpha_n)>0$ or, since
$\phi_n^{\alpha*}(\alpha_n)>0$, this is equivalent with $\varsigma_n\,p_n^*(\alpha_n)/\pi_n^\alpha(\alpha_n)>0$.
Yet another way to put this is requiring that $\phi_n^*(z)/\dB_n^\beta(z)$ is positive at $z=\alpha_n$.
This does not give a problem with~0 or~$\infty$ since
\begin{gather}\label{eqstar}
\dB_n^\alpha(z)\phi_{n*}(z)=
\frac{\phi_n^*(z)}{\dB_n^\beta(z)}=\frac{\dvarsigma_n^\alpha\,p_n^*(z)}{\dpi_n^\alpha(z)\dprod_{j\in\mathbbm{b}_n}{(z-\beta_j)}},\qquad
\dvarsigma_n^\alpha=\prod_{j\in\mathbbm{a}_n}\sigma_j.
\end{gather}
It is clear that neither the numerator nor the denominator will vanish for $z=\alpha_n$.

Of course a similar argument can be given if $\gamma_n=\beta_n$.
Then we choose $\phi_n^*(z)/\dB^\alpha_n(z)$ to be positive at $z=\beta_n$ or equivalently
$\varsigma_n\,p_n^*(\beta_n)/\pi_n^\beta(\beta_n)\dprod_{j\in\mathbbm{a}_n}(-|\alpha_j|)>0$.

Let us formulate the result about the numerators as a lemma for further reference.
\begin{Lemma}\label{lemnum}
With the normalization that we just imposed the numerators $p_n^\nu$ of $\phi_n^\nu=p_n^\nu/\pi_n^\nu$,
$\nu\in\{\alpha,\beta,\gamma\}$ and $n\ge 1$ are related by
\begin{gather*}
p_n(z)=p_n^\alpha(z)\mathop\Dprod_{j\in\mathbbm{b}_n}(-|\beta_j|)
=p_n^{\beta*}(z)\varsigma_n\mathop\Dprod_{j\in\mathbbm{a}_n}(-|\alpha_j|),\qquad \text{if}\quad\gamma_n=\alpha_n
\end{gather*}
and
\begin{gather*}
p_n(z)=p_n^\beta(z)\mathop\Dprod_{j\in\mathbbm{a}_n}(-|\alpha_j|)
=p_n^{\alpha*}(z)\varsigma_n\mathop\Dprod_{j\in\mathbbm{b}_n}(-|\beta_j|),\qquad \text{if}\quad \gamma_n=\beta_n,
\end{gather*}
where as before $\varsigma_n=\prod_{j=1}^n \sigma_j$.
\end{Lemma}
\begin{proof}
The f\/irst expression for $\gamma_n=\alpha_n$ has been proved above. The second one follows in a similar way
from the relation $\phi_n(z)=\phi_n^{\beta*}(z)\dB_n^\alpha(z)$.
Indeed
\begin{gather*}
\frac{p_n(z)}{\pi_n(z)}=\frac{\varsigma_n p_n^{\beta*}(z)}
{\prod\limits_{j\in\mathbbm{a}_n}\varpi_j^\beta(z)\prod\limits_{j\in\mathbbm{b}_n}\varpi_j^\beta(z)}
\prod\limits_{j\in\mathbbm{a}_n}\sigma_j\frac{z-\alpha_j}{1-\oa_j z}\\
\hphantom{\frac{p_n(z)}{\pi_n(z)}}{}= \frac{\varsigma_n p_n^{\beta*}(z)}
{\prod\limits_{j\in\mathbbm{a}_n}\varpi_j^\alpha(z)\prod\limits_{j\in\mathbbm{b}_n}\varpi_j^\beta(z)}
\prod\limits_{j\in\mathbbm{a}_n}\sigma_j\frac{z-\alpha_j}{1-\ob_j z}=
\frac{\varsigma_n p_n^{\beta*}(z)}{\pi_n(z)}\mathop\Dprod_{j\in\mathbbm{a}_n}\sigma_j(-\alpha_j)\frac{z-\alpha_j}{z-\alpha_j}.
\end{gather*}
With $-\sigma_j\alpha_j=-|\alpha_j|$ the result follows.

The case $\gamma_n=\beta_n$ is similar.
\end{proof}

Note that this normalization again means that we take the leading coef\/f\/icient of $\phi_n$ to be positive
in the following sense.
If $\gamma_n=\alpha_n$ then $\phi_n(z)=\overline{(\dB_n^\alpha\phi_{n*})(\alpha_n)}\dB_n^\alpha(z)+\psi_{n-1}(z)$ with $\psi_{n-1}\in\L_{n-1}$, while $\dB_n^\alpha\phi_{n*}=\phi_{n}^{\alpha*}$ and $\phi_{n}^{\alpha*}(\alpha_n)>0$.
If $\gamma_n=\beta_n$ then $\phi_n(z)=\overline{(\dB_n^\beta\phi_{n*})(\beta_n)}\dB_n^\beta(z)+\psi_{n-1}(z)$ with $\psi_{n-1}\in\L_{n-1}$ and the conclusion follows similarly.

Whenever we use the term orthonormal, we assume this normalization and $\{\phi_n\colon n=0,1,$ $2,\ldots\}$ will denote this orthonormal system.

Thus we have proved the following Theorem.
It says that if $\gamma_n=\alpha_n$, then $\phi_n$ is a `shifted' version of $\phi_n^\alpha$ where `shifted' means multiplied by $\dB_n^\beta$:
\begin{gather*}
\dB_n^\beta(z)\phi_n(z)=\dB_n^\beta(z)[a_0B_0^\alpha+\cdots+a_nB_n^\alpha(z)]=
a_0\dB_n^\beta(z)+\cdots+a_n\dB_n^\alpha(z),
\end{gather*}
and a similar interpretation if $\gamma_n=\beta_n$.
We summarize this in the following theorem.
\begin{Theorem}\label{thm1}
Assume all ORFs $\phi_n^\nu$, $\nu\in\{\alpha,\beta,\gamma\}$ are orthonormal with positive leading coefficient, i.e.,
\begin{gather*}
\phi_n^{\alpha*}({\alpha}_n)>0\qquad\mbox{and}\qquad
\phi_n^{\beta*}({\beta}_n)>0
\qquad \text{and}\qquad
\begin{cases}
(\phi_n^*/\dB_n^\beta)(\alpha_n)>0 & \mbox{if}~~\gamma_n=\alpha_n,\\
(\phi_n^*/\dB_n^\alpha)(\beta_n)>0 & \mbox{if}~~\gamma_n=\beta_n.
\end{cases}
\end{gather*}
Then for all $n\ge0$
\begin{gather*}
\phi_n=
(\phi_n^\alpha)\dB_n^\beta=
(\phi_n^\beta)^*\dB_n^\alpha
\qquad \mbox{and}\qquad \phi_n^*=
(\phi_n^\alpha)^*\dB_n^\beta=
(\phi_n^\beta)\dB_n^\alpha\qquad \mbox{if~~$\gamma_n=\alpha_n$},
\end{gather*}
while
\begin{gather*}
\phi_n=(\phi_n^\beta)\dB_n^\alpha=
(\phi_n^\alpha)^*\dB_n^\beta
\qquad \mbox{and}\qquad \phi_n^*=
(\phi_n^\beta)^*\dB_n^\alpha=
(\phi_n^\alpha)\dB_n^\beta\qquad \mbox{if~~$\gamma_n=\beta_n$}.
\end{gather*}
\end{Theorem}

\begin{Corollary}\label{cor0}
We have for all $n\ge1$ that $(\phi_n^{\nu})^*\perp\zeta_n^\nu\L_{n-1}^\nu$, $\nu\in\{\alpha,\beta,\gamma\}$.
\end{Corollary}

\begin{Corollary}\label{cor1}
The rational functions $\phi_n^\alpha$ and $\phi_n^{\alpha*}$ are in $\L_n^\alpha$ and hence have all their poles in $\{\beta_j\colon j=1,\ldots,n\}\subset\EE$ while
the zeros of $\phi_n^\alpha$ are all in $\DD$ and the zeros of $\phi_n^{\alpha*}$ are all in $\EE$.

The rational functions $\phi_n^\beta$ and $\phi_n^{\beta*}$ are in $\L_n^\beta$ and hence have all their poles in $\{\alpha_j\colon j=1,\ldots,n\}\subset\DD$ while
the zeros of $\phi_n^\beta$ are all in $\EE$ and the zeros of $\phi_n^{\beta*}$ are all in $\DD$.

The rational functions $\phi_n$ and $\phi_n^*$ are in $\L_n$ and hence have all their poles in $\{\beta_j\colon j\in\mathbbm{a}_n\} \cup \{\alpha_j\colon j\in\mathbbm{b}_n\}$.

The zeros of $\phi_n$ are the same as the zeros of
$\phi_n^\alpha$ and thus are all in $\DD$ if $\gamma_n=\alpha_n$ and they are the same as the zeros of $\phi_n^\beta$ and thus they are all in $\EE$
 if $\gamma_n=\beta_n$.
\end{Corollary}
\begin{proof}
It is well known that the zeros of $\phi_n^\alpha$ are all in $\DD$ \cite[Corollary~3.1.4]{BooksBGHN99}, and because $\phi_n^\beta=\phi_{n*}^\alpha$, this means that the zeros of $\phi_n^\beta$ are all in $\EE$.

Because $\phi_n=(\phi_n^\alpha)\dB_n^\beta=(\phi_n^\alpha)/\prod_{j\in\mathbbm{b}_n}\zeta_j^\alpha$
if $\gamma_n=\alpha_n$, i.e., $n\in\mathbbm{a}_n$, and the product with $\dB_n^\beta$ will only exchange the poles $1/\oa_j=\beta_j$, $j\in\mathbbm{b}_n$ in $\phi_n$ for poles $\alpha_j=1/\ob_j$, the zeros of $\phi_n^\alpha$
are left unaltered.

The proof for $n\in\mathbbm{b}_n$ is similar.
\end{proof}

One may summarize that for $f\in\L_n^\nu$ the $f_*$ transform ref\/lects both zeros and poles in $\TT$ since $z\mapsto z_*=1/\overline{z}$,
while the transform $f\to f^*$ as it is def\/ined in the spaces $\L_n^\nu$, $\nu\in\{\alpha,\beta,\gamma\}$,
keeps the poles but ref\/lects the zeros since
the multiplication with the respective factors $B_n^\alpha$, $B_n^\beta$ and $\dB_n^\alpha\dB_n^\beta$ will only undo the ref\/lection of the poles
that resulted from the $f_*$ operation.

\section{Christof\/fel--Darboux relations and reproducing kernels}\label{secCD}
For $\nu\in\{\alpha,\beta,\gamma\}$, one may def\/ine the reproducing kernels
for the space $\L_n^\nu$. Given the ORF $\phi_k^\nu$, the kernels are def\/ined by
\begin{gather*}
k_n^\nu(z,w)=\sum_{k=0}^n \phi_k^\nu(z)\overline{\phi_k^\nu(w)}.
\end{gather*}
They reproduce $f\in\L_n^\nu$ by $\SP{k_n^\nu(\cdot,z)}{f}=f(z)$.

The proof of the Christof\/fel--Darboux relations goes exactly like in the classical case and we shall not repeat it here (see, e.g., \cite[Theorem~3.1.3]{BooksBGHN99}).
\begin{Theorem}
The Christoffel--Darboux relations
\begin{gather*}
k_n^\nu(z,w)=
\frac{\phi_n^{\nu*}(z)\overline{\phi_n^{\nu*}(w)}-
\zeta_n^\nu(z)\overline{\zeta_n^\nu(w)}
 \phi_n^{\nu}(z)\overline{\phi_n^{\nu}(w)}}
{1-\zeta_n^\nu(z)\overline{\zeta_n^\nu(w)}}\\
\hphantom{k_n^\nu(z,w)}{}
=
\frac{\phi_{n+1}^{\nu*}(z)\overline{\phi_{n+1}^{\nu*}(w)}-
 \phi_{n+1}^{\nu}(z)\overline{\phi_{n+1}^{\nu}(w)}}
{1-\zeta_{n+1}^\nu(z)\overline{\zeta_{n+1}^\nu(w)}}
\end{gather*}
hold for $n\ge0$, $\nu\in\{\alpha,\beta,\gamma\}$ and $z$, $w$ not among the poles of $\phi_n^\nu$ and not on $\TT$.
\end{Theorem}
As an immediate consequence we have:
\begin{Theorem}
The following relations hold true:
\begin{gather*}
k_n^\alpha(z,w)\dB_n^\beta(z)\overline{\dB_n^\beta(w)}=k_n(z,w)=
k_n^\beta(z,w)\dB_{n}^\alpha(z)\overline{\dB_{n}^\alpha(w)}
\end{gather*}
for $n\ge0$ and $z,w\not\in(\TT\cup\{\beta_j\colon j\in\mathbbm{a}_n\}\cup\{\alpha_j\colon j\in\mathbbm{b}_n\})$.
\end{Theorem}
\begin{proof}
The f\/irst relation was directly shown above for the case $\gamma_n=\alpha_n$.
It also follows in the case $\gamma_{n+1}=\alpha_{n+1}$ and using in the second CD-relation the f\/irst expressions from Theorem~\ref{thm1} for $\phi_{n+1}$ and $\phi_{n+1}^*$. The relation is thus valid independent of $\gamma_n=\alpha_n$ or $\gamma_n=\beta_n$.

Similarly the second expression was derived before in the case $\gamma_n=\beta_n$, but again, it also follows from the second CD-relation and the f\/irst expressions from Theorem~\ref{thm1} for~$\phi_{n+1}$ and~$\phi_{n+1}^*$ in the case $\gamma_{n+1}=\beta_{n+1}$.
Again the relation holds independently of $\gamma_n=\alpha_n$ or $\gamma_n=\beta_n$.

Alternatively, the second relation can also be derived from the second CD-relation in the case $\gamma_{n+1}=\alpha_{n+1}$ but using the second expressions from
Theorem~\ref{thm1} for $\phi_{n+1}$ and $\phi_{n+1}^*$.
\end{proof}

Evaluation of the CD-relation in $\nu_n$ for $\nu\in\{\alpha,\beta\}$ results in
another useful corollary.
\begin{Corollary}\label{corCCD}
For $\nu\in\{\alpha,\beta\}$ we have for $n\ge0$
\begin{gather*}
k_n^\nu(z,\nu_n)=\phi_n^{\nu*}(z)\overline{\phi_n^{\nu*}(\nu_n)}
\qquad\text{and}\qquad
k_n^\nu(\nu_n,\nu_n)=|\phi_n^{\nu*}(\nu_n)|^2.
\end{gather*}
\end{Corollary}

The latter corollary cannot immediately be used when $\nu=\gamma$ because
$\gamma_n$ could be equal to some pole of $\phi_n$ if it equals some $1/\overline{\gamma}_j$ for $j<n$. In that case we can remove the denominators in the CD relation and only keep the numerators.
Hence setting
\begin{gather*}
k_n(z,w)=\frac{C_n(z,w)}{\pi_n(z)\overline{\pi_n(w)}},\qquad
\phi_n(z)=\frac{p_n(z)}{\pi_n(z)},\qquad
\phi_n^*(z)=\frac{\varsigma_n p_n^*(z)}{\pi_n(z)},\qquad\varsigma_n\in\TT,
\end{gather*}
the CD relation becomes
\begin{gather}
C_n(z,w)=\frac{p_n^*(z)\overline{p_n^*(w)}-\zeta_n(z)\overline{\zeta_n(w)}p_n(z)\overline{p_n(w)}}{1-\zeta_n(z)\overline{\zeta_n(w)}}\nonumber\\
\hphantom{C_n(z,w)}{}
=
\frac{p_{n+1}^*(z)\overline{p_{n+1}^*(w)}-p_{n+1}(z)\overline{p_{n+1}(w)}}{(1-\zeta_{n+1}(z)\overline{\zeta_{n+1}(w)})\varpi_{n+1}(z)\overline{\varpi_{n+1}(w)}}.
\label{eqCDN}\end{gather}
Thus, the f\/irst form gives
\begin{gather*}
C_n(z,\gamma_n)=p_n^*(z)\overline{p_n^*(\gamma_n)}\qquad
\text{and}\qquad
C_n(\gamma_n,\gamma_n)=|p_n^*(\gamma_n)|^2.
\end{gather*}
Evaluating a polynomial at inf\/inity means taking its highest degree coef\/f\/icient, i.e., if $q_n(z)$
is a~polynomial of degree $n$, then $q_n(\infty)$ stands for its coef\/f\/icient of $z^n$.

The second form of (\ref{eqCDN}) gives for $\gamma_{n+1}\ne\infty$ and $\gamma_n\ne\infty$
\begin{gather*}
C_n(z,\gamma_n)=\frac{p_{n+1}^*(z)\overline{p_{n+1}^*(\gamma_n)}-p_{n+1}(z)\overline{p_{n+1}(\gamma_n)}}{(1-\overline{\gamma}_nz)(1-|\gamma_{n+1}|^2)}
\qquad \text{and}\\
C_n(\gamma_n,\gamma_n)=\frac{|p_{n+1}^*(\gamma_n)|^2-|p_{n+1}(\gamma_n)|^2}{(1-|\gamma_{n}|^2)(1-|\gamma_{n+1}|^2)}.
\end{gather*}
Coupling the f\/irst and the second form in (\ref{eqCDN}) gives
\begin{gather*}
\frac{|p_{n+1}^*(\gamma_n)|^2-|p_{n+1}(\gamma_n)|^2}{(1-|\gamma_{n}|^2)(1-|\gamma_{n+1}|^2)}=|p_n^*(\gamma_n)|^2.
\end{gather*}
For $\gamma_{n+1}=\infty$ and $\gamma_n\ne\infty$ we get
\begin{gather*}
C_n(z,\gamma_n)=\frac{p_{n+1}^*(z)\overline{p_{n+1}^*(\gamma_n)}-p_{n+1}(z)\overline{p_{n+1}(\gamma_n)}}{-(1-\overline{\gamma}_nz)}=p_n^*(z)\overline{p_n^*(\gamma_n)}
\end{gather*}
and
\begin{gather*}
C_n(\gamma_n,\gamma_n)=\frac{|p_{n+1}^*(\gamma_n)|^2-|p_{n+1}(\gamma_n)|^2}{-(1-|\gamma_{n}|^2)}=|p_n^*(\gamma_n)|^2.
\end{gather*}
If $\gamma_{n+1}=\infty$ and $\gamma_n=\infty$, the denominators in (\ref{eqCDN}) have to be replaced by 1, which gives
\begin{gather*}
C_n(z,\gamma_n)={p_{n+1}^*(z)\overline{p_{n+1}^*(\gamma_n)}-p_{n+1}(z)\overline{p_{n+1}(\gamma_n)}}=p_n^*(z)\overline{p_n^*(\gamma_n)}
\end{gather*}
and
\begin{gather*}
C_n(\gamma_n,\gamma_n)={|p_{n+1}^*(\gamma_n)|^2-|p_{n+1}(\gamma_n)|^2}=|p_n^*(\gamma_n)|^2.
\end{gather*}
For $\gamma_n=\infty$ and $\gamma_{n+1}\ne\infty$ we obtain in a similar way
\begin{gather*}
C_n(z,\gamma_n)=\frac{p_{n+1}^*(z)\overline{p_{n+1}^*(\gamma_n)}-p_{n+1}(z)\overline{p_{n+1}(\gamma_n)}}{z(1-|\gamma_{n+1}|^2)}=p_n^*(z)\overline{p_n^*(\gamma_n)}
\end{gather*}
and
\begin{gather*}
C_n(\gamma_n,\gamma_n)=\frac{|p_{n+1}^*(\gamma_n)|^2-|p_{n+1}(\gamma_n)|^2}{(1-|\gamma_{n+1}|^2)}=|p_n^*(\gamma_n)|^2.
\end{gather*}
To summarize, the relations of Corollary~\ref{corCCD} may not hold for the ORF if $\nu=\gamma$,
but similar relations do hold for the numerators as stated in the next corollary.
\begin{Corollary}\label{rorCCD2}
If $C_n(z,w)$ is the numerator in the CD relation and $p_n(z)$ is the numerator
of the ORF $\phi_n$ for the sequence $\seq{\gamma}$ then we have for $n\ge0$
\begin{gather*}
C_n(z,\gamma_n)=p_n^*(z)\overline{p_n^*(\gamma_n)}
\qquad\mbox{and}\qquad
C_n(\gamma_n,\gamma_n)=|p_n^*(\gamma_n)|^2.
\end{gather*}
\end{Corollary}

\section{Recurrence relation}\label{secRR}

The recurrence for the $\phi_n^\alpha$ is well known.
For a proof see, e.g., \cite[Theorem~4.1.1]{BooksBGHN99}.
For $\phi_n^\beta$ the proof can be copied by symmetry.
However, also for $\nu=\gamma$ the same recurrence and its proof can be copied, with the
exception that the derivation fails when $p_n^{*}(\gamma_{n-1})=0$ where $p_n=\phi_n\pi_n$.
This can (only) happen
if $(1-|\gamma_n|)(1-|\gamma_{n-1}|)<0$ (i.e., one of these $\gamma$'s is in $\DD$ and the other is in $\EE$).
We shall say that $\phi_n$ is {\it regular} if $p_n^*(\gamma_{n-1})\ne0$.
If $\nu=\alpha$ or $\nu=\beta$ then the whole sequence $(\phi_n^\nu)_{n\ge0}$ will be automatically
regular.
Thus we have the following theorem:
\begin{Theorem}\label{thmrec}
Let $\nu\in\{\alpha,\beta,\gamma\}$
and if $\nu=\gamma$ assume moreover that $\phi_n^\nu$ is regular,
then the following recursion holds with initial condition $\phi_0^\nu=\phi_0^{\nu*}=1$
\begin{gather*}
\left[\begin{matrix}\phi_n^\nu(z)\\\phi_n^{\nu*}(z)\end{matrix}\right]=
H_n^\nu\frac{\varpi_{n-1}^\nu(z)}{\varpi_n^\nu(z)}
\left[\begin{matrix}1&{\lambda}_n^\nu\\\overline{\lambda}_n^\nu & 1\end{matrix}\right]
\left[\begin{matrix}\zeta_{n-1}^\nu(z)&0\\0 & 1\end{matrix}\right]
\left[\begin{matrix}\phi_{n-1}^\nu(z)\\\phi_{n-1}^{\nu*}(z)\end{matrix}\right],
\end{gather*}
where $H_n^\nu$ is a nonzero constant times a unitary matrix:
\begin{gather*}
H_n^\nu=e_n^\nu\left[\begin{matrix}\eta_{n1}^\nu&0\\0&\eta_{n2}^\nu\end{matrix}\right],\qquad
e_n^\nu\in\CC\setminus\{0\},\qquad \eta_{n1}^\nu, \eta_{n2}^\nu\in\TT.
\end{gather*}
The constant $\eta_{n1}^\nu$ is chosen such that the normalization condition for the ORFs is maintained.
The other constant $\eta_{n2}^\nu$ is then automatically related to
$\eta_{n1}^\nu$ by $\eta_{n2}^\nu=\overline{\eta}_{n1}^\nu\overline{\sigma}_{n-1}\sigma_n$.
The Szeg\H{o} parameter $\lambda_n^\nu$ is given by
\begin{gather*}
\lambda_n^\nu=\eta_n^\nu\frac{p_n^\nu(\nu_{n-1})}{\overline{p_n^{\nu*}(\nu_{n-1})}}
\qquad\mbox{with}\qquad
\eta_n^\nu=\overline{\varsigma}_{n-2}\frac{\overline{p_{n-1}^{\nu*}(\nu_{n-1})}}{p_{n-1}^{\nu*}(\nu_{n-1})}\in\TT,
\end{gather*}
where $\phi_n^\nu(z)=p_n^\nu(z)/\pi_n^\nu(z)$.
\end{Theorem}
\begin{proof}
We immediately concentrate on the general situation $\nu=\gamma$.
Of course $\nu=\alpha$ and $\nu=\beta$ will drop out as special cases.
For simplicity assume that $\gamma_n$ and $\gamma_{n-1}$ are not 0 and not~$\infty$.
The technicalities when this is not true are left as an exercise.
It is easy because formally it follows the steps of the proof below
but one has to
replace a linear factor involving inf\/inity by the coef\/f\/icient of inf\/inity
(like $1-\infty z=-z$ and $z-\infty=-1$)
and evaluating a polynomial at~$\infty$ means taking its leading coef\/f\/icient.

First we show that there are some numbers $c_n$ and $d_n$ such that
\begin{gather*}
\phi(z):=\frac{1-\overline{\gamma}_nz}{z-\gamma_{n-1}}\phi_n(z)-d_n\phi_{n-1}(z)-c_n\frac{1-\overline{\gamma}_{n-1}z}{z-\gamma_{n-1}}\phi_{n-1}^*\in\L_{n-2}.
\end{gather*}
This can be written as $N(z)/[(z-\gamma_{n-1})\pi_{n-1}(z)]$. Thus the $c_n$ and $d_n$ are def\/ined by the
conditions $N(\gamma_{n-1})=N(1/\overline{\gamma}_{n-1})=0$.
If we denote $\phi_k=p_k/\pi_k$ and thus $\phi_k^*=p_k^*\varsigma_k/\pi_k$, it is clear that
\begin{gather*}
N(z)=p_n(z)-d_n(z-\gamma_{n-1})p_{n-1}(z)-c_n(1-\overline{\gamma}_{n-1}z)p_{n-1}^*(z)\varsigma_{n-1}.
\end{gather*}
Thus the f\/irst condition gives
\begin{gather*}
c_n=\frac{\overline{\varsigma}_{n-1}p_n(\gamma_{n-1})}{(1-|\gamma_{n-1}|^2)p_{n-1}^*(\gamma_{n-1})}
\end{gather*}
and the second one
\begin{gather*}
d_n=\frac{p_n(1/\overline{\gamma}_{n-1})}{(1/\overline{\gamma}_{n-1}-\gamma_{n-1})p_{n-1}(1/\overline{\gamma}_{n-1})}
=\frac{\overline{p_n^*(\gamma_{n-1})}}{(1-|\gamma_{n-1}|^2)\overline{p_{n-1}^*(\gamma_{n-1})}}.
\end{gather*}
Note that $p_{n-1}^*(\gamma_{n-1})$ cannot be zero by Corollary~\ref{cor1}, and that also
$p_{n}^{*}(\gamma_{n-1})$ does not vanish by our assumption of regularity.

Furthermore, by using the orthognality of $\phi_n$ and $\phi_{n-1}$ and Corollary~\ref{cor0},
it is not dif\/f\/icult to show that $\phi\perp\L_{n-2}$ so that it must be identically zero.
Thus
\begin{gather*}
\phi_n(z)=d_n\overline{\sigma}_{n-1}\frac{1-\overline{\gamma}_{n-1}z}{1-\overline{\gamma}_nz}
[\zeta_{n-1}(z)\phi_{n-1}(z)+\lambda_n\phi_{n-1}^*(z)],
\end{gather*}
with
\begin{gather*}
\lambda_n=\eta_n\frac{p_n(\gamma_{n-1})}{\overline{p_n^*(\gamma_{n-1})}},\qquad
\eta_n=\overline{\varsigma}_{n-1}\sigma_{n-1}\frac{\overline{p_{n-1}^*(\gamma_{n-1})}}{p_{n-1}^*(\gamma_{n-1})}.
\end{gather*}
By taking the $(\ )^*$ transform (in $\L_n$) we obtain
\begin{gather*}
\phi_n^*(z)=\overline{d}_n{\sigma}_{n}\frac{1-\overline{\gamma}_{n-1}z}{1-\overline{\gamma}_nz}
[\overline{\lambda}_n\zeta_{n-1}(z)\phi_{n-1}(z)+\phi_{n-1}^*(z)].
\end{gather*}
This proves the recurrence by taking $e_n=|d_n|$ and $\overline{\eta}_{n1}={\sigma}_{n-1}\mathfrak{u}({d}_n)$.

It remains to show that the initial step for $n=1$ is true.
Since $\phi_0=\phi_0^*=1$, then in case $\gamma_0=\alpha_0=0$, hence $\zeta_0=z$, we have
\begin{gather*}
\phi_1(z)=e_1\eta_{11}\frac{z+\lambda_1}{1-\overline{\gamma}_1z}
\qquad \text{and}\qquad
\phi_1^*(z)=e_1{\eta}_{12}\frac{\overline{\lambda}_1z+1}{1-\overline{\gamma}_1z}.
\end{gather*}
Thus
\begin{gather*}
p_1(z)=e_1\eta_{11}(z+\lambda_1)
\qquad \text{and}\qquad
p_1^*(z)=e_1\overline{\eta}_{11}(\overline{\lambda}_1z+1).
\end{gather*}
This implies that $\lambda_1$ is indeed given by the general formula because
\begin{gather*}
\lambda_1=\eta_1\frac{p_1(\gamma_0)}{\overline{p_1^*(\gamma_0)}}=
\frac{p_1(0)}{\overline{p_1^*(0)}}=\frac{e_1\eta_{11}\lambda_1}{e_1\eta_{11}}.
\end{gather*}
In case $\gamma_0=\beta_0=\infty$, then $\zeta_0=1/z$, so that
\begin{gather*}
\phi_1(z)=e_1\eta_{11}\frac{-1-\lambda_1z}{1-\overline{\gamma}_1z}
\qquad\text{and}\qquad
\phi_1^*(z)=e_1{\eta}_{12}\frac{-\overline{\lambda}_1-z}{1-\overline{\gamma}_1z},
\end{gather*}
and thus
\begin{gather*}
p_1(z)=-e_1\eta_{11}(1+\lambda_1z)
\qquad\text{and}\qquad
p_1^*(z)=-e_1\overline{\eta}_{11}\overline{\sigma}_1(\overline{\lambda}_1+z),
\end{gather*}
and again $\lambda_1$ is given by the general formula
\begin{gather*}
\lambda_1=\eta_1\frac{p_1(\gamma_0)}{\overline{p_1^*(\gamma_0)}}=
1\frac{p_1(\infty)}{\overline{p_1^*(\infty)}}=\frac{e_1\eta_{11}\lambda_1}{e_1\eta_{11}}.
\end{gather*}
This proves the theorem.
\end{proof}
\begin{Remark}
If $\nu\in\{\alpha,\beta\}$ we could rewrite $\lambda_n^\nu$ in terms of $\phi_n^\nu$ because
by dividing and multiplying with the appropriate denominators $\pi_n^\nu$ one gets
\begin{gather*}
\lambda_n^\nu=\eta_n^\nu \frac{\phi_n^\nu(\nu_{n-1})}{\overline{\phi_n^{\nu*}(\nu_{n-1})}},\qquad
\eta_n^\nu=\sigma_{n-1}\overline{\sigma}_n\frac{(1-\overline{\nu}_n\nu_{n-1})}{(1-\nu_n\overline{\nu}_{n-1})}
\frac{\overline{\phi_{n-1}^{\nu*}(\nu_{n-1})}}{\phi_{n-1}^{\nu*}(\nu_{n-1})},\qquad n\ge1.
\end{gather*}
Note that also this $\eta_n^\nu\in\TT$, but it dif\/fers from the $\eta_n^\nu$
in the previous theorem.
However if $\nu=\gamma$, then this expression has the disadvantage
that $\gamma_{n-1}$ could be equal to $1/\overline{\gamma}_n$ or
it could be equal to a pole of $\phi_n$ in which case it would not make sense
to evaluate these expressions in $\gamma_{n-1}$.
The latter expressions only make sense if we interpret them as limiting values
\begin{gather*}
\frac{{\phi_n^\nu(\nu_{n-1})}}{\overline{\phi_n^{\nu*}(\nu_{n-1})}}=\lim_{z\to\nu_{n-1}}
\frac{{\phi_n^\nu(z)}}{\overline{\phi_n^{\nu*}(z)}}
\qquad \text{and}\\
\frac{(1-\overline{\nu}_n\nu_{n-1})}{(1-\nu_n\overline{\nu}_{n-1})}
\frac{\overline{\phi_{n-1}^{\nu*}(\nu_{n-1})}}{\phi_{n-1}^{\nu*}(\nu_{n-1})}=
\lim_{z\to\nu_{n-1}}
\frac{(1-\overline{\nu}_nz)}{(1-\nu_n\overline{z})}
\frac{\overline{\phi_{n-1}^{\nu*}(z)}}{\phi_{n-1}^{\nu*}(z)},
\end{gather*}
where one has to assume that $\lim\limits_{\xi\to0}[\xi/\overline{\xi}]=1$.
We shall from now on occasionally use these expressions with this
interpretation, but the expressions for $\lambda_n^\nu$ from
Theorem~\ref{thmrec} using the numerators are more direct since they immediately give the limiting values.
\end{Remark}

Note that $\lambda_n^\alpha$ is the value of a Blaschke product with all its zeros in $\DD$ evaluated at $\alpha_{n-1}\in\DD$ and therefore $\lambda_n^\alpha\in\DD$.
Similarly, $\lambda_n^\beta$ is the value of a Blaschke product with all its zeros in $\EE$, evaluated at
$\beta_{n-1}\in\EE$ so that $\lambda_n^\beta\in\DD$.
Since the zeros of $\phi_n$ are the zeros of $\phi_n^\alpha$ if $n\in\mathbbm{a}_n$ and they are the
zeros of $\phi_n^\beta$ if $n\in\mathbbm{b}_n$, it follows that if $n$ and $n-1$ are both in $\mathbbm{a}_n$ or both in $\mathbbm{b}_n$, then $\lambda_n\in\DD$ but if $n\in\mathbbm{a}_n$ and
$n-1\in\mathbbm{b}_n$ or vice versa, then $\lambda_n\in\EE$.
Therefore
\begin{gather}\label{eqe2}
(e_n^\nu)^2=\frac{1-|\nu_n|^2}{1-|\nu_{n-1}|^2}\frac{1}{1-|\lambda_n^\nu|^2}>0
\end{gather}
and we can choose $e_n$ as the positive square root of this expression.
The above expression is derived in \cite[Theorem~4.1.2]{BooksBGHN99}
for the case $\nu=\alpha$ by using the CD relations. By symmetry, this also holds
for $\nu=\beta$. For $\nu=\gamma$, the same relation can be obtained
by stripping the denominators as we explained after the proof of the CD-relation in Section~\ref{secCD}.

What goes wrong with the recurrence relation when $\phi_n$ is not regular?
From the proof of Theorem~\ref{thmrec}, it follows that then $d_n=0$.
We still have the relation
\begin{gather*}
\phi_n(z)=\frac{\overline{\sigma}_{n-1}}{(1-|\gamma_{n-1}|^2)\overline{p_{n-1}^*(\gamma_{n-1})}}
\left[
p_n^*(\gamma_{n-1})\zeta_{n-1}(z)\phi_{n-1}(z)+s_{n-1}p_n(\gamma_{n-1})\phi_{n-1}^*(z)
\right]
\end{gather*}
with $s_{n-1}=\frac{\varsigma_{n-1}p_{n-1}^*(\gamma_{n-1})}{\overline{\sigma}_{n-1}\overline{p_{n-1}^*(\gamma_{n-1})}}\in\TT$ and $p_n^*(\gamma_{n-1})=0$.
Thus, there is some positive constant $e_n$ and some $\eta_{n1}\in\TT$ such that
\begin{gather*}
\phi_n(z)=e_n\eta_{n1}\frac{\varpi_{n-1}(z)}{\varpi_n(z)}\big[0\;\zeta_{n-1}(z)\phi_{n-1}(z)+\phi_{n-1}^*(z)\big],
\end{gather*}
i.e., the f\/irst term in the sum between square brackets vanishes.
Applying Theorem~\ref{thmrec} in this case would give $\lambda_n=\infty$, and the
previous relations show that we only have to replace in Theorem~\ref{thmrec}
the matrix
\begin{gather*}
\left[\begin{matrix}1&\lambda_n^\nu\\\overline{\lambda_n^\nu}&1\end{matrix}\right]
\qquad\text{by}\qquad
\left[\begin{matrix}0&1\\1&0\end{matrix}\right].
\end{gather*}
This is in line with how we have dealt with $\infty$ so far where the
rule of thumb was to set $a-\nu b=-b$ if $\nu=\infty$.
So let us therefore also use Theorem~\ref{thmrec} with this interpretation
when~$\phi_n$ is not regular and thus $\lambda_n=\infty$.
With the expressions at the end of Section~\ref{secCD}, it can also be shown that
in this case
\begin{gather*}
e_n^2=-\frac{1-|\gamma_n|^2}{1-|\gamma_{n-1}|^2}>0.
\end{gather*}
Note that this corresponds to replacing $1-|\lambda_n|^2$ when $\lambda_n=\infty$ by $-1$.
Since this non-regular situation can only occur when $(1-|\gamma_n|)(1-|\gamma_{n-1}|)<0$, this expression for $e_n^2$ is indeed positive. A similar rule can be applied if $\gamma_n$ or $\gamma_{n-1}$ is inf\/inite, just replace in this or previous expression $1-|\infty|^2$ by $-1$.
The positivity of the expressions for $e_n^2$ also follows from the following result.

\begin{Theorem}\label{thmcon}
The Szeg\H{o} parameters satisfy for all $n\ge1$:

If $\gamma_n=\alpha_n$ and $\gamma_{n-1}=\alpha_{n-1}$ then $\lambda_n = \lambda_n^\alpha=\overline{\lambda_n^\beta}\in\DD$.

If $\gamma_n=\beta_n$ and $\gamma_{n-1}=\beta_{n-1}$ then $\lambda_n = \lambda_n^\beta=\overline{\lambda_n^\alpha}\in\DD$.

If $\gamma_n=\alpha_n$ and $\gamma_{n-1}=\beta_{n-1}$ then $\lambda_n = 1/\overline{\lambda_n^\beta}=1/\lambda_n^\alpha\in\EE$.

If $\gamma_n=\beta_n$ and $\gamma_{n-1}=\alpha_{n-1}$ then $\lambda_n=
1/\overline{\lambda_n^\alpha}=1/\lambda_n^\beta\in\EE$.
\end{Theorem}
\begin{proof}
Suppose $\gamma_n=\alpha_n$ and $\gamma_{n-1}=\alpha_{n-1}$, then
by Theorems~\ref{thmrec} and \ref{thm1}, or better still by Lemma~\ref{lemnum},
\begin{gather*}
\lambda_n=\left(\overline{\varsigma}_{n-2}
\frac{\overline{p_{n-1}^*(\alpha_{n-1})}}{p_{n-1}^*(\alpha_{n-1})}\right)\;
\frac{p_{n}(\alpha_{n-1})}{\overline{p_n^*(\alpha_{n-1})}}=
\left(\overline{\varsigma}_{n-2}
\frac{\overline{p_{n-1}^{\alpha*}(\alpha_{n-1})}}{p_{n-1}^{\alpha*}(\alpha_{n-1})}\right)\;
\frac{p_n^\alpha(\alpha_{n-1})}{\overline{p_n^{\alpha*}(\alpha_{n-1})}}=\lambda_n^\alpha.
\end{gather*}
When using $p_n^\alpha(z)=\varsigma_n\,p_n^{\beta*}(z)\mathop\dprod_{j=1}^n(-|\alpha_j|)$ and
$\alpha_j=1/\ob_j$, the previous relation
becomes
\begin{gather*}
\lambda_n=\left(\overline{\varsigma}_{n-2}
\frac{\varsigma_{n-1}\overline{p_{n-1}^\beta(1/\ob_{n-1})}}{\overline{\varsigma}_{n-1}p_{n-1}^\beta(1/\ob_{n-1})}\right)
\frac{p_n^{\beta*}(1/\ob_{n-1})}{\overline{p_n^\beta(1/\ob_{n-1})}}
\\
\hphantom{\lambda_n}{}= \left(\sigma_{n-1}^2\varsigma_{n-2}\frac{p_n^{\beta*}(\beta_{n-1})}{\overline{p_n^{\beta*}(\beta_{n-1})}} \frac{\ob_{n-1}^{n-1}}{\beta_{n-1}^{n-1}}\right)
\frac{\overline{p_n^\beta(\beta_{n-1})}}{p_n^{\beta*}(\beta_{n-1})}
\frac{\beta_{n-1}^{n}}{\ob_{n-1}^{n}}\\
\hphantom{\lambda_n}{}= \sigma_{n-1}^2\frac{\beta_{n-1}}{\ob_{n-1}}\overline{\lambda_n^\beta}=\overline{\lambda_n^\beta}.
\end{gather*}
The proof for $\gamma_n=\beta_n$ and $\gamma_{n-1}=\beta_{n-1}$, $n\ge 1$ is similar.

Next consider $\gamma_n=\alpha_n$ and $\gamma_{n-1}=\beta_{n-1}$, then
\begin{gather*}
\lambda_n=\left(\overline{\varsigma}_{n-2}
\frac{\overline{p_{n-1}^{\beta*}(\beta_{n-1})}}{p_{n-1}^{\beta*}(\beta_{n-1})}\right)
\frac{p_n^{\alpha}(\beta_{n-1})}{\overline{p_n^{\alpha*}(\beta_{n-1})}}=
\eta_n^\beta\frac{p_n^{\beta*}(\beta_{n-1})}{\overline{p_n^{\beta}(\beta_{n-1})}}=
\frac{\eta_n^\beta\overline{\eta_n^\beta}}{\overline{\lambda_n^\beta}}=
\frac{1}{\overline{\lambda_n^\beta}}=\frac{1}{\lambda_n^\alpha}.
\end{gather*}
The remaining case $\gamma_n=\beta_n$ and $\gamma_{n-1}=\alpha_{n-1}$, is again similar.
\end{proof}

\begin{Remark} \label{remlambda}
It should also be clear that the expression for the parameters $\lambda_n^\nu$ of Theorem~\ref{thmrec} are for
theoretical use only. They are expressed in terms of $p_n^\nu$, which is the numerator of $\phi_n^\nu$, the object
that should be the result of the computation, and hence unknown at the moment of computing $\lambda_n^\nu$.
For practical use, these parameters $\lambda_n^\nu$ should be obtained in a dif\/ferent way.
In inverse eigenvalue problems the recursion is used backwards, i.e., for decreasing degrees of the ORFs
and then these expressions can be used of course.

Even if we know the $\lambda_n^\nu$, then in Theorem~\ref{thmrec}, we still need the normalizing factor
$\eta_{n1}^\nu\in\TT$ which is characterized by ``$\eta_{n1}^\nu$ is chosen such that the
normalization condition or the ORFs is maintained''. We could compute $\phi_n^\nu$ with $\eta_{n1}^\nu=1$,
check the value of $\phi_n^{\nu*}(\gamma_{n-1})$, and derive $\eta_{n1}^\nu$ from it.
What this means is shown in the next lemma.
\end{Remark}

\begin{Lemma}\label{lemeta}
For $\nu\in\{\alpha,\beta\}$,
the phase $\theta_n^\nu$ of the unitary factor $\eta_{n1}^\nu=e^{i\theta_n^\nu}$ is given by
\begin{gather*}
\theta_n^\nu=\mathrm{arg}\bigl(\overline{\sigma}_{n-1}\sigma_n\overline{\varpi^\nu_{n}(\nu_{n-1})}\overline{\phi_n^{\nu*}(\nu_{n-1})}\bigr)
\end{gather*}
or equivalently
\begin{gather*}
\eta_{n1}^\nu=\overline{\sigma}_{n-1}\sigma_n\mathfrak{u}\bigl(\varpi_n^\nu(\nu_{n-1})\phi_n^{\nu*}(\nu_{n-1})\bigr).
\end{gather*}
$($Recall $\mathfrak{u}(z)=\overline{z}/|z|$.$)$
\end{Lemma}
\begin{proof}
Take the f\/irst relation of Theorem~\ref{thmrec} and evaluate for $z=\nu_{n-1}$, then because\linebreak
$\varpi_{n-1}^{\nu*}(\nu_{n-1})=0$ we get
\begin{gather*}
\phi_n^\nu(\nu_{n-1})\varpi_n^{\nu}(\nu_{n-1})=e_n^\nu\eta_{n1}^\nu\varpi_{n-1}^\nu(\nu_{n-1})\lambda_n^\nu\phi_{n-1}^{\nu*}(\nu_{n-1})
\end{gather*}
or
\begin{gather*}
\eta_{n1}^\nu=
\frac{\varpi_n^{\nu}(\nu_{n-1})\phi_n^{\nu}(\nu_{n-1})}{e_n^\nu\varpi_{n-1}^\nu(\nu_{n-1})\lambda_n^\nu\phi_{n-1}^{\nu*}(\nu_{n-1})}.
\end{gather*}
Use the def\/inition of $\lambda_n^\nu=\eta_n^\nu\frac{\phi_n^\nu(\nu_{n-1})}{\overline{\phi_n^{\nu*}(\nu_{n-1})}}$
with $\eta_n^\nu=\sigma_{n-1}\overline{\sigma}_n\frac{\varpi_n^\nu(\nu_{n-1})}{\varpi_{n-1}^\nu(\nu_n)}$ and knowing that $\phi_{n-1}^{\nu*}(\nu_{n-1})>0$ we obtain
after simplif\/ication and leaving out all the factors with phase zero
\begin{gather*}
\theta_n^\nu=\mathrm{arg}\big(
\overline{\sigma}_{n-1}\sigma_n\overline{\varpi^\nu_{n}(\nu_{n-1})}\overline{\phi_n^{\nu*}(\nu_{n-1})}
\big)
\end{gather*}
as claimed.
\end{proof}

Note that this expression for $\eta_{n1}^\nu$ is well def\/ined because $\phi_n^{\nu*}(\nu_{n-1})\ne0$ if $\nu\in\{\alpha,\beta\}$.

For $\nu=\gamma$, the expression is a bit more involved but it can be obtained in a similar way from the
normalization conditions given in Theorem~\ref{thm1}. We skip the details.

Another solution of the recurrence relation is formed by the functions of the second kind.
Like in the classical case (i.e., for $\nu=\alpha$)
we can introduce them for $\nu\in\{\alpha,\beta,\gamma\}$ by (see \cite[p.~83]{BooksBGHN99})
\begin{gather*}
\psi_n^\nu(z)=\int_{\TT} [E(t,z)\phi_n^\nu(t)-D(t,z)\phi_n^\nu(z)]\d\mu(t).
\end{gather*}
where $D(t,z)=\frac{t+z}{t-z}$ and $E(t,z)=D(t,z)+1=\frac{2t}{t-z}$.
This results in
\begin{gather*}
\psi_0^\nu=1\qquad\mbox{and}\qquad\psi_n^\nu(z)=\int_{\TT} D(t,z)[\phi_n^\nu(t)-\phi_n^\nu(z)]\d\mu(t),\qquad n\ge 1,
\end{gather*}
which may be generalized to (see \cite[Lemma~4.2.2]{BooksBGHN99})
\begin{gather*}
\psi_n^\nu(z)f(z)=\int_{\TT} D(t,z)[\phi_n^\nu(t)f(t)-\phi_n^\nu(z)f(z)]\d\mu(t),\qquad n\ge1
\end{gather*}
with $f$ arbitrary in $\L_{(n-1)*}^\nu$.
It also holds that (see \cite[Lemma~4.2.3]{BooksBGHN99})
\begin{gather*}
\psi_n^{\nu*}(z)g(z)=\int_{\TT} D(t,z)[\phi_n^{\nu*}(t)g(t)-\phi_n^{\nu*}(z)g(z)]\d\mu(t),\qquad n\ge1
\end{gather*}
with $g$ arbitrary in $\L_{n*}^\nu(\nu_{n*})$.
Recall that $\L_{n*}^\nu(\nu_{n*})$ is the space of all functions in $\L_{n*}^\nu$ that vanish for $z=\nu_{n*}=1/\overline{\nu}_n$.
This space is spanned by $\{B_k^\nu/B_n^\nu\colon k=0,\ldots,n-1\}$ if $\nu\in\{\alpha,\beta\}$. For $\nu=\gamma$, the space can be characterized as (see Lemma~\ref{lem1})
\begin{gather*}
\L_{n*}(\gamma_{n*})=
\Span\left\{\frac{B_k}{\dB_n^\alpha\dB_n^\beta}\right\}_{k=0}^{n-1}=
\Span\left\{\frac{B_k^\alpha}{\zeta_n\dB_{n-1}^\alpha}\right\}_{k=0}^{n-1}=
\Span\left\{\frac{B_k^\beta}{\zeta_n\dB_{n-1}^\beta}\right\}_{k=0}^{n-1}.
\end{gather*}
\begin{Theorem}\label{thmrel2knd}
The following relations for the functions of the second kind hold for $n\ge0$:
\begin{gather*}
\psi_n=
(\psi_n^\alpha)\dB_n^\beta=
(\psi_n^\beta)^*\dB_n^\alpha
\qquad\mbox{and}\qquad \psi_n^*=
(\psi_n^\alpha)^*\dB_n^\beta=
(\psi_n^\beta)\dB_n^\alpha\qquad \mbox{if~~$\gamma_n=\alpha_n$},
\end{gather*}
while
\begin{gather*}
\psi_n=(\psi_n^\beta)\dB_n^\alpha=
(\psi_n^\alpha)^*\dB_n^\beta
\qquad \mbox{and}\qquad \psi_n^*=
(\psi_n^\beta)^*\dB_n^\alpha=
(\psi_n^\alpha)\dB_n^\beta\qquad \mbox{if~~$\gamma_n=\beta_n$}.
\end{gather*}
We assume the normalization of Theorem~{\rm \ref{thm1}}.
\end{Theorem}
\begin{proof}
This is trivial for $n=0$, hence
suppose $n\ge1$ and $\gamma_n=\alpha_n$ then
\begin{gather*}
\psi_n(z)=\int_{\TT} D(t,z)[\phi_n(t)-\phi_n(z)]\d\mu(t)=\int_{\TT} D(t,z)[\phi_n^\alpha(t)\dB_n^\beta(t)-\phi_n^\alpha(z)\dB_n^\beta(z)]\d\mu(t)\\
\hphantom{\psi_n(z)}{}= \psi_n^\alpha(z)\dB_n^\beta(z),
\end{gather*}
because
\begin{gather*}
\dB_n^\beta(z)=\dB_{n-1}^\beta(z)=\prod_{j\in\mathbbm{b}_{n-1}}\zeta_j^\beta(z)=
\prod_{j\in\mathbbm{b}_{n-1}}\zeta_{j*}^\alpha(z)\in\L_{(n-1)*}^\alpha.
\end{gather*}
Moreover, using $\phi_{n*}^\alpha=\phi_n^\beta$ we also have
\begin{gather*}
\psi_n(z)=\int_{\TT} D(t,z)[\phi_n(t)-\phi_n(z)]\d\mu(t)=\int_{\TT} D(t,z)[\phi_n^\alpha(t)\dB_n^\beta(t)-\phi_n^\alpha(z)\dB_n^\beta(z)]\d\mu(t)\\
\hphantom{\psi_n(z)}{}=
\int_{\TT} D(t,z)[\phi_{n*}^\beta(t)\dB_n^\beta(t)-\phi_{n*}^\beta(z)\dB_n^\beta(z)]\d\mu(t)\\
\hphantom{\psi_n(z)}{}=\int_{\TT} D(t,z)[\phi_{n}^{\beta*}(t)\dB_n^\alpha(t)-\phi_{n}^{\beta*}(z)\dB_n^\alpha(z)]\d\mu(t),
\end{gather*}
and since $\dB_n^\alpha\in\L_{n*}^\beta(\beta_{n*})$, we also get the second part:
$\psi_n=(\psi_n^\beta)^*\dB_n^\alpha$.

Moreover
$\psi_n^*=[\psi_n^\alpha\dB_n^\beta]_*\dB_n^\alpha\dB_n^\beta=
\psi_{n*}^{\alpha}\dB_n^\alpha\dB_n^\beta\dB_{n*}^\beta=
[\psi_{n*}^{\alpha}\dB_n^\alpha/\dB_n^\beta]\dB_n^\beta
=\psi_n^{\alpha*}\dB_n^\beta$.
It follows in a~similar way that $\psi_n^*=\psi_n^\beta\dB_n^\alpha$.

The case $\gamma_n=\beta_n$ is proved similarly.
\end{proof}

With these relations, it is not dif\/f\/icult to mimic the arguments of \cite[Theorem~4.2.4]{BooksBGHN99}
and obtain the following.
\begin{Theorem}\label{thm2knd}
These functions satisfy the recurrence relation
\begin{gather*}
\left[\begin{matrix}\psi_n^\nu(z)\\-\psi_n^{\nu*}(z)\end{matrix}\right]=
H_n^\nu\frac{\varpi_{n-1}^\nu(z)}{\varpi_n^\nu(z)}
\left[\begin{matrix}1&{\lambda}_n^\nu\\\overline{\lambda}_n^\nu & 1\end{matrix}\right]
\left[\begin{matrix}\zeta_{n-1}^\nu(z)&0\\0 & 1\end{matrix}\right]
\left[\begin{matrix}\psi_{n-1}^\nu(z)\\-\psi_{n-1}^{\nu*}(z)\end{matrix}\right],\qquad n\ge1
\end{gather*}
with $\psi_0^\nu=\psi_0^{\nu*}=1$ and all other quantities as in Theorem~{\rm \ref{thmrec}}.
\end{Theorem}

\section{Para-orthogonal rational functions}\label{secPORF}

Before we move on to quadrature formulas, we def\/ine para-orthogonal functions (PORF) by
\begin{gather}\label{defPORF}
Q_n^\nu(z,\tau_n^\nu)=\phi_n^\nu(z)+\tau_n^\nu\phi_n^{\nu*}(z),\qquad
\tau_n^\nu\in\TT,\qquad
\nu\in\{\alpha,\beta,\gamma\}.
\end{gather}
The PORF $Q_n^\nu(z,\tau_n^\nu)$ is in $\L_n^\nu$ and it is called para-orthogonal because it is not orthogonal
to~$\L_{n-1}^\nu$ but it is orthogonal to a particular subspace of $\L_{n-1}^\nu$ of dimension $n-1$.
\begin{Theorem}\label{thm2}
The para-orthogonal rational function $Q_n(z)=Q_n(z,\tau_n)$, $\tau_n\in\TT$ with $n\ge2$ is orthogonal to
$\dzeta_n^\alpha\L_{n-1}\cap\dzeta_n^\beta\L_{n-1}=\zeta_n\L_{n-1}\cap\L_{n-1}=\L_{n-1}(\gamma_n)$ with
\begin{gather*}
\L_{n-1}(\gamma_n)=\{f\in\L_{n-1}\colon f(\gamma_n)=0\}=\left\{\frac{\varpi_n^*(z)p_{n-2}(z)}{\pi_{n-1}(z)}\colon p_{n-2}\in\P_{n-2}\right\}.
\end{gather*}
Recall that $\varpi_n^*(z)=z-\gamma_n$ if $\gamma_n\ne\infty$, and $\varpi_n^*(z)=-1$ if $\gamma_n=\infty$.
\end{Theorem}
\begin{proof}
Suppose $\gamma_n=\alpha_n$ then $\dzeta_n^\alpha=\zeta_n^\alpha$ and $\dzeta_n^\beta=1$.
Hence $\phi_n\perp\L_{n-1}$ and $\phi_n^*\perp \zeta_n^\alpha\L_{n-1}$
and therefore $Q_n\perp \L_{n-1}\cap\zeta_n^\alpha\L_{n-1}=\L_{n-1}(\alpha_n)$.
The proof for $\gamma_n=\beta_n$ is similar.
\end{proof}

One may also def\/ine associated functions of the second kind as
\begin{gather*}
P_n^\nu(z,\tau_n)=\psi_n^\nu(z)-\tau_n^\nu\psi_n^{\nu*}(z),\qquad \nu\in\{\alpha,\beta,\gamma\}.
\end{gather*}
From Theorems~\ref{thm1} and \ref{thmrel2knd} we can easily obtain the following corollary.
\begin{Corollary}\label{cor2knd}
With the notation $Q_n$ and $P_n$ for the PORF and the associated functions just introduced, we have for $n\ge1$
\begin{gather*}
Q_n(z,\tau_n)=\dB_n^\beta(z)Q_n^\alpha(z,\tau_n)=
\tau_n \dB_n^\alpha(z) Q_n^\beta(z,\overline{\tau}_n)
\qquad \mbox{if~~$\gamma_n=\alpha_n$},
\end{gather*}
while
\begin{gather*}
Q_n(z,\tau_n)=\dB_n^\alpha(z)Q_n^\beta(z,\tau_n)=
\tau_n \dB_n^\beta(z) Q_n^\alpha(z,\overline{\tau}_n)
\qquad \mbox{if~~$\gamma_n=\beta_n$}.
\end{gather*}
Similarly
\begin{gather*}
P_n(z,\tau_n)=\dB_n^\beta(z)P_n^\alpha(z,\tau_n)=
-\tau_n \dB_n^\alpha(z) P_n^\beta(z,\overline{\tau}_n)
\qquad \mbox{if~~$\gamma_n=\alpha_n$},
\end{gather*}
while
\begin{gather*}
P_n(z,\tau_n)=\dB_n^\alpha(z)P_n^\beta(z,\tau_n)=
-\tau_n \dB_n^\beta(z) P_n^\alpha(z,\overline{\tau}_n)
\qquad \mbox{if~~$\gamma_n=\beta_n$}.
\end{gather*}
\end{Corollary}
\begin{proof}
Assume that $\gamma_n=\alpha_n$, then $\phi_n=\phi_n^\alpha\dB_n^\beta$ and $\phi_n^*=(\phi_n^\alpha)^*\dB_n^\beta$.
Thus
\begin{gather*}
Q_n(\cdot,\tau_n)=\dB_n^\beta\left[\phi_n^\alpha+\tau_n(\phi_n^\alpha)^*\right]=
\dB_n^\beta Q_n^\alpha(\cdot,\tau_n).
\end{gather*}
In a similar way one has
\begin{gather*}
Q_n(\cdot,\tau_n)=\phi_n^{\beta*}\dB_n^\alpha+\tau_n \phi_n^\beta\dB_n^\alpha
=\tau_n \dB_n^\alpha\left[\phi_n^\beta+\overline{\tau}_n\phi_n^{\beta*}\right]
=\tau_n \dB_n^\alpha Q_n^\beta(\cdot,\overline{\tau}_n).
\end{gather*}
The proofs for $P_n$ and for $\gamma_n=\beta_n$ are similar.
\end{proof}

We are now ready to state that the zeros of the para-orthogonal rational functions
$Q_n^\nu(z,\tau_n^\nu)$
will be simple and on $\TT$ no matter whether $\nu=\alpha$, $\beta$ or $\gamma$.

\begin{Corollary}\label{corzeros}
Take $n\ge1$, $\tau_n\in\TT$ and $\nu\in\{\alpha,\beta,\gamma\}$ then the zeros of $Q_n^\nu(z,\tau_n)$
are all simple and on $\TT$.
\end{Corollary}
\begin{proof}
The case $\nu=\alpha$ was proved in \cite[Theorem~5.2.1]{BooksBGHN99} and by symmetry, this also
holds for $\nu=\beta$. For $\nu=\gamma$,
assume that $\gamma_n=\alpha_n$, then by the previous corollary $Q_n(z,\tau_n)=\dB_n^\beta(z)Q_n^\alpha(z,\tau_n)$
so that the zeros of
$Q_n(z,\tau_n)$ are the zeros of $Q_n^\alpha(z,\tau_n)$
(and hence will not cancel against any of the poles).

The proof for $\gamma_n=\beta_n$ is similar.
\end{proof}
These PORF properties are important because the zeros will deliver the nodes
for the rational Szeg\H{o} quadrature as we will explain in the next section.

For further reference we give the following property.
\begin{Proposition}\label{propPORF}
For $n\ge1$ and $\nu\in\{\alpha,\beta,\gamma\}$, the PORF satisfy, using the notation of Theorem~{\rm \ref{thmrec}}
\begin{gather*}
Q_n^\nu(z,\tau_n)=c_n^\nu\frac{\varpi_{n-1}^\nu(z)}{\varpi_n^\nu(z)}
\big[\zeta_{n-1}^\nu(z)\phi_{n-1}^\nu(z)+\tilde{\tau}_n^\nu\phi_{n-1}^{\nu*}(z)\big]
\end{gather*}
with
\begin{gather*}
c_n^\nu=e_n^\nu\big(\eta_{n1}+\eta_{n2}\tau_n\overline{\lambda}_n^\nu\big),\qquad
\tilde{\tau}_n^\nu=\frac{\hat{\tau}_n+\lambda_n^\nu}{1+\hat{\tau}_n\overline{\lambda}_n^\nu}\in\TT,\qquad
\hat{\tau}_n=\overline{\eta}_{n1}\eta_{n2}\tau_n.
\end{gather*}
\end{Proposition}
\begin{proof}
Just take the recurrence relation of Theorem~\ref{thmrec} and premultiply it with $[1~~\tau_n]$.
After re-arrangement of the terms, the expression above will result.
\end{proof}

The importance of this property is that, up to a constant nonzero factor $c_n^\nu$, we can compute~$Q_n^\nu(z,\tau_n)$ by exactly the same recurrence that gives $\phi_k^\nu$ in terms of $\phi_{k-1}^\nu$ and
$\phi_{k-1}^{\nu*}$ for $k=1,2,\ldots,n$, except that we have to replace the trailing parameter $\lambda_n^\nu$
by a unimodular constant~$\tilde{\tau}_n^\nu$ to get~$Q_n^\nu(z,\tau_n)$.
When we are interested in the zeros of~$Q_n^\nu(z,\tau_n)$, then this normalizing factor $c_n^\nu$
does not matter since it is nonzero.

\section{Quadrature}\label{secQF}
We start by proving that the subspace of rational functions in which the
quadrature formulas will be exact only depends on the points
$\{\alpha_k\colon k=0,\ldots,n-1\}$ no matter whether the points $\alpha_k$
are introduced as a pole $\alpha_k$ or $\alpha_{k*}=1/\oa_k$ in the sequence
$\seq{\gamma}_n=(\gamma_k)_{k=0}^n$.

\begin{Lemma}
$\R_n\ed\L_n\cdot\L_{n*}=\R_n^\alpha\ed\L_n^\alpha\cdot\L_{n*}^\alpha=\R_n^\beta\ed\L_n^\beta\cdot\L_{n*}^\beta$.
\end{Lemma}
\begin{proof}
The space $\L_n$ contains rational functions of degree at most $n$ whose denominators
have zeros in $\{\beta_j\colon j\in\mathbbm{a}_n\}\cup \{\alpha_j\colon j\in\mathbbm{b}_n\}$.
The space $\L_{n*}$ contains rational functions of degree at most $n$ whose denominators
have zeros in $\{\beta_j\colon j\in\mathbbm{b}_n\}\cup \{\alpha_j\colon j\in\mathbbm{a}_n\}$.
Thus the space $\R_n$ contains rational functions of degree at most $2n$ whose denominators
have zeros in $\{\beta_j\colon j=1,\ldots,n\}\cup \{\alpha_j\colon j=1,\ldots,n\}$.
Since the denominators of functions in $\L_n^\alpha$ have zeros in
$\{\beta_j\colon j=1,\ldots,n\}$ and
functions in $\L_{n*}^\alpha$ have zeros in $\{\alpha_j\colon j=1,\ldots,n\}$,
the denominator of functions in $\R_n^\alpha$ have zeros in $\{\beta_j\colon j=1,\ldots,n\}\cup \{\alpha_j\colon j=1,\ldots,n\}$ so that $\R_n=\R_n^\alpha$.
Of course a similar argument shows that also $\R_n=\R_n^\beta$.
\end{proof}

As is well known from the classical case associated with the sequence
$\seq{\alpha}$ discussed for example in \cite[Chapter~5]{BooksBGHN99}, the rational Szeg\H{o} quadrature formulas are of the form
\begin{gather*}
I_n^\alpha(f)=\sum_{j=0}^{n-1} w_{nj}^\alpha f(\xi_{nj}^\alpha),\qquad n\ge1,
\end{gather*}
where $\xi_{nj}^\alpha=\xi_{nj}^\alpha(\tau_n^\alpha)$, $j=0,\ldots,n-1$
are the zeros of the para-orthogonal rational function $Q_n^\alpha(z,\tau_n^\alpha)$ with $\tau_n^\alpha\in\TT$ and the weights are given by
\begin{gather*}
w_{nj}^\alpha=w_{nj}^\alpha(\tau_n)=\frac{1}{\sum\limits_{k=0}^{n-1} |\phi_k^\alpha(\xi_{nj}^\alpha(\tau_n^\alpha))|^2}>0.
\end{gather*}
These formulas have maximal degree of exactness, meaning that
\begin{gather*}
\int_{\TT} f(t)\d\mu(t)=I_n^\alpha(f),\qquad \forall\, f\in\R_{n-1}^\alpha
\end{gather*}
and that no $n$-point interpolatory quadrature formula with nodes on $\TT$
and positive weights can be exact in a larger $(2n-1)$-dimensional space of the form $\L_n^\alpha\cdot\L_{(n-1)*}^\alpha$ or $\L_{n*}^\alpha\cdot\L_{n-1}^\alpha$.
They are however exact in a subspace of dimension $2n-1$ of a dif\/ferent form as shown in~\cite{NPNSL07} and~\cite{NPCBDMPPJ17} for Laurent polynomials and in~\cite{ArtBHGN07} for the rational case.

Our previous results show that taking an arbitrary order of selecting the $\seq{\gamma}$ does not add new quadrature formulas.
Indeed, if $\gamma_n=\alpha_n$ then the zeros of $Q_n(z,\tau_n)$, $Q_n^\alpha(z,\tau_n)$ and $Q_n^\beta(z,\overline{\tau}_n)$ are all the same,
i.e., $\xi_{nj}(\tau_n)=\xi_{nj}^\alpha(\tau_n)=\xi_{nj}^\beta(\overline{\tau}_n)$, $j=1,\ldots,n$. Dropping the dependency on $\tau$ from the notation, it
is directly seen that
\begin{gather*}
\sum_{k=0}^{n-1}|\phi_k(\xi_{nj})|^2=
\sum_{k=0}^{n-1}|\phi_k^\alpha(\xi_{nj})\dB_k^\beta(\xi_{nj})|^2=
\sum_{k=0}^{n-1}|\phi_k^\alpha(\xi_{nj})|^2=
\sum_{k=0}^{n-1}|\phi_k^\alpha(\xi_{nj}^\alpha)|^2
\end{gather*}
and thus we also have $w_{nj}=w_{nj}^\alpha$ and similarly
$w_{nj}=w_{nj}^\beta$.
Therefore $I_n=I_n^\alpha=I_n^\beta$ for appropriate choices of the def\/ining
parameters, i.e.,
\begin{gather*}
\tau_n^\alpha=
\tau_n
\qquad\mbox{and}\qquad
\tau_n^\beta=\overline{\tau}_n
\qquad\mbox{if~~$\gamma_n=\alpha_n$},\\
\tau_n^\beta=
\tau_n
\qquad\mbox{and}\qquad
\tau_n^\alpha=\overline{\tau}_n
\qquad\mbox{if~~$\gamma_n=\beta_n$}.
\end{gather*}

An alternative expression for the weights is also known (\cite[Theorem~5.4.1]{BooksBGHN99}) and it will obviously
also coincide for $\nu\in\{\alpha,\beta,\gamma\}$:
\begin{gather*}
w_{nj}^\nu=\frac{1}{2\xi_{nj}^\nu}\left.\frac{P_n^\nu(z)}{\frac{\d}{\d z}Q_n^\nu(z)}\right|_{z=\xi_{nj}^\nu},
\end{gather*}
where $Q_n^\nu$ is the PORF with zeros $\{\xi_{nj}^\nu\}_{j=0}^{n-1}$
and $P_n^\nu$ the associated functions
of the second kind as in Corollary~\ref{cor2knd}.
We have dropped again the obvious dependence on the parameter $\tau_n^\nu$ from the notation.

A conclusion to be drawn from this section is that whether we choose the sequence
$\seq{\gamma}$, $\seq{\alpha}$ or~$\seq{\beta}$, the resulting quadrature formula we obtain
is an $n$-point formula that is exact in the space $\R_{n-1}=\L_{n-1}\cdot\L_{(n-1)*}$.
Such a quadrature formula is called an $n$-point rational Szeg\H{o} quadrature and these are the
only positive interpolating quadrature formulas exact in this space. They are unique up to the choice of the parameter $\tau_n\in\TT$.
Thus, to obtain the nodes and weights for a~general sequence $\seq{\gamma}$ and a choice for
$\tau_n\in\TT$, we can as well compute them for the sequence~$\seq{\alpha}$ and an appropriate
$\tau_n^\alpha\in\TT$ or for the sequence that alternates between one $\alpha$ and one $\beta$. The resulting quadrature will be the same.

In practice nodes and weights of a quadrature formula are computed by solving an eigenvalue
problem. It will turn out that also in these computations, nothing is gained from adopting a~general order
for the $\gamma$ sequence but that a sequence alternating between $\alpha$ and $\beta$ (the balanced situation) is
the most economical one.
We shall come back to this in Section~\ref{secSA} after we provide in the next sections the structure of the matrices whose eigenvalues will give the nodes of the quadrature formula.

Rational interpolatory and Szeg\H{o} quadrature formulas with arbitrary order of the poles including error estimates and numerical experiments were considered in \cite{MISCCYGV07}.

\section{Block diagonal with Hessenberg blocks}\label{secHB}
The orthogonal polynomials are a special case
of ORF obtained by choosing a sequence $\seq{\gamma}=\seq{\alpha}$
that contains only zeros.
Another special case is given by the orthogonal Laurent polynomials (OLP)
obtained by choosing an alternating sequence $\seq{\gamma}=(0,\infty,0,\infty,0,\infty,\ldots)$.
In \cite{MISCVel07}, Vel\'azquez described spectral methods for
ORF on the unit circle that were based on these two special choices.
The result was that the shift operator
$\T_\mu\colon L_\mu^2\to L_\mu^2\colon f(z)\mapsto zf(z)$ has a~matrix representation
that is a matrix M\"obius transform of a structured matrix.
The structure is a Hessenberg structure in the case of
$\seq{\gamma}=\seq{\alpha}$
and it is a f\/ive-diagonal matrix (a so-called CMV matrix) for the sequence
$\seq{\gamma}=(0,\beta_1,\alpha_2,\beta_3,\alpha_4,\ldots)$.
In the case of orthogonal (Laurent) polynomials the
M\"obius transform turns out to be just the identity and we get the
plain Hessenberg matrix for the polynomials and the plain CMV matrix for the Laurent polynomials.

In \cite{MISCBD07}, the OLP case was discussed using an alternative linear algebra approach
when the~$0$,~$\infty$ choice did not alternate nicely, but the order in
which they were added was arbitrary.
Then the structured matrix generalized to a so-called snake-shape matrix referring
to the zig-zag shape of the graphical representation that f\/ixed the order in which the
elementary factors of the matrix had to be multiplied.
This structure is a generalization of both the Hessenberg and the CMV structures
mentioned above. It is a block diagonal where the blocks alternate
between upper and lower Hessenberg structure.

To illustrate this for our ORF, we start in this section by using
the approach of Vel\'azquez in~\cite{MISCVel07}, skipping all the details
just to obtain the block structure of the matrix that will represent the shift operator.
In the next sections we shall use the linear algebra approach of \cite{MISCBD07}
to analyze the computational aspects of obtaining the quadrature formulas.

We start from the recurrence for the $\phi_k^\alpha$ and transform it into a
recurrence relation for the~$\phi_k$, which will eventually result in a
representation of the shift operator.

To avoid a tedious book-keeping of normalizing constants, we will just exploit the fact
that there are some constants such that certain dependencies hold.
For example the recurrence
\begin{gather*}
\phi_n^\alpha(z)=e_n\eta_{n1}\frac{\varpi_{n-1}^\alpha(z)}{\varpi_n^\alpha(z)}\big[\zeta_{n-1}^\alpha(z)\phi_{n-1}^\alpha(z)+\lambda_n^\alpha \phi_{n-1}^{\alpha*}(z)\big]
\end{gather*}
or equivalently
\begin{gather*}
\phi_n^\alpha(z)=
 e_n\eta_{n1}\left[
 \sigma_{n-1}
 \frac{\varpi_{n-1}^{\alpha*}(z)}{\varpi_n^\alpha(z)}
 \phi_{n-1}^\alpha(z)+
 \lambda_n^\alpha
 \frac{\varpi_{n-1}^{\alpha}(z)}{\varpi_n^\alpha(z)}
 \phi_{n-1}^{\alpha*}(z)\right]
\end{gather*}
or
\begin{gather}\label{eqs0}
\varpi_{n-1}^{\alpha*}(z)\phi_{n-1}^\alpha(z)=
e_n^{-1}\overline{\eta}_{n1}\overline{\sigma}_{n-1}
 \varpi_n^\alpha(z)\phi_n^\alpha(z)-
\lambda_n^\alpha\overline{\sigma}_{n-1}
 \varpi_{n-1}^\alpha(z)\phi_{n-1}^{\alpha*}(z)
\end{gather}
will be expressed as
\begin{gather}\label{eqs1}
\varpi_{n-1}^{\alpha*}\phi_{n-1}^\alpha\in
\Span\big\{\varpi_{n}^\alpha\phi_n^\alpha,
 \varpi_{n-1}^\alpha\phi_{n-1}^{\alpha*}\big\}.
\end{gather}
Similarly, by combining the recurrence for $\phi_n^\alpha$ and $\phi_n^{\alpha*}$ from Theorem~\ref{thmrec} and eliminating $\phi_{n-1}^\alpha$,
we have{\samepage
\begin{gather}\label{eqs2}
\varpi_{n}^{\alpha}\phi_{n}^{\alpha*}\in
\Span\big\{\varpi_{n}^\alpha\phi_n^\alpha,
 \varpi_{n-1}^\alpha\phi_{n-1}^{\alpha*}\big\}.
\end{gather}
Note that this relation always holds whether or not $\lambda_n^\alpha\ne0$.}

Now suppose that our sequence $\seq{\gamma}$ has the following conf\/iguration{\samepage
\begin{gather*}
\seq{\gamma}=(
\ldots,
\beta_{k-1},\alpha_{k},\ldots,\alpha_{n},\ldots,\alpha_{m-1},\beta_{m},\ldots,\beta_{l-1},\alpha_{l},\dots)
\end{gather*}
with $2\le k\le n < m<l$ where $k$, $m$, $l$ are f\/ixed and $n$ varying from $k$ to $m-1$.}

Using (\ref{eqs1}) and then repeatedly (\ref{eqs2}) we get
\begin{gather*}
\varpi_{n}^{\alpha*}\phi_{n}^\alpha\in
\Span\big\{
\varpi_{n+1}^\alpha\phi_{n+1}^\alpha
,\ldots,
\varpi_{k}^\alpha\phi_{k}^\alpha,
\varpi_{k-1}^\alpha\phi_{k-1}^{\alpha*}
\big\}.
\end{gather*}
Multiply this with $\dB_{n}^\beta=\dB_{n-1}^\beta=\cdots=\dB_{k-1}^\beta$
\begin{gather*}
\varpi_{n}^{\alpha*}\dB_{n}^\beta\phi_{n}^\alpha\in
\Span\big\{
\varpi_{n+1}^\alpha\dB_{n}^\beta\phi_{n+1}^\alpha
,\ldots,
\varpi_{k}^\alpha\dB_{k}^\beta\phi_{k}^\alpha,
\varpi_{k-1}^\alpha\dB_{k-1}^\beta\phi_{k-1}^{\alpha*}\big\},
\end{gather*}
and since by Theorem~\ref{thm1},
$\phi_{p}=\dB_{p}^\beta\phi_{p}^\alpha$
for $p=n,n-1,\ldots,k$
and $\phi_{k-1}=\dB_{k-1}^\beta\phi_{k-1}^{\alpha*}$, this becomes
\begin{gather}\label{eqsp1}
\varpi_{n}^{\alpha*}\phi_{n}\in
\Span\big\{
\varpi_{n+1}^\alpha\dB_{n}^\beta\phi_{n+1}^\alpha,
\varpi_{n}^\alpha\phi_{n}
,\ldots,
\varpi_{k}^\alpha\phi_{k},
\varpi_{k-1}^\alpha\phi_{k-1}\big\}.
\end{gather}
Thus if $n+1<m$, then $\gamma_{n+1}=\alpha_{n+1}$, hence $\dB_{n+1}^\beta=\dB_{n}^\beta$ and $\phi_{n+1}=\dB_{n+1}^\beta\phi_{n+1}^\alpha$, so that
\begin{gather}\label{eqnp1ltm}
\varpi_{n}^{\alpha*}\phi_{n}\in
\Span\{
\varpi_{p}^\alpha\phi_{p}\}_{p=k-1}^{n+1},\qquad k\le n<m-1.
\end{gather}
If $n+1=m$, then $\gamma_{n+1}=\beta_m$ and we start dealing with the subsequence
of $\beta$'s in
\begin{gather*}
\seq{\gamma}=(
\ldots,
\alpha_{m-1},\beta_m,
\ldots,
\beta_{l-1},
\alpha_{l},\ldots).
\end{gather*}
Therefore we note that
\begin{gather}\label{eqto*}
\varpi_m^\alpha(z)=
(z-\alpha_m)\frac{1-\oa_mz}{z-\alpha_m}=(z-\alpha_m)\frac{z-\beta_m}{1-\ob_m z}\sigma_m^2=\sigma_m\varpi_m^{\alpha*}(z)\zeta_m^\beta(z).
\end{gather}
Thus
\begin{gather*}
\varpi_m^\alpha\dB_{m-1}^\beta\phi_m^\alpha=\sigma_m\varpi_m^{\alpha*}\dB_m^\beta\phi_m^\alpha.
\end{gather*}
Using (\ref{eqs1}) again repeatedly we then see that
\begin{gather*}
\varpi_m^\alpha\dB_{m-1}^\beta\phi_m^\alpha\in
\dB_m^\beta\Span\big\{\varpi_m^\alpha\phi_m^{\alpha*},\varpi_{m+1}^\alpha\phi_{m+1}^\alpha\big\}=
\dB_m^\beta\Span\big\{\varpi_m^\alpha\phi_m^{\alpha*},\varpi_{m+1}^{\alpha*}\zeta_{m+1}^\beta\phi_{m+1}^\alpha\big\}\\
\hphantom{\varpi_m^\alpha\dB_{m-1}^\beta\phi_m^\alpha}{}=
\Span\big\{\varpi_m^\alpha\dB_m^\beta\phi_m^{\alpha*},
 \varpi_{m+1}^{\alpha*} \dB_{m+1}^\beta\phi_{m+1}^{\alpha}\big\}\\
\hphantom{\varpi_m^\alpha\dB_{m-1}^\beta\phi_m^\alpha}{}= \cdots \\
\hphantom{\varpi_m^\alpha\dB_{m-1}^\beta\phi_m^\alpha}{}=
\Span\big\{\varpi_m^\alpha\dB_m^\beta\phi_m^{\alpha*},\ldots,
\varpi_{l-1}^\alpha\dB_{l-1}^\beta\phi_{l-1}^{\alpha*},\varpi_{l}^\alpha\dB_{l}^\beta\phi_{l}^\alpha\big\}\\
\hphantom{\varpi_m^\alpha\dB_{m-1}^\beta\phi_m^\alpha}{}=
\Span\big\{\varpi_m^\alpha\phi_m,\ldots,\varpi_{l-1}^\alpha\phi_{l-1},
\varpi_{l}^\alpha\phi_{l}\big\},
\end{gather*}
so that by plugging this into (\ref{eqsp1}) with $n+1=m$ gives
\begin{gather}\label{eqsp2}
\varpi_{m-1}^{\alpha*}\phi_{m-1}\in
\Span\big\{\varpi_{p}^\alpha\phi_{p}\big\}_{p=k-1}^{l}.
\end{gather}
The two relations (\ref{eqsp1}) and (\ref{eqsp2}) tell us how we should
express $\varpi_{n}^{\alpha*}\phi_{n}$ in terms of the $\{\varpi_k^\alpha\phi_k\colon k\le n+1\}$ in the case that $\gamma_{n}=\alpha_{n}$.

The next step is to express $\varpi_n^{\alpha*}\phi_n$ in terms of $\{\varpi_p^\alpha\phi_p\}$, $p\ge n-1$
when $\gamma_n=\beta_n$. This goes along the same lines. Let us consider the mirror situation
\begin{gather*}
\seq{\gamma}=
(\ldots,
\alpha_{m-1},\beta_{m},\ldots,\beta_{n},\ldots,\beta_{l-1},\alpha_{l},\ldots,\alpha_{r-1},\beta_r,\ldots)
\end{gather*}
with $2\le m\le n<l<r$ where $m$, $l$, $r$ are again f\/ixed and $n$ is varying from $m$ to $l-1$.

We f\/irst use (\ref{eqto*}) for $m=n$ to write
\begin{gather*}
\varpi_n^{\alpha*}\phi_n=\varpi_n^{\alpha*}\dB_n^\beta\phi_n^{\alpha*}=
\overline{\sigma}_n\varpi_n^{\alpha}\dB_{n-1}^\beta\phi_n^{\alpha*},
\end{gather*}
which implies that we shall need expressions for $\varpi_n^{\alpha}\phi_n^{\alpha*}$.
We use repeatedly (\ref{eqs2}) in combination with the previous relation to get
\begin{gather*}
\varpi_n^{\alpha*}\phi_n
\in \dB_{n-1}^\beta
\Span\big\{ \varpi_{n-1}^{\alpha}\phi_{n-1}^{\alpha*}, \varpi_n^{\alpha}\phi_n^{\alpha} \big\} \\
\hphantom{\varpi_n^{\alpha*}\phi_n}{}= \Span\big\{
\dB_{n-1}^\beta\varpi_{n-1}^{\alpha}\phi_{n-1}^{\alpha*}, \dB_n^\beta\varpi_n^{\alpha*}\phi_n^\alpha \big\} \\
\hphantom{\varpi_n^{\alpha*}\phi_n}{}= \Span\big\{
\dB_{n-1}^\beta\varpi_{n-1}^{\alpha}\phi_{n-1}^{\alpha*},
\dB_n^\beta\varpi_n^{\alpha}\phi_n^{\alpha*},
\dB_n^\beta\varpi_{n+1}^\alpha\phi_{n+1}^\alpha
\big\} \\
\hphantom{\varpi_n^{\alpha*}\phi_n}{}= \Span\big\{
\dB_{n-1}^\beta\varpi_{n-1}^{\alpha}\phi_{n-1}^{\alpha*},
\dB_n^\beta\varpi_n^{\alpha}\phi_n^{\alpha*},
\dB_{n+1}^\beta\varpi_{n+1}^{\alpha*}\phi_{n+1}^\alpha
\big\} \\
\hphantom{\varpi_n^{\alpha*}\phi_n}{}= \Span\big\{
\dB_{n-1}^\beta\varpi_{n-1}^{\alpha}\phi_{n-1}^{\alpha*},
\dB_n^\beta\varpi_n^{\alpha}\phi_n^{\alpha*},
\dB_{n+1}^\beta\varpi_{n+1}^{\alpha}\phi_{n+1}^{\alpha*},
\dB_{n+1}^\beta\varpi_{n+2}^{\alpha}\phi_{n+2}^{\alpha}
\big\} \\
\hphantom{\varpi_n^{\alpha*}\phi_n}{}=\cdots\\
\hphantom{\varpi_n^{\alpha*}\phi_n}{}=
\Span\big\{
\dB_{n-1}^\beta\varpi_{n-1}^{\alpha}\phi_{n-1}^{\alpha*},
\ldots,
\dB_{m-1}^\beta\varpi_{l-1}^{\alpha}\phi_{l-1}^{\alpha*},
\dB_{m-1}^\beta\varpi_{l}^{\alpha}\phi_{l}^{\alpha}
\big\} \\
\hphantom{\varpi_n^{\alpha*}\phi_n}{}=
\Span\big\{
\dB_{n-1}^\beta\varpi_{n-1}^{\alpha}\phi_{n-1}^{\alpha*},
\varpi_{n}^{\alpha}\phi_{n},
\ldots,
\varpi_{l-1}^{\alpha}\phi_{l-1},
\varpi_{l}^{\alpha}\phi_{l}\big\},
\end{gather*}
where in the last step we used $\gamma_l=\alpha_l$ and hence $\dB_{l-1}^\beta=\dB_l^\beta$ and
$\dB_{l-1}^\beta\phi_l^\alpha=\phi_l$.
Thus if $n>m$, also $\gamma_{n-1}=\beta_{n-1}$ so that also the f\/irst term is $\varpi_{n-1}^{\alpha}\phi_{n-1}$ and we have found that
\begin{gather}\label{eqsp1b}
\varpi_n^{\alpha*}\phi_n\in\Span\big\{\varpi_p^\alpha\phi_p^\alpha\big\}_{p=n-1}^l,
\qquad m<n\le l-1.
\end{gather}
Now we are left with the remaining case $n=m$, i.e., $\gamma_{n-1}=\alpha_{n-1}$ while $\gamma_n=\beta_n$.
This is the missing link that should connect a $\beta$ block to the previous $\alpha$ block at the
boundary $\ldots,\alpha_{m-1},\beta_m,\ldots$.

Using (\ref{eqs2}) repeatedly we get
\begin{gather*}
\varpi_{m-1}^\alpha\phi_{m-1}^{\alpha*}
\in
\Span\big\{
\varpi_{m-1}^\alpha\phi_{m-1}^\alpha,
\varpi_{m-2}^\alpha\phi_{m-2}^{\alpha*}
\big\} \\
\hphantom{\varpi_{m-1}^\alpha\phi_{m-1}^{\alpha*}}{}=
\Span\big\{
\varpi_{m-1}^\alpha\phi_{m-1}^\alpha,
\varpi_{m-2}^\alpha\phi_{m-2}^{\alpha},
\varpi_{m-2}^\alpha\phi_{m-2}^{\alpha*}
\big\} \\
\hphantom{\varpi_{m-1}^\alpha\phi_{m-1}^{\alpha*}}{}= \cdots\\
\hphantom{\varpi_{m-1}^\alpha\phi_{m-1}^{\alpha*}}{}=
\Span\big\{
\varpi_{m-1}^\alpha\phi_{m-1}^\alpha,
\varpi_{m-2}^\alpha\phi_{m-2}^{\alpha},\ldots,
\varpi_{k}^\alpha\phi_{k}^{\alpha},
\varpi_{k-1}^\alpha\phi_{k-1}^{\alpha*}
\big\}.
\end{gather*}
After multiplying with $\dB_{k}^\beta=\dB_{k+1}^\beta=\cdots=\dB_{n-1}^\beta$ we get
\begin{gather*}
\varpi_{m-1}^\alpha\phi_{m-1}^{\alpha*}\in
\Span\big\{\varpi_{p}^\alpha\phi_{p}\big\}_{p=k-1}^{n-1}.
\end{gather*}
We plug this in our previous expression for $\varpi_{n}^{\alpha*}\phi_{n}$ and we arrive at
\begin{gather}\label{eqsp2b}
\varpi_{m}^{\alpha*}\phi_{m}\in\Span\big\{\varpi_{p}^\alpha\phi_{p}\big\}_{p=k-1}^l.
\end{gather}
To summarize, (\ref{eqsp1}) and (\ref{eqsp2}) in the case $\gamma_n=\alpha_n$
and (\ref{eqsp1b}) and (\ref{eqsp2b}) in the case $\gamma_n=\beta_n$ show short recurrences
for the $\phi_p$ that fully rely on the recursion for the $\alpha$-related quantities:~$\phi_n^\alpha$~and~$\phi_n^{\alpha*}$.
They use only factors~$\varpi_p^\alpha$ and~$\varpi_p^{\alpha*}$ and in the
recurrence relations only the rational Szeg\H{o} parameters $\lambda_p^\alpha$ for the $\alpha$ sequence are used. The longest recursions are needed to switch from a~$\alpha$ to a~$\beta$ block, which is for
$n\in\{m-1,m\}$ as given in~(\ref{eqsp2}) and~(\ref{eqsp2b}) since these need all elements form both blocks.

This analysis should illustrate that as long as we are in a succession of $\alpha$'s
we are building up an upper Hessenberg block.
At the instant the $\alpha$ sequence switches to a $\beta$ sequence one starts building
a lower Hessenberg block, which switches back to upper Hessenberg when again $\alpha$'s enter the
$\gamma$ sequence etc. See Fig.~\ref{fig1}.
Of course if there are only $\alpha$'s in the sequence, we end up with just an upper Hessenberg
as in the classical case.
If we alternate between one~$\alpha$ and one~$\beta$'s we get the so called CMV type matrix which is
a~f\/ive-diagonal matrix with a characteristic block structure a given in Fig.~\ref{fig2}.

Suppose we start with a set of $\alpha$'s, and set
\begin{gather*}
\A=\diag(\alpha_0,\alpha_1,\ldots)
\end{gather*}
and let us recycle our earlier notations
$\varpi_n(z)=1-\oa_n z$ and $\varpi_n^*(z)=z-\alpha_n$
to now act on an arbitrary operator $\T$ by doing the transform for the whole sequence $\seq{\alpha}$ simultaneously
as follows: $\varpi_{\A}(\T)=\I-\T\A^*$ and $\varpi_{\A}^*(\T)=\T-\A$ where $\I$ is the identity operator.
Then (at least formally) we have
\begin{gather}\label{eqR}
[\phi_0,\phi_1,\ldots] \varpi_{\A}^*(z\I) =[\phi_0,\phi_1,\ldots]\varpi_{\A}(z\I)\hat{\G}
\end{gather}
with $\hat{\G}$ having the structure shown in Fig.~\ref{fig1}.
\begin{figure}[t]\centering
\raisebox{-3cm}{\includegraphics[scale=0.5]{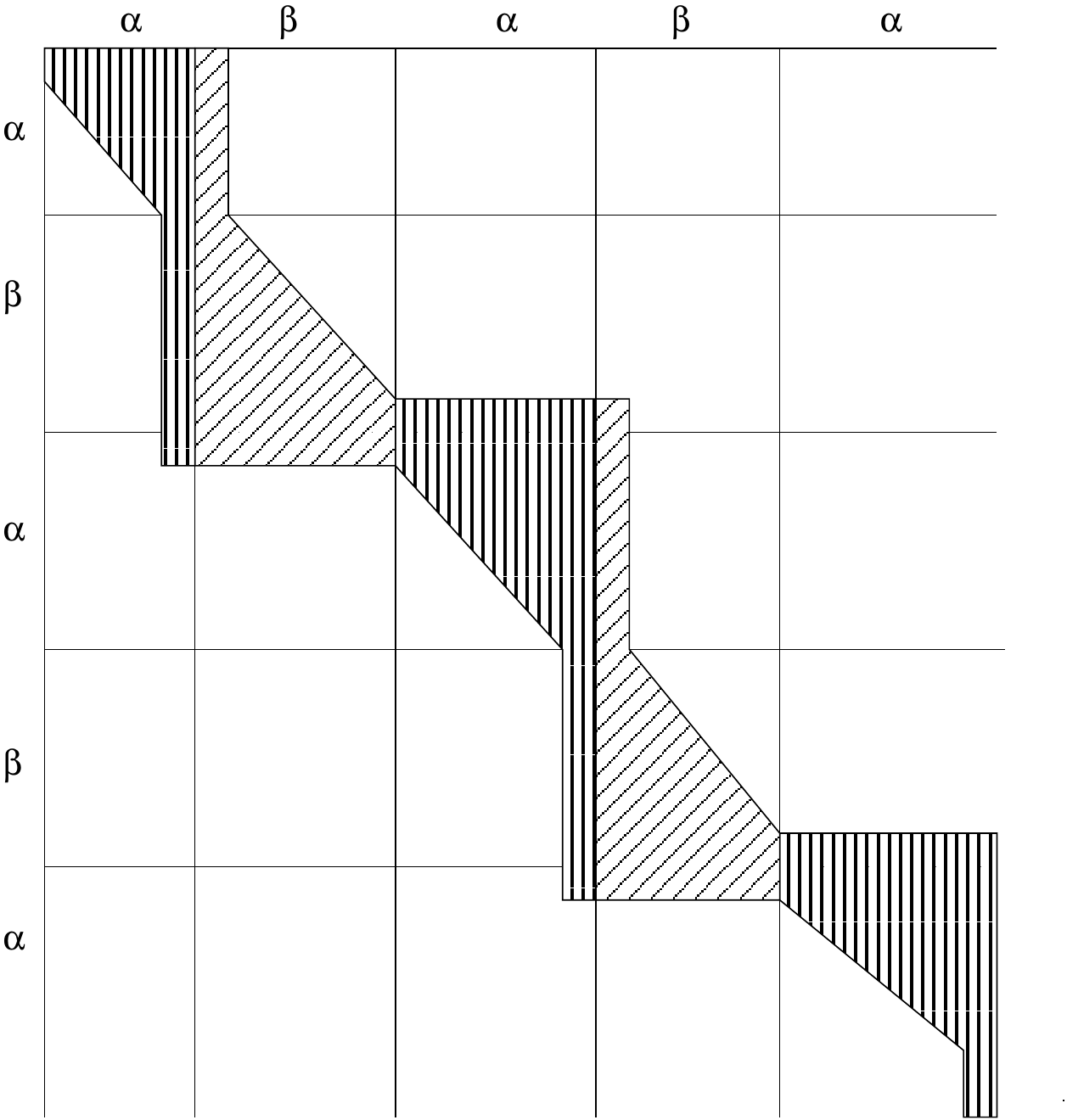}}~$=$~\raisebox{-3cm}{\includegraphics[scale=0.5]{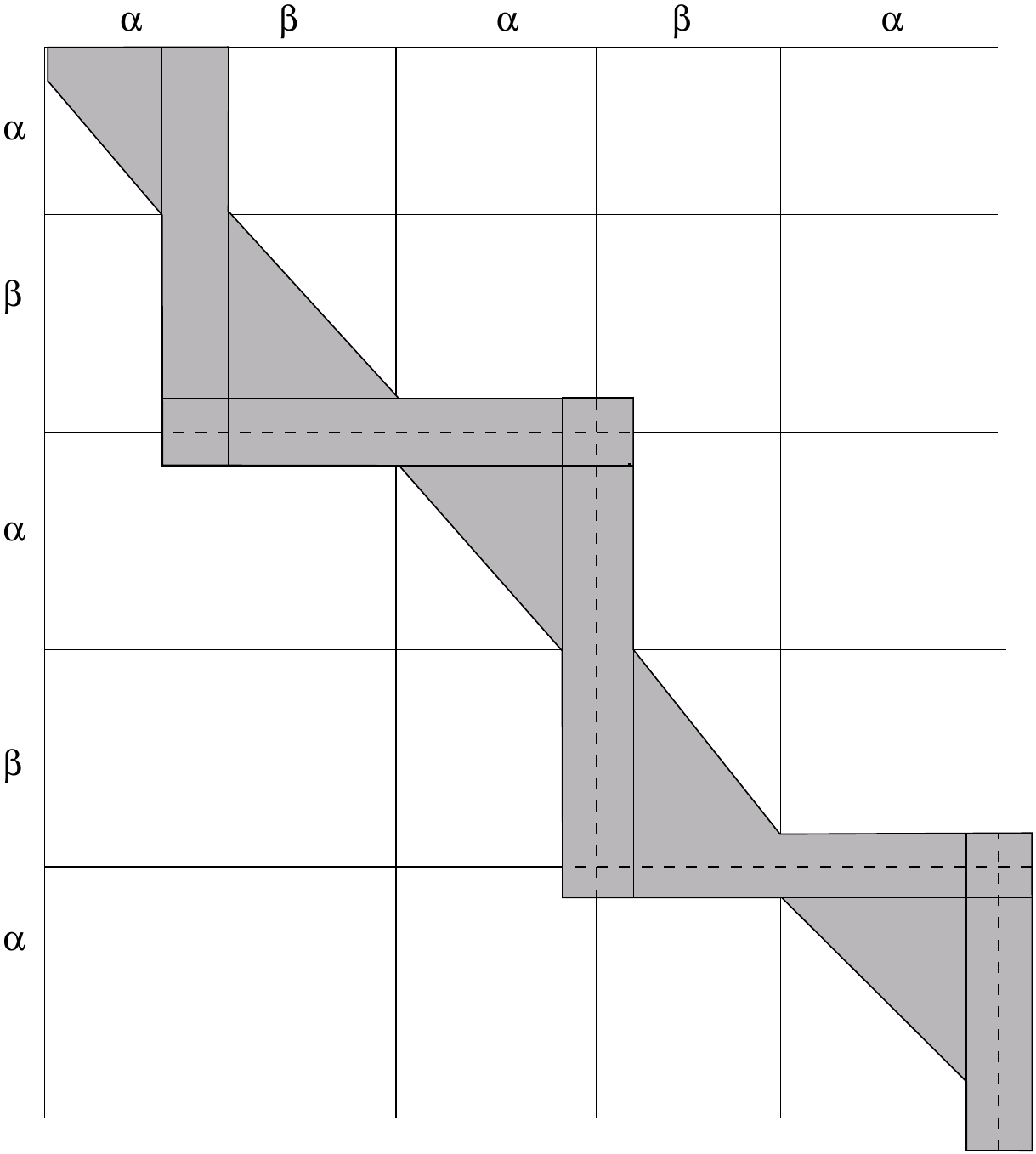}}
\caption{Structure of the matrix $\hat{\G}$ which is the same as the structure of the matrix $\G_{\A}$.}\label{fig1}
\end{figure}
In the special case that the $\alpha$'s and $\beta$'s alternate as in
\begin{gather*}
\seq{\gamma}=(\alpha_1,\beta_2,\alpha_3,\beta_4,\alpha_5,\ldots),
\end{gather*}
we get the f\/ive-diagonal matrix as in \cite{MISCCMV05} for polynomials and for ORF
in \cite{MISCVel07}.
\begin{figure}[t]\centering
\includegraphics[scale=0.5]{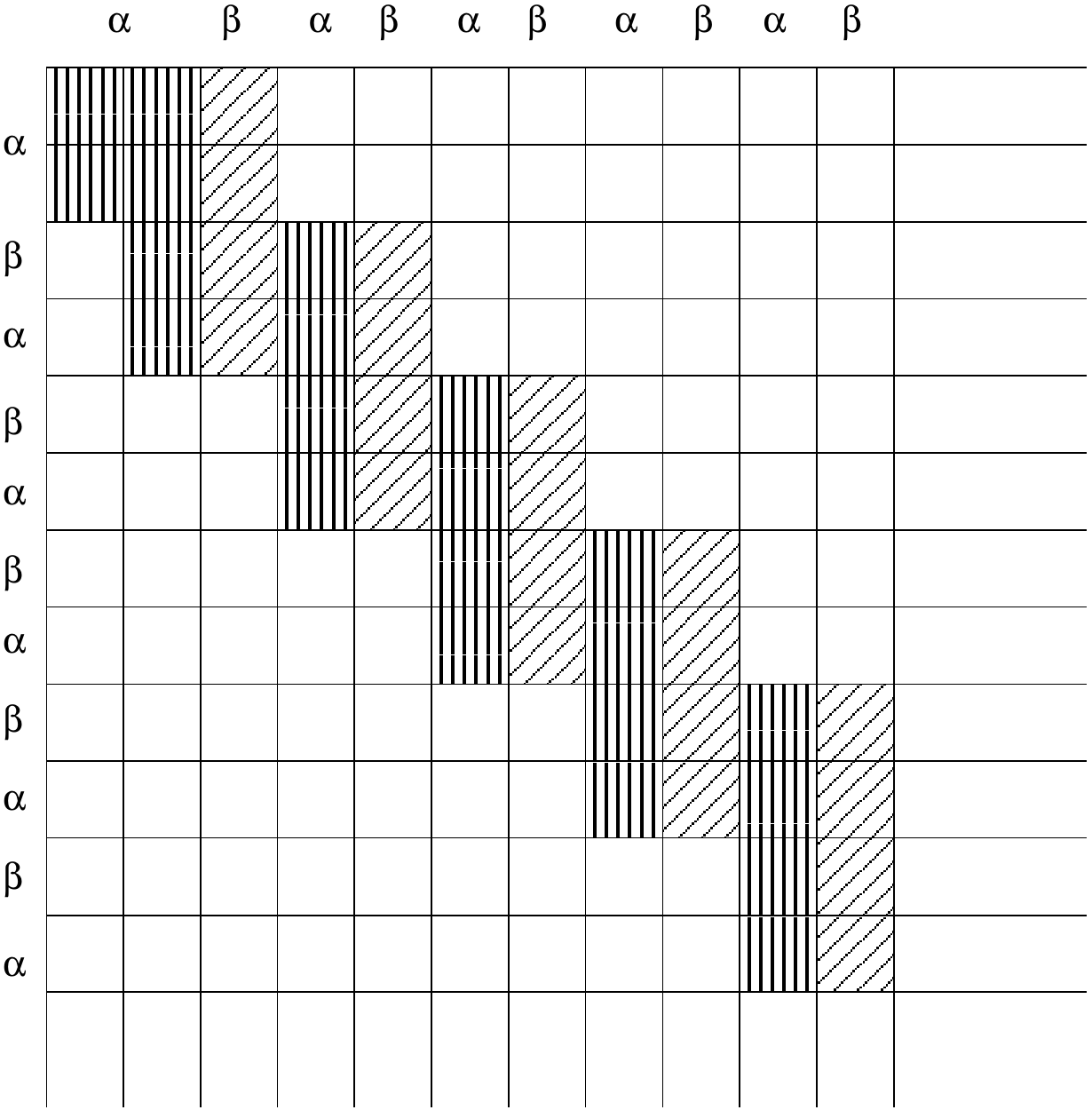}
\caption{Structure of the matrix $\hat{\G}_m$ when $\alpha$'s and $\beta$'s alternate, which is the
f\/ive-diagonal CMV matrix given in \cite{MISCCMV05,MISCVel07}.}\label{fig2}
\end{figure}
Since for $k\ge1$ each $\alpha_{2k-1}$-column is the last in an $\alpha$ block and each $\beta_{2k}$-column
is the f\/irst in a $\beta$ block, we get a succession of $4\times2$ blocks that shift
down by two rows as illustrated by Fig.~\ref{fig2}.

The particular role of the last column in an $\alpha$ block and the f\/irst column in
a $\beta$ block is due to the fact that we have chosen to derive these relations for the $\phi_k$
from the recurrence for the~$\phi_k^\alpha$ leading to factors $\varpi_{\A}$ and $\varpi_{\A}^*$
and the use of parameters $\lambda_k^\alpha$.
A symmetric derivation could have been given by starting from the $\phi_k^\beta$ recursion,
in which case the $\beta$ blocks will correspond to upper Hessenberg blocks and the
$\alpha$ blocks to lower Hessenberg blocks. Then the longer recurrence would occur in the last column of the $\beta$ block and the
f\/irst column of the $\alpha$ block.
The shape of the matrix would be the transposed shape of the f\/igures shown here.
However while it is rather simple to deal with $\alpha_j=0$ in the current derivation, it
requires dealing with more technical issues to write down the formulas with the corresponding
$\beta_j$ equal to $\infty$. Therefore we stick to the present form.

We can def\/ine as in \cite{MISCVel07} an operator M\"obius transform of an operator $\T$ by
\begin{gather*}
{\zeta}_{\A}(\T)=\eta_{\A}\varpi_{\A}(\T)^{-1}\varpi^*_{\A}(\T)\eta_{\A^*}^{-1},\qquad
\varpi_{\A}(\T)=\I-\T\A^*, \qquad
\varpi_{\A}^*(\T)=\T-\A,
\end{gather*}
and its inverse
\begin{gather*}
\tilde{\zeta}_{\A}(\T)=\eta_{\A}^{-1}\tilde{\varpi}_{\A}^*(\T)\tilde{\varpi}_{\A}(\T)^{-1}\eta_{\A^*},\qquad
\tilde{\varpi}_{\A}(\T)=\I+\A^*\T, \qquad
\tilde{\varpi}_{\A}^*(\T)=\T+\A,
\end{gather*}
where
\begin{gather*}
\eta_{\A}:=\varpi_{\A}(\A)^{1/2}=\sqrt{\I-\A\A^*}=\sqrt{\I-\A^*\A}=\eta_{\A^*}=\varpi_{\A^*}(\A^*)^{1/2}\\
\hphantom{\eta_{\A}}{} =\diag\big(1,\sqrt{1-|\alpha_1|^2},\sqrt{1-|\alpha_2|^2},\ldots\big).
\end{gather*}
Recall the meaning of the $(\ )^*$ notation for a (constant) matrix from Remark~\ref{rem3.1}.
Note that $\tilde{\varpi}_\A$ is like $\varpi_{\A}$, except that the minus sign is changed into
a plus sign and $\T\A^*$ is changed into~$\A^*\T$.
It should also be clear that the alternating case (i.e., the case of the CMV
representation) gives the smallest possible bandwidth, as was also reported in~\cite{MISCCMV05}.

Now it is not dif\/f\/icult to see that (\ref{eqR}) is equivalent to
\begin{gather*}
 [\phi_0(z),\phi_1(z),\phi_2(z),\ldots]\big(z\I-\tilde{\zeta}_{\A}(\G)\big)=0,\qquad
\hat{\G}=\eta_{\A}^{-1}{\G}\,\eta_{\A^*}.
\end{gather*}
To see this, exactly the same argument as used in \cite{MISCVel07} to derive relation (19) of that paper can be applied here.
Instead we give a direct derivation.
The relation (\ref{eqR}) can be written as
\begin{gather*}
0=\Phi\big[\varpi_{\A}^*(z\I)-\varpi_{\A}(z\I)\hat{\G}\big]=\Phi\big[(z\I-\A)-(\I-z\A^*)\hat{\G}\big]
=\Phi\big[z\big(\I+\A^*\hat{\G}\big)-\big(\A+\hat{\G}\big)\big].
\end{gather*}
Since $\tilde{\varpi}_\A(\hat{\G})=(\I+\A^*\hat{\G})$ is invertible, this is equivalent with
\begin{gather*}
0=\Phi\big[z\I-\tilde{\varpi}_\A^*\big(\hat{\G}\big)\tilde{\varpi}_{\A}\big(\hat{\G}\big)^{-1}\big]=
\Phi\big[z\I-\eta_\A^{-1}\tilde{\varpi}_\A^*(\G)\tilde{\varpi}_{\A}(\G)^{-1}\eta_{\A^*}\big]=
\Phi\big[z\I-\tilde{\zeta}_{\A}(\G)\big],
\end{gather*}
where we used the fact that the matrix $\eta_A=\eta_{\A^*}$ and its inverse commute with any
diagonal matrix.

The matrices $\G$ and $\hat{\G}$ have the same shape, but $\G$ is rescaled to make it isometric.
So we see that on $\L=\Span\{\phi_0,\phi_1,\ldots\}$
the operator $z\I$ has a matrix representation $\tilde{\zeta}_{\A}(\G)$
with respect to the $\phi$-basis.
We shall have a more careful discussion of this fact in the next section.

\section{Factorization of a general CMV matrix}\label{secMF}

It is well known (see \cite{ArtBC08, MISCVel07} in the rational case and \cite{MISCBD07} for the polynomial case) that the Hessenberg matrix
$\G=\cH$ (obtained when $\seq{\gamma}=\seq{\alpha}$) can be written as an inf\/inite product
of elementary matrices, i.e.,
\begin{gather*}
\G=\cH=G_{1}G_{2}G_{3}G_{4}\cdots,
\end{gather*}
where
\begin{gather}\label{eqG0}
G_{k}:=\left[
\begin{matrix}
I_{k-1} &0 &0 \\ 0 & \tilde{G}_{k} &0 \\ 0 & 0 & I_{\infty}
\end{matrix}\right], \qquad
\tilde{G}_{k}:=
\left[ \begin{matrix} -\delta_k & \eta_k \\ \eta_k & \overline{\delta_k}\end{matrix}\right],
\qquad \forall\, k \geq 1,
\end{gather}
where $I_{k-1}$ and $I_{\infty}$ are the identity matrices of sizes
$(k-1)$ and $\infty$ respectively, $I_0$ is the empty matrix, $\delta_k=\lambda_k^\alpha$ is the $k$-th Szeg\H{o}
parameter and $\eta_k:=\sqrt{1-|\delta_k|^2}$.

Also the CMV matrix $\C^\varepsilon$
associated with the alternating sequence $\seq{\gamma}=\seq{\varepsilon}$
where $\gamma_{2k-1}=\alpha_{2k-1}$ and $\gamma_{2k}=\beta_{2k}$, for all $k\ge1$
uses \emph{the same} elementary transforms, but they are multiplied in a~dif\/ferent order:
\begin{gather*}
\G=\C_o\C_e=
\left( \cdots G_{9}G_{7}G_{5}G_{3}G_{1}\right)
\cdot
\left( G_{2}G_{4}G_{6}G_{8}G_{10}\cdots \right)\\
\hphantom{G}{}=
\left(G_{1}G_{3}G_{5}G_{7}G_{9} \cdots\right)
\cdot
\left( G_{2}G_{4}G_{6}G_{8}G_{10}\cdots \right).
\end{gather*}

To f\/ind explicit expressions for these elementary factors in our case, taking into account the proper normalization, requires a more detailed analysis.
We start with the following
\begin{Theorem}\label{thmSA1}
In the case $\seq{\gamma}=\seq{\alpha}$, the recurrence relation of Theorem~{\rm \ref{thmrec}}
can be rewritten in this form
\begin{gather}\label{eqT9.1a}
\left[\begin{matrix}
\DS\frac{\varpi_{n-1}^{\alpha*}\phi_{n-1}^\alpha}{\sqrt{1-|\alpha_{n-1}|^2}}
&
\DS\frac{\varpi_{n}^{\alpha}\phi_{n}^{\alpha*}}{\sqrt{1-|\alpha_{n}|^2}}
\end{matrix}\right] =
\left[\begin{matrix}
\DS\frac{\varpi_{n-1}^{\alpha}\phi_{n-1}^{\alpha*}}{\sqrt{1-|\alpha_{n-1}|^2}}
&
\DS\frac{\varpi_{n}^{\alpha}\phi_{n}^\alpha}{\sqrt{1-|\alpha_{n}|^2}}
\end{matrix}\right]
\tilde{G}_n^\alpha
\end{gather}
with
\begin{gather*}
\tilde{G}_n^\alpha=\overline{\sigma}_{n-1}\overline{\eta}^\alpha_{n1}
\left[\begin{matrix}
-\lambda_n^\alpha\eta_{n1}^\alpha & \sqrt{1-|\lambda_n^\alpha|^2}\\
\sqrt{1-|\lambda_n^\alpha|^2} & \overline{\lambda_n^\alpha}\overline{\eta}^\alpha_{n1}
\end{matrix}\right]
\left[\begin{matrix}
1 & 0\\0 & \sigma_n
\end{matrix}\right]
\end{gather*}
or
\begin{gather}\label{eqT9.1b}
\left[\begin{matrix}\varpi_{n-1}^{\alpha*}\phi_{n-1}^\alpha & \varpi_{n}^{\alpha}\phi_{n}^{\alpha*}\end{matrix}\right] =
\left[\begin{matrix}\varpi_{n-1}^{\alpha}\phi_{n-1}^{\alpha*} & \varpi_{n}^{\alpha}\phi_{n}^\alpha\end{matrix}\right] \hat{G}_n^\alpha
\end{gather}
with
\begin{gather*}
\hat{G}_n^\alpha=\left[\begin{matrix}\big(1-|\alpha_{n-1}|^2\big)^{-1/2} & 0\\
0 & \big(1-|\alpha_{n}|^2\big)^{-1/2}\end{matrix}\right]
\tilde{G}_n^\alpha
\left[\begin{matrix}\big(1-|\alpha_{n-1}|^2\big)^{1/2} & 0\\
0 & \big(1-|\alpha_{n}|^2\big)^{1/2}\end{matrix}\right].
\end{gather*}
\end{Theorem}
\begin{proof}
This is in fact an explicit form of what in the previous section was expressed as (\ref{eqs1}) and (\ref{eqs2}).

Taking the f\/irst line of the recurrence relation from Theorem~\ref{thmrec} and solving for $\varpi^{\alpha*}_{n-1}\phi^{\alpha}_{n-1}$
we get the explicit from of (\ref{eqs0}), i.e.,
\begin{gather*}
\varpi_{n-1}^{\alpha*}\phi_{n-1}^\alpha=
(e_n^\alpha)^{-1}\overline{\eta}^\alpha_{n1}\overline{\sigma}_{n-1}
 \varpi_n^\alpha\phi_n^\alpha-
\lambda_n^\alpha\overline{\sigma}_{n-1}
 \varpi_{n-1}^\alpha\phi_{n-1}^{\alpha*}.
\end{gather*}
Using $e_n^\alpha=\frac{\sqrt{1-|\alpha_{n}|^2}}{\sqrt{1-|\alpha_{n-1}|^2}}\frac{1}{\sqrt{1-|\lambda_n^\alpha|^2}}$
and $\eta_{n2}^\alpha=\overline{\eta}_{n1}^\alpha\overline{\sigma}_{n-1}\sigma_n$
this becomes
\begin{gather}\label{eqeta}
\frac{\varpi_{n-1}^{\alpha*}\phi_{n-1}^{\alpha}}{\sqrt{1-|\alpha_{n-1}|^2}}=
\overline{\eta}^\alpha_{n1}\overline{\sigma}_{n-1}
\left[-\lambda_n^\alpha\eta^\alpha_{n1}\frac{\varpi_{n-1}^\alpha\phi_{n-1}^{\alpha*}}{\sqrt{1-|\alpha_{n-1}|^2}}
+ \sqrt{1-|\lambda_n^\alpha|^2}\frac{\varpi_n^\alpha\phi_n^\alpha}{\sqrt{1-|\alpha_n|^2}}\right].
\end{gather}
The second line of the recurrence relation in Theorem~\ref{thmrec} is
\begin{gather*}
\varpi_{n}^{\alpha}\phi_{n}^{\alpha*}=
e^\alpha_n{\eta}^\alpha_{n2}{\sigma}_{n-1}
 \overline{\lambda_n^\alpha}\varpi_{n-1}^{\alpha*}\phi_{n-1}^\alpha+
e_n^\alpha\eta_{n2}^\alpha
 \varpi_{n-1}^\alpha\phi_{n-1}^{\alpha*}.
\end{gather*}
Eliminate the term $\varpi_{n-1}^{\alpha*}\phi_{n-1}^\alpha$ with the f\/irst line of the recurrence
and after inserting the value of $e_n^\alpha$ from (\ref{eqe2}) and rearranging, we get
\begin{gather*}
\frac{\varpi_n^\alpha\phi_n^\alpha}{\sqrt{1-|\alpha_n|^2}}=
\overline{\eta}_{n1}^\alpha\overline{\sigma}_{n-1}
\left[
\frac{\varpi_{n-1}^\alpha\phi_{n-1}^\alpha}{\sqrt{1-|\alpha_{n-1}|^2}}
\sqrt{1-|\lambda_n^\alpha|^2}\sigma_n
+\frac{\varpi_{n}^\alpha\phi_{n}^\alpha}{\sqrt{1-|\alpha_{n}|^2}}
\overline{\lambda_n^\alpha}\overline{\eta}_{n1}^\alpha\sigma_n
\right].
\end{gather*}
This is equivalent with the formula (\ref{eqT9.1a}).
Derivation of the other form (\ref{eqT9.1b}) is obvious.
\end{proof}

Note that $\tilde{G}_n^\alpha$ is unitary but $\hat{G}_n^\alpha$ is not
unless $|\lambda_n^\alpha|=1$ or $\alpha_n=\alpha_{n-1}$.

Now suppose that $n\in\mathbbm{a}_n$, i.e., $\gamma_n=\alpha_n$ then $\dB_n^\beta=\dB_{n-1}^\beta$ and hence
\begin{gather}\label{eq71}
\left[\begin{matrix}
\varpi_{n-1}^{\alpha*}\phi_{n-1}^\alpha\dB_{n-1}^\beta
&
\varpi_{n}^{\alpha}\phi_{n}^{\alpha*}\dB_n^\beta
\end{matrix}\right] =
\left[\begin{matrix}
\varpi_{n-1}^{\alpha}\phi_{n-1}^{\alpha*}\dB_{n-1}^\beta
&
\varpi_{n}^{\alpha}\phi_{n}^\alpha\dB_n^\beta
\end{matrix}\right]
\hat{G}^\alpha_n.
\end{gather}

If $n\in\mathbbm{b}_n$, i.e., $\gamma_n=\beta_n$, then $\dB_n^\beta=\dB_{n-1}^\beta\zeta_n^\beta=
\overline{\sigma}_n\dB_{n-1}^\beta\frac{\varpi_n^{\alpha}}{\varpi_n^{\alpha*}}$.
Def\/ine then
\begin{gather}\label{eqab}
\tilde{G}_n^\beta=
\left[\begin{matrix}\overline{\sigma}_n&0\\0&1\end{matrix}\right]
\tilde{G}_n^\alpha
\left[\begin{matrix}\sigma_n&0\\0&1\end{matrix}\right]=
\tilde{S}_n^*\tilde{G}_n^\alpha\tilde{S}_n,\qquad \tilde{S}_n=\diag(\sigma_n,1),
\end{gather}
and as for $\hat{G}_n^\alpha$
\begin{gather*}
\hat{G}_n^\beta =\left[\begin{matrix}\big(1-|\alpha_{n-1}|^2\big)^{-1/2} & 0\\
0 & \big(1-|\alpha_{n}|^2\big)^{-1/2}\end{matrix}\right]
\tilde{G}_n^\beta
\left[\begin{matrix}\big(1-|\alpha_{n-1}|^2\big)^{1/2} & 0\\
0 & \big(1-|\alpha_{n}|^2\big)^{1/2}\end{matrix}\right] \\
\hphantom{\hat{G}_n^\beta}{}=\left[\begin{matrix}\overline{\sigma}_n&0\\0&1\end{matrix}\right] \hat{G}_n^{\alpha} \left[\begin{matrix}\sigma_n&0\\0&1\end{matrix}\right],
\end{gather*}
so that
\begin{gather*}
\left[\begin{matrix}
\varpi_{n-1}^{\alpha}\phi_{n-1}^{\alpha*}\dB_{n-1}^\beta
&
\varpi_{n}^{\alpha*}\phi_{n}^\alpha\dB_n^\beta\end{matrix}\right]
\hat{G}^\beta_n\\
\qquad{}=
\left[\begin{matrix}
\varpi_{n-1}^{\alpha}\phi_{n-1}^{\alpha*}\dB_{n-1}^\beta
&
\varpi_{n}^{\alpha}\phi_{n}^\alpha\dB_{n-1}^\beta\end{matrix}\right]\overline{\sigma}_n\hat{G}_n^\alpha\left[\begin{matrix}\sigma_n&0\\0&1\end{matrix}\right]\\
\qquad{}=
\left[\begin{matrix}
\varpi_{n-1}^{\alpha*}\phi_{n-1}^{\alpha}\dB_{n-1}^\beta
&
\varpi_{n}^{\alpha}\phi_{n}^{\alpha*}\dB_{n-1}^\beta\end{matrix}\right]\left[\begin{matrix}1&0\\0&\overline{\sigma}_n\end{matrix}\right]\\
\qquad{}=
\left[\begin{matrix}
\varpi_{n-1}^{\alpha*}\phi_{n-1}^{\alpha}\dB_{n-1}^\beta
&
\varpi_{n}^{\alpha*}\phi_{n}^{\alpha*}\dB_{n}^\beta\end{matrix}\right].
\end{gather*}
Thus
\begin{gather}\label{eq72}
\left[\begin{matrix}
\varpi_{n-1}^{\alpha*}\phi_{n-1}^{\alpha}\dB_{n-1}^\beta
&
\varpi_{n}^{\alpha*}\phi_{n}^{\alpha*}\dB_{n}^\beta\end{matrix}\right]
=\left[\begin{matrix}
\varpi_{n-1}^{\alpha}\phi_{n-1}^{\alpha*}\dB_{n-1}^\beta
&
\varpi_{n}^{\alpha*}\phi_{n}^\alpha\dB_n^\beta\end{matrix}\right]
\hat{G}^\beta_n.
\end{gather}

We build the general CMV matrix as a product of $G$-factors.
Set $\alpha_0=0$ and
suppose we consider for $1 \leq n < k < m < l < \cdots\le\infty$
\begin{gather*}
\alpha_0,\alpha_1,\ldots,\alpha_{n-1}|\beta_n,\ldots,\beta_{k-1}|\alpha_k,\ldots,\alpha_{m-1}|\beta_m,\ldots,\beta_{\ell-1}|\alpha_\ell,\ldots.
\end{gather*}
We use vertical bars to separate the dif\/ferent blocks. These separations are of course clear from the notation
$\alpha$ and $\beta$, but since that will not so obvious when we denote the corresponding functions in the vectors below, we introduced them here and keep them at the corresponding positions in the vectors that we are about to introduce now.
That should improve readability of the formulas.

We consider the vector $\Phi=[\phi_0,\phi_1,\phi_2,\ldots]$ so that
\begin{gather*}
\Phi(\I-\A^*z)=
\big[
\varpi_0^\alpha\phi_0,\ldots,\varpi_{n-1}^\alpha\phi_{n-1}
\big|
\varpi_{n}^\alpha\phi_{n},\ldots,
\varpi_{k-1}^\alpha\phi_{k-1}\big|
\varpi_{k}^\alpha\phi_{k},\ldots
\\
\hphantom{\Phi(\I-\A^*z)=\big[}{}
\ldots,\varpi_{m-1}^\alpha\phi_{m-1}\big|
\varpi_{m}^\alpha\phi_{m},\ldots,
\varpi_{\ell-1}^\alpha\phi_{\ell-1}\big|
\varpi_{\ell}^\alpha\phi_{\ell},\ldots
\big]\\
\hphantom{\Phi(\I-\A^*z)}{}=
\big[
\varpi_0^\alpha\phi^\alpha_0,\ldots,\varpi_{n-1}^\alpha\phi_{n-1}^\alpha\dB_{n-1}^\beta
\big|
\varpi_{n}^\alpha\phi_{n}^{\alpha*}\dB_{n}^\beta,\ldots,
\varpi_{k-1}^\alpha\phi_{k-1}^{\alpha*}\dB_{k-1}^\beta\big|
\varpi_{k}^\alpha\phi_{k}^{\alpha}\dB_{k}^\beta,\ldots
\\
\hphantom{\Phi(\I-\A^*z)=\big[}{}
\ldots,\varpi_{m-1}^\alpha\phi_{m-1}^{\alpha}\dB_{m-1}^\beta\big|
\varpi_{m}^\alpha\phi_{m}^{\alpha*}\dB_{m}^\beta,\ldots,
\varpi_{\ell-1}^\alpha\phi_{\ell-1}^{\alpha*}\dB_{\ell-1}^\beta\big|
\varpi_{\ell}^\alpha\phi_{\ell}^{\alpha}\dB_{\ell}^\beta,\ldots
\big].
\end{gather*}

Def\/ine for $\nu\in\{\alpha,\beta\}$ (recall $\eta_\A=\sqrt{\I-\A^*\A}$)
\begin{gather}\label{eqG}
G_n^\nu=\left[\begin{matrix} I_{n-1}& 0 & 0\\ 0 &\tilde{G}_n^\nu & 0 \\
0 & 0 & I_\infty\end{matrix}\right]=\diag(I_{n-1},\tilde{G}_n^\nu,I_\infty),\qquad\text{and}\qquad
\hat{G}_n^\nu=\eta_\A^{-1}G_n^\nu \eta_\A.
\end{gather}

Apply successively the $G_n^\nu$ to the right on the $\Phi$ vector.
Keep an increasing order for the factors {\it within} an $\alpha$ block
and a decreasing order for the factors {\it covering} a $\beta$ block.
Thus for the example given above the order will be
\begin{gather}
\seq{\gamma}= (\alpha_0,\alpha_1,\ldots,\alpha_{n-1}\|\beta_n,\ldots,\beta_{k-1}|\alpha_k\|\alpha_{k+1},
\ldots,\alpha_{m-1}\|\beta_m,\ldots,\beta_{\ell-1}|\alpha_\ell\|\alpha_{\ell+1},\ldots ),\nonumber\\
\hat{\G}=
\underbrace{(\hat{G}_1^\alpha\hat{G}_2^\alpha\cdots \hat{G}_{n-1}^\alpha)}_{\DS\hat{G}_\alpha^1}
\underbrace{(\hat{G}_{k}^\alpha\hat{G}_{k-1}^\beta\cdots \hat{G}_n^\beta)}_{\DS\hat{G}_\beta^1}
\underbrace{(\hat{G}_{k+1}^\alpha\cdots \hat{G}_{m-1}^\alpha)}_{\DS\hat{G}_\alpha^2}
\underbrace{(\hat{G}_{\ell}^\alpha\hat{G}_{\ell-1}^\beta\cdots \hat{G}_m^\beta)}_{\DS\hat{G}_\beta^2}
(\hat{G}_{\ell+1}^\alpha\cdots)\label{eqGblocks}\\
\hphantom{\hat{\G}}{}=
\eta_\A^{-1}
\underbrace{
\underbrace{({G}_1^\alpha{G}_2^\alpha\cdots {G}_{n-1}^\alpha)}_{\DS{G}_\alpha^1}
\underbrace{({G}_{k}^\alpha{G}_{k-1}^\beta\cdots {G}_n^\beta)}_{\DS{G}_\beta^1}
\underbrace{({G}_{k+1}^\alpha\cdots {G}_{m-1}^\alpha)}_{\DS{G}_\alpha^2}
\underbrace{({G}_{\ell}^\alpha{G}_{\ell-1}^\beta\cdots {G}_m^\beta)}_{\DS{G}_\beta^2}
({G}_{\ell+1}^\alpha\cdots)}
\eta_\A \nonumber\\
\hphantom{\hat{\G}}{}= \eta_{\A}^{-1}\rule{64mm}{0pt}\G\rule{64mm}{0pt} \eta_\A \nonumber
\end{gather}
with $\hat{G}_n^\nu$ and $G_n^\nu$, $\nu\in\{\alpha,\beta\}$ as in (\ref{eqG}) and similarly
\begin{gather}\label{eqcalG}
\hat{G}_\nu^n=\eta_\A^{-1}G_\nu^n\eta_\A,\qquad \text{for $n=1,2,\ldots$; \ \ $\nu\in\{\alpha,\beta\}$}\qquad \text{and}\qquad
\hat{\G}=\eta_\A^{-1}\G\eta_\A.
\end{gather}
Note that we have aligned these formulas properly so that a shift in the blocks from the f\/irst to the second line in~(\ref{eqGblocks}) is made clear.
The rationale is that in a $G_\beta^n$ group the indices are decreasing from left to right and in a~$G_\alpha^n$
group the indices increase from left to right. This is what def\/ines the shape of the result.
We used the vertical bars to separate $\alpha$ and $\beta$ blocks in the sequence $\seq{\gamma}$ as before.
Now we use double vertical bars to indicate what corresponds to the~$G_\alpha^n$ and $G_\beta^n$ blocks on the third line.
On the second line in~(\ref{eqGblocks}) you may note that the f\/irst $\alpha$ in an~$\alpha$ block of the $\gamma$ sequence
corresponds to a~$\hat{G}^\alpha$ factor that migrates to the front of the product group of previous~$\beta$ block.
This is due to the fact that we analysed the $\gamma$ sequence using the parameters $\lambda_k^\alpha$ of the $\alpha$ recurrence (recall the def\/initions of the $\hat{G}_k^\alpha$ and $\hat{G}_k^\beta$ from Theorem~\ref{thmSA1} and equation~(\ref{eqab})).

To see the ef\/fect of applying all these elementary operations, we start with the f\/irst block~$G_\alpha^1$.
Using (\ref{eq71}) when multiplying $\Phi(\I-\A^*z)$ with $\hat{G}_\alpha^1$ gives (note $\phi_0=\phi_0^*$)
\begin{gather*}
\big[
\varpi_0^{\alpha*}\phi_0,\ldots,\varpi_{n-2}^{\alpha*}\phi_{n-2},
\varpi_{n-1}^\alpha\phi_{n-1}^{\alpha*}\dB_{n-1}^\beta\big|\varpi_n^\alpha\phi_n,\ldots
\big].
\end{gather*}
While multiplying this result with $\hat{G}_k^\alpha \hat{G}_{k-1}^\beta\cdots \hat{G}_{n+1}^\beta$, we make use of (\ref{eq71}) and (\ref{eq72}) to obtain
\begin{gather*}
\big[
\varpi_0^{\alpha*}\phi_0,\ldots,\varpi_{n-2}^{\alpha*}\phi_{n-2},
\varpi_{n-1}^\alpha\phi_{n-1}^{\alpha*}\dB_{n-1}^\beta\big|\\
\qquad{}
\varpi_{n}^{\alpha*}\phi_{n}^\alpha\dB_n,\varpi_{n+1}^{\alpha*}\phi_{n+1},\ldots,\varpi_{k-1}^{\alpha*}\phi_{k-1}\big|
\varpi_k^{\alpha}\phi_k^*,\varpi_{k+1}^\alpha\phi_{k+1},\ldots\big]
\end{gather*}
and the remaining multiplication $\hat{G}_n^\beta$ of $\hat{G}_\beta^1$ links the current $\beta$ block to the previous $\alpha$ block, resulting in
\begin{gather*}
\big[
\varpi_0^{\alpha*}\phi_0,\ldots,\varpi_{n-1}^{\alpha*}\phi_{n-1}\big|
\varpi_{n}^{\alpha*}\phi_{n},\ldots,\varpi_{k-1}^{\alpha*}\phi_{k-1}\big|
\varpi_k^{\alpha}\phi_k^*,\varpi_{k+1}^\alpha\phi_{k+1},\ldots\big].
\end{gather*}
The next block is again an $\alpha$ block treated by the product $\hat{G}_\alpha^2$,
and one may continue like this to f\/inally get
\begin{gather}\label{eq73}
\Phi(\I-\A^*z)\hat{\G}=\Phi(z\I-\A)\qquad \text{or}\qquad
\Phi(\I-\A^*z){\G}=\Phi(z\I-\A),~~\hat{\G}=\eta_\A^{-1}\hat{\G}\eta_\A,
\end{gather}
since $\eta_\A$ is invertible.
This $\hat{\G}$ can be factored as the product of elementary matrices having only one non-trivial $2\times2$ diagonal block: $\hat{\G}=\hat{G}_1\hat{G}_2\cdots$.
By setting $\hat{\G}=\eta_\A^{-1}{\G}\eta_\A=(\eta_\A^{-1}G_1\eta_\A)(\eta_\A^{-1}G_2\eta_\A)\cdots=G_1G_2\cdots$, we made $\G$ as well as these elementary factors $G_k$ unitary.

These elementary factors are grouped together in blocks as a consequence of the $\alpha$ and
$\beta$ blocks in the $\gamma$ sequence.
For example $G_\beta^2=G_\ell^\alpha G_{\ell-1}^\beta\cdots G_m^\beta$
corresponds to a $\beta$ block. Because $\ell>m$ there must be at least two
factors in such a $\beta$ block: the f\/irst one is $G_l^\alpha$ and the last one is a $G_m^\beta$.
However for an $\alpha$ block we have the product
$G_\alpha^2=G_{k+1}^\alpha\cdots G_{m-1}^\alpha$ thus if $k=m-1$ then the
initial index $k+1$ is larger than the end index $m-1$. Thus if an $\alpha$ block has only one element
then there are no factors in this product, which means that this $G_\alpha$ is just the identity.

In the case of $\seq{\gamma}=\seq{\alpha}$ (no $\beta$'s), then there is of course only one inf\/inite
block so $n=\infty$ and there is no $k,m,\ell,\dots$ and
\begin{gather*}
\G=G_1^\alpha G_2^\alpha G_3^\alpha\cdots =:\mathcal{H}
\end{gather*}
is the familiar upper Hessenberg matrix, and in the case of alternating $\alpha$-$\beta$ sequence,
we have in the case $\alpha_1,\beta_2,\alpha_3,\beta_4,\ldots$
\begin{gather*}
\G=G_1^\alpha\big(G_3^\alpha G_2^\beta\big)I_\infty\big(G_5^\alpha G_4^\beta\big)I_\infty\big(G_7^\alpha G_6^\beta\big)\cdots = \big(G_1^\alpha G_3^\alpha G_5^\alpha \cdots\big)\big(G_2^\beta G_4^\beta G_6^\beta\cdots\big)=\C=\C_o^\alpha \C_e^\beta,
\end{gather*}
where we grouped the product of the $G_k^\alpha$ factors as $\C_o^\alpha$ and of the $G_k^\beta$ factors as $\C_e^\beta$.
This rearrangement of the factors is possible because the blocks commute when their subscript dif\/fers by at least two.
If we def\/ine $\cS_e=\diag\big(1,\tilde{S}_2,\tilde{S}_4,\tilde{S}_6,\ldots\big)$, $\tilde{S}_{2k}=\diag(\sigma_{2k},1)$,
then of course $\C=\C_o^\alpha\cS_e^*\C_e^\alpha\cS_e$.

In the case of an alternating sequence of the form $\beta_1,\alpha_2,\beta_3,\alpha_4,\ldots$, then
\begin{gather*}
\G=\big(G_2^\alpha G_1^\beta\big)I_\infty\big(G_4^\alpha G_3^\beta\big)I_\infty\big(G_6^\alpha G_5^\beta\big)\cdots = \big(G_2^\alpha G_4^\alpha G_6^\alpha \cdots\big)\big(G_1^\beta G_3^\beta G_5^\beta\cdots\big)=\C=\C_e^\alpha\C_o^\beta.
\end{gather*}
As in the previous case this is $\C=\C_e^\alpha\cS_o^*\C_o^\alpha\cS_o$
with
\begin{gather*}
\cS_o=\diag\big(\tilde{S}_1,\tilde{S}_3,\tilde{S}_5,\ldots\big), \qquad \tilde{S}_{2k-1}=\diag(\sigma_{2k-1},1).
\end{gather*}

These are the classical CMV matrices as also given in \cite{MISCVel07},
except for the $\cS$ factors, which were not in \cite{MISCVel07}.
That is because we have used a particular normalization for the $\phi_n$-basis.
A~slightly dif\/ferent normalization of the $\phi_n$-basis will remove the $\cS$ factors.
Indeed, replacing all~$\phi_n$ by $\varphi_n^\alpha=\varsigma^\beta_n\phi_n$, where
\begin{gather*}
\varsigma_n^\beta=\frac{\dB_n^\beta}{|\dB_n^\beta|}=\mathop{\prod}_{j\in\mathbbm{b}_n}\sigma_j,
\end{gather*}
will do.

To see this, note that the relation (\ref{eqab}) between $\tilde{G}_n^\beta$
and $\tilde{G}_n^\alpha$ can also be written as (multiply with $\sigma_n\overline{\sigma}_n=1$)
\begin{gather}\label{eqab2}
\tilde{G}_n^\beta=
\left[\begin{matrix}1 & 0\\0 & \sigma_n\end{matrix}\right]
\tilde{G}_n^\alpha
\left[\begin{matrix} 1 & 0 \\ 0 & \overline{\sigma}_n\end{matrix}\right].
\end{gather}
The f\/irst and the last of these factors can then be moved to the $\varphi_n^\alpha$
so that the $\beta$-relation (\ref{eq72}) now becomes
\begin{gather}\label{eq72b}
\left[\begin{matrix}
\varpi_{n-1}^{\alpha}\dvarsigma_{n-1}^\beta\phi_{n-1}^{\alpha}\dB_{n-1}^\beta
&
\varpi_{n}^{\alpha*}\dvarsigma_{n}^\beta\phi_{n}^{\alpha*}\dB_{n}^\beta\end{matrix}\right]
=\left[\begin{matrix}
\varpi_{n-1}^{\alpha}\dvarsigma_{n-1}^\beta\phi_{n-1}^{\alpha*}\dB_{n-1}^\beta
&
\varpi_{n}^{\alpha*}\dvarsigma_{n}^\beta\phi_{n}^\alpha\dB_n^\beta\end{matrix}\right]
\hat{G}^\alpha_n.
\end{gather}
Note that the multiplication on the right is now with $\hat{G}^\alpha_n$ and not with $\hat{G}^\beta_n$ anymore.

For the $\alpha$-case, i.e., $\gamma_n=\alpha_n$ nothing essentially changes since then $\dvarsigma_n^\beta=\dvarsigma_{n-1}^\beta$.

It should be clear that following the same arguments used above,
the relation (\ref{eq73}) then becomes
\begin{gather*}
\Phi^\alpha(\I-\A^*z)\hat{\G}^\alpha=\Phi^\alpha(z\I-\A)
\qquad
\text{with}\quad
\Phi^\alpha=[{\varphi}^\alpha_0,{\varphi}^\alpha_1,{\varphi}^\alpha_2,\ldots]
\end{gather*}
and $\hat{\G}^\alpha$ is exactly like $\hat{\G}$ in (\ref{eqcalG}), except that
all $G_j^\beta$ should be replaced by a $G_j^\alpha$, that is, replace all the $G_k^\beta$ factors in $\G$
by the corresponding $G_k^\alpha$ by removing the $S_k$ factors in (\ref{eqab}) and $\G$ becomes $\G^\alpha$.

From (\ref{eq73}) we derive that
\begin{gather*}
z\I=\big(\hat{\G}+\A\big)\big(\I+\A^*\hat{\G}\big)^{-1}=\eta_\A^{-1}(\G+\A)\big(\I+\A^*\G\big)^{-1}\eta_\A=\tilde{\zeta}_\A(\G),
\end{gather*}
which is the matrix representation of the shift operator in $\L$ with respect to the basis
$(\phi_0,\phi_1$, $\phi_2,\ldots)$.
With respect to the basis $\Phi^\alpha$, the expression is the same except that $\G$ should be replaced by~$\G^\alpha$.

We summarize our previous results.

\begin{Theorem}
In the general case of a sequence $\seq{\gamma}$, then multiplying the vector of basis functions $\Phi(z)=[\phi_0,\phi_1(z),\phi_2(z),\ldots]$
by $z$ corresponds to multiplying it from the right with the infinite matrix
\begin{gather*}
\big(\hat{\G}+\A\big)\big(\I+\A^*\hat{\G}\big)^{-1}=\eta_\A^{-1}(\G+\A)(\I+\A^*\G)^{-1}\eta_\A,
\end{gather*}
where $\A=\diag(\alpha_0,\alpha_1,\ldots)$, $\eta_\A=\sqrt{\I-\A^*\A}$ and $\G$ is a product of unitary
factors $G_k$ defined in~\eqref{eqG} where the order of multiplication is from left to right in a block of successive $\alpha$-values and from right to left in a block of successive $\beta$-values as explained in detail above.

For the sequence $\seq{\alpha}$ we get the classical Hessenberg matrix
\begin{gather*}
\G=G_1^\alpha G_2^\alpha G_3^\alpha \cdots =\mathcal{H}.
\end{gather*}
The classical CMV matrices are obtained for
$\alpha_1,\beta_2,\alpha_3,\beta_4,\ldots$ as
\begin{gather*}
\G=G_1^\alpha \big(G_3^\alpha G_2^\beta\big)\big(G_5^\alpha G_4^\beta\big)\big(G_7^\alpha G_6^\beta\big)\cdots = \big(G_1^\alpha G_3^\alpha G_5^\alpha \cdots\big)\big(G_2^\beta G_4^\beta G_6^\beta\cdots\big)=\C=\C_o^\alpha\C_e^\beta
\end{gather*}
and in the case $\beta_1,\alpha_2,\beta_3,\alpha_4,\ldots$, then
\begin{gather*}
\G=\big(G_2^\alpha G_1^\beta\big)\big(G_4^\alpha G_3^\beta\big)\big(G_6^\alpha G_5^\beta\big)\cdots = \big(G_2^\alpha G_4^\alpha G_6^\alpha \cdots\big)\big(G_1^\beta G_3^\beta G_5^\beta\cdots\big)=\C=\C_e^\alpha\C_o^\beta.
\end{gather*}
It we use the slightly different orthonormalized basis $(\varphi_n^\alpha)_{n\in\NN}$,
then the previous relations still hold true, except that all $G_j^\beta$ can be replaced by $G_j^\alpha$.
\end{Theorem}

Note that $\lambda_n^\alpha=0$ does not give any problem for these formulas of the $G$-factors.
If $\seq{\gamma}=\seq{\alpha}$ then all $\phi_n^\alpha$ are regular, but for
a general sequence $\seq{\gamma}$, it is possible that $\phi_n$ is not regular. Recall that then
$\lambda_n=\infty$ but by Theorem~\ref{thmcon} this means that $\lambda_n^\alpha=0$ and thus there is
no problem in using the $G$-matrices introduced above, even if the ORF sequence is not regular.

Another thing to note here is that, as we remarked before, the factors
${G}_k^\nu$ are unitary, although the factors $\hat{G}_k^\nu=\eta_{\A}^{-1}G_k^\nu\eta_{\A}$ are
in general not unitary.
Moreover
\begin{gather*}
\tilde{\zeta}_{\A}(\G)=\eta_{\A}^{-1}(\G+\A)(\I+\A^*\G)^{-1}\eta_{\A}
\end{gather*}
is unitary when $\G$ is unitary as can be directly verif\/ied in the f\/inite-dimensional case.
For example $[\tilde{\zeta}_{\A}(\G)][\tilde{\zeta}_{\A}(\G)]^*=\I$ if and only if
\begin{gather*}
(\I+\A\G^*)\eta_{\A}^{-2}(\I+\G\A)=(\G^*+\A^*)\eta_\A^{-2}(\G+\A).
\end{gather*}
Note that $\eta_\A$ is invertible but $\A$ is not (since for example $\alpha_0=0$) and that $\A$ and $\eta_\A$ commute.
To see that the relation holds, we multiply out and bring everything to one side, which gives
\begin{gather*}
\G^*\big(\eta_\A^{-2}\A-\A\eta_\A^{-2}\big)+\big(A^*\eta_\A^{-1}-\eta_\A^{-1}\A^*\big)
+\G^*\eta_\A^{-1}\G-\eta_\A^{-2}-\G^*\A\eta_\A^{-2}\A^*\G+\A^*\eta_\A^{-2}\A.
\end{gather*}
So we have to prove that this is zero.
Because $\eta_\A^{-2}\A=\A\eta_\A^{-2}$ and $\eta_\A^{-2}\A^*=\A^*\eta_\A^{-2}$, the f\/irst two terms vanish.
The four remaining terms can be multiplied from the left and the right with $\eta_\A=\eta_\A^*$ to give
\begin{gather*}
\eta_\A\G^*\eta_\A^{-2}\G\eta_\A -\I -\eta_\A\G^*\A\eta_\A^{-2}\A^*\G\eta_\A+\eta_\A\A^*\eta_\A^{-2}\A\eta_\A.
\end{gather*}
Because $\eta_\A\A^*=\A^*\eta_\A$ and $\A\eta_\A=\A\eta_\A$, the last term equals $\A^*\A$ and combined with the second term $-\I$ it equals $-\eta_\A^2$. Multiply everything again with $\eta_\A^{-1}$ from the left and from the right
gives
\begin{gather*}
\G^*\eta_\A^{-2}\G - \G^*\A\eta_\A^{-2}\A^*\G-\eta_\A^{-1}\eta_A^2\eta_\A^{-1}=
\G^*\eta_\A^{-2}\G - \G^*\A\eta_\A^{-2}\A^*\G-\I\\
\hphantom{\G^*\eta_\A^{-2}\G - \G^*\A\eta_\A^{-2}\A^*\G-\eta_\A^{-1}\eta_A^2\eta_\A^{-1}}{}=
(\G^*\eta_\A^{-1})(\I-\eta_\A\A\eta_\A^{-2}\A^*\eta_\A)(\eta_\A^{-1}\G)-\I\\
\hphantom{\G^*\eta_\A^{-2}\G - \G^*\A\eta_\A^{-2}\A^*\G-\eta_\A^{-1}\eta_A^2\eta_\A^{-1}}{}=
(\G^*\eta_\A^{-1})(\I-\A\eta_\A\eta_\A^{-2}\eta_\A\A^*)(\eta_\A^{-1}\G)-\I\\
\hphantom{\G^*\eta_\A^{-2}\G - \G^*\A\eta_\A^{-2}\A^*\G-\eta_\A^{-1}\eta_A^2\eta_\A^{-1}}{}=
(\G^*\eta_\A^{-1})(\I-\A\A^*)(\eta_\A^{-1}\G)-\I\\
\hphantom{\G^*\eta_\A^{-2}\G - \G^*\A\eta_\A^{-2}\A^*\G-\eta_\A^{-1}\eta_A^2\eta_\A^{-1}}{}=
\G^*\eta_\A^{-1}\eta_\A^2\eta_\A^{-1}\G-\I\\
\hphantom{\G^*\eta_\A^{-2}\G - \G^*\A\eta_\A^{-2}\A^*\G-\eta_\A^{-1}\eta_A^2\eta_\A^{-1}}{}=
\G^*\G-\I,
\end{gather*}
which is zero because $\G$ is unitary.
A similar calculation verif\/ies that $[\tilde{\zeta}_{\A}(\G)]^*[\tilde{\zeta}_{\A}(\G)]=\I$.
The same holds for the inf\/inite case, i.e., in the whole of $L_\mu^2$ if $\L$ is dense (see also \cite{MISCVel07}).

When using the recursion for the $\phi_n$ with respect to a general sequence $\seq{\gamma}$,
Theorem~\ref{thmrec} completed by Lemma~\ref{lemeta} says what the Szeg\H{o} parameters $\lambda_n$ are.
Note however that for the basis $\Phi$ corresponding to a general sequence $\seq{\gamma}$, the matrix $\G$
had factors $G_k^\alpha$ and $G_k^\beta$ that were all def\/ined in terms of $\lambda_k^\alpha$ and not $\lambda_k$
(see Theorem~\ref{thmSA1} and equation (\ref{eqab})).
Theorem~\ref{thmcon} explains that, depending on $\gamma_n$ and $\gamma_{n-1}$ an $\lambda_n$ is either equal to
$\lambda_n^\alpha$ or $\overline{\lambda_n^\alpha}$ or the inverses of these.
Thus it should be no problem to obtain all the $\lambda_n^\alpha$ from the $\lambda_n$ or vice versa.

To illustrate how to use the quantities for a general
$\seq{\gamma}$ sequence in generating the $G_n^\alpha$ matrix we consider the following example.
If $n-1\in\mathbbm{b}_n$ and $n\in\mathbbm{a}_n$, then $\gamma_n=\alpha_n$, $\gamma_{n-1}=\beta_{n-1}$, $\lambda_n^\alpha=1/\overline{\lambda}_n$,
$\phi_n^\alpha=\phi_n/\dB_n^\beta$, and $\phi_n^{\alpha*}=\phi_n^*/\dB_n^\beta$. This allows us to express $\eta_{n1}^\alpha$ and all the other elements of $G_n$ in terms of the $\seq{\gamma}$-related elements.
If we assume that $\phi_n$ is regular, then
\begin{gather*}
\tilde{G}_n^\alpha=\overline{\sigma}_{n-1}\overline{\eta}_{n1}^\alpha
\left[\begin{matrix}
-\eta_{n1}^\alpha/\overline{\lambda}_n & \sqrt{1-1/|\lambda_n|^2}\\
\sqrt{1-1/|\lambda_n|^2} & \overline{\eta}_{n1}^\alpha/{\lambda}_n
\end{matrix}\right]
\left[\begin{matrix}1&0\\0&\sigma_n
\end{matrix}\right]\\
\hphantom{\tilde{G}_n^\alpha}{}=\frac{\overline{\sigma}_{n-1}\overline{\eta}_{n1}^\alpha}{|\lambda_n|}
\left[\begin{matrix}u_n&0\\0&1\end{matrix}\right]
\left[\begin{matrix}
-1 & \sqrt{|\lambda_n|^2-1}\\
\sqrt{|\lambda_n|^2-1} & 1
\end{matrix}\right]
\left[\begin{matrix}1&0\\0&\overline{u}_n\sigma_n
\end{matrix}\right],
\end{gather*}
where $u_n=(\lambda_n/|\lambda_n|)\eta_{n1}^\alpha\in\TT$.

It still requires $\eta_{n1}^\alpha$ in terms of $\seq{\gamma}$-elements. This is a bit messy
but still possible using $\phi_n^{\alpha*}=\phi_n^*/\dB_n^\beta$ and Lemma~\ref{lemeta}.
The factor $1/\dB_n^\beta$ has an essential role here because we have to evaluate $\phi_n^{\alpha*}(z)$ for $z=\alpha_{n-1}$,
and if $1/\overline{\alpha}_{n-1}=\beta_{n-1}\in\{\gamma_k\colon k\in\mathbbm{b}_n\}$, which means that $\alpha_{n-1}$
 has been introduced as a pole in a previous step, then this pole in $\phi_n^*$ will cancel against the same pole in $\dB_n^\beta$.
Thus we should use a limiting process or expand $\phi_n^{\alpha*}=\phi_n^*/\dB_n^\beta$ or make the cancellation explicit as we did for example in equation
(\ref{eqstar}). We skip further details.

If $n$ is not a regular index, then $\lambda_n=\infty$ and in that case the matrix $\tilde{G}_n^\alpha$ has the form
\begin{gather*}
\tilde{G}_n^\alpha=\overline{\sigma}_{n-1}\overline{\eta}_{n1}^\alpha
\left[\begin{matrix}
0 & 1\\
1 & 0
\end{matrix}\right]
\left[\begin{matrix}1&0\\0&\sigma_n
\end{matrix}\right]
=\overline{\sigma}_{n-1}\overline{\eta}_{n1}^\alpha
\left[\begin{matrix}
0 & \sigma_n\\
1 & 0
\end{matrix}\right].
\end{gather*}

\section{Spectral analysis}\label{secSA}

The full spectral analysis of the matrices $\tilde{\zeta}_\A(\G)$ and how this relates to the shift operator in $L^2_\mu$
is given in \cite{MISCVel07} for the case $\seq{\gamma}=\seq{\alpha}$ (then $\G=\cH$ is a Hessenberg matrix)
and in the case
of the alternating $\seq{\gamma}=\seq{\varepsilon}$ where either $\varepsilon_{2k}=\alpha_{2k}$ and $\varepsilon_{2k+1}=\beta_{2k+1}$ or $\varepsilon_{2k}=\beta_{2k}$ and $\varepsilon_{2k+1}=\alpha_{2k+1}$
(then $\G=\C$ is a CMV matrix).

Vel\'azquez shows for the cases $\nu\in\{\alpha,\varepsilon\}$ that if $\T_\mu$ is the shift operator on $L^2_\mu$,
(i.e., multiplication with $z$ in $L^2_\mu$)
and if $\L$ is the closure of $\L_\infty$ in $L^2_\mu$,
then the matrix representation of~$T_\mu\upharpoonright \L$ with respect to the basis $\Phi^\alpha=[\varphi_0^\alpha,\varphi_1^\alpha,\varphi_2^\alpha,\ldots]$ is given by $\tilde{\zeta}_\A(\G^\alpha)$
(see \cite[Theorems~4.1, 4.2, 5.1, 5.4]{MISCVel07}).

We mention here the basis $(\varphi_k^\alpha)_{k\in\NN}$ because in \cite{MISCVel07} only one type of
$G_k^\alpha$ matrices is used. As we have explained, this is related to the precise normalization used for the basis functions, but as long
as the basis is orthogonal and has norm 1, a further normalization with a unimodular constant will not inf\/luence the spectral properties of the matrix representations so that similar results hold for the basis $\Phi$ and the matrix
$\tilde{\zeta}_\A(\G)$.

If moreover $\L$ is dense in $L_\mu^2$ and hence $\L=L_\mu^2$, then $\tilde{\zeta}_\A(\G)$ is not only isometric
but it is a~representation of the full matrix.
For $\set{\alpha}$ compactly included in $\DD$
this happens for $\seq{\gamma}=\seq{\alpha}$ when $\log\mu'\not\in L^1$ and in the case $\seq{\gamma}=\seq{\varepsilon}$
when $\sum\limits_{k=1}^\infty(1-|\alpha_{2k}|)=\infty=\sum\limits_{k=1}^\infty(1-|\alpha_{2k+1}|)$
(see \cite[Theorems~4.1, Proposition~5.3]{MISCVel07}).

In fact, the proof of \cite[Proposition~5.3]{MISCVel07} needs only minor adaptations to be applicable to the present situation. So we can state without further details:
\begin{Theorem}\label{thmbasis2}
If the sequence $\seq{\gamma}$ is such that $\sum\limits_{k\in\mathbbm{a}_\infty}(1-|\alpha_k|)=\infty$ and
$\sum\limits_{k\in\mathbbm{b}_\infty}(1-|\alpha_k|)=\infty$, then both $(\phi_k)_{k\in\mathbbm{a}_\infty}$ and
$(\phi_k)_{k\in\mathbbm{b}_\infty}$ are bases of $L^2_\mu$.
\end{Theorem}
Now also \cite[Theorem~5.4]{MISCVel07} can be directly translated to our situation:
\begin{Theorem}\label{thmspec}
For the general sequence $\seq{\gamma}$ as in the previous theorem,
$\U=\tilde{\zeta}_\A({\G})$ is the matrix representation of the shift operator $\T_\mu$ of $L^2_\mu$ and
$\sigma(\U)=\supp\mu$. Moreover, if $\Phi=[\phi_0,\phi_1,\ldots]$ then $\xi$ is a mass point of $\mu$ if and only if
$\|\Phi(\xi)\|_{\ell^2}<\infty$, $\mu(\{\xi\})=1/\|\Phi(\xi)\|_{\ell^2}^2$ and $\Phi(\xi)$ is a~$($left$)$ eigenvector for the eigenvalue $\xi$ of $\U$.
\end{Theorem}

In \cite{MISCVel07} it is also explained how the spectrum of $\tilde{\zeta}_\A({\G})=\eta_\A^{-1}\tilde{\varpi}^*_\A(\G)\tilde{\varpi}_\A(\G)^{-1}\eta_\A$ relates to
the spectrum of the pencil $(\varpi^*_\A(\G),\varpi_\A(\G))$.
In the case of an alternating sequence $\seq{\varepsilon}$, the matrix $\G$ was a CMV matrix which we denoted as $\C$.
It was then possible to regroup the even and odd elementary factors so that $\C=\C_o^\alpha\C_e^\beta$ or
$\C=\C_e^\alpha\C_o^\beta$ with $\C_o^\nu=G_1^\nu G_3^\nu G_5^\nu\cdots$ and $\C_e=G_2^\nu G_4^\nu G_6^\nu\cdots$ for $\nu\in\{\alpha,\beta\}$. As a consequence, because all factors are unitary, the previous result could be formulated in terms of the spectral
properties for the pencils $(\A\C_e^{\beta*}+\C_o^\alpha,\C_e^{\beta*}+\A^*\C_o^\alpha)$ or $(\A\C_o^{\beta*}+\C_e^\alpha,\C_o^{\alpha*}+\A^*\C_e^\alpha)$.
In the general case, such a regrouping is possible, but rather complicated because it will
depend on the sizes of the $\alpha$ and $\beta$ blocks in the $\gamma$ sequence, i.e., on the number of elementary $G$-factors that appears in the product in increasing or decreasing order
(see (\ref{eqGblocks})).

There are also spectral interpretations for the eigenvalues of the f\/initely truncated sections of the matrices.
Again, a careful discussion is given in \cite{MISCVel07} for this. If for an operator $\T$ we denote its truncation to
$\L_n$ as $\T_n=\mathcal{P}_n\T\upharpoonright \L_n$ with $\mathcal{P}_n$ the orthogonal projection operator on~$\L_n$,
then Vel\'azquez shows that the zeros of the $\phi_{n+1}$ are the eigenvalues of the f\/inite matrices
$\V_n=\tilde{\zeta}_{\A_n}(\G_n)$
with $\G_n$ the leading principle submatrix of size $(n+1)$ of the matrix $\G$ and $\A_n=\diag(\alpha_0,\ldots,\alpha_n)$.
As illustrated in
\cite[equation~(19)]{MISCVel07}, $\G_n$ is the product of unitary matrices $G_k$ of the form (\ref{eqG0}) for $k=1,\ldots,n+1$ except that $G_{n+1}$ is `truncated' and is of the form
$\diag(I_n\oplus -\delta_{n+1})$ where $\delta_{n+1}$ (up to a unimodular constant) is $\lambda_{n+1}$ and thus
this truncated factor $G_{n+1}$ is not unitary because $\lambda_{n+1}\not\in\TT$ in general.
For $\seq{\gamma}=\seq{\alpha}$ this is proved in \cite[Theorem~4.5]{MISCVel07}
and for $\seq{\gamma}=\seq{\varepsilon}$ in \cite[Theorem~5.8]{MISCVel07}.

The same arguments can be used for a general sequence $\seq{\gamma}$.
Indeed (\ref{eq71}) and (\ref{eq72}) show that in the truncated version we shall end up with
\begin{gather*}
[\phi_0,\ldots,\phi_n](\varpi_{\A_{n}}^*-\varpi_{\A_{n}}\G_n)=[0,\ldots,0,\chi_{n+1}],
\end{gather*}
where $\chi_{n+1}$ is of the form $c_{n+1}\phi^\alpha_{n+1}\dB_{n+1}^\beta$ or $c_{n+1}\phi^{\alpha*}_{n+1}\dB_{n+1}^\beta$
depending on whether the truncation falls inside and $\alpha$ block or a $\beta$ block.
In both cases, this is a constant multiple of $\phi_{n+1}$ (see Theorem~\ref{thm1}).
Like in \cite{MISCVel07} this can be transformed into
\begin{gather*}
[\phi_0(z),\ldots,\phi_n(z)]\big(z-\tilde{\zeta}_{\A_{n}}(\G_n)\big)=[0,\ldots,0,c_{n+1}\phi_{n+1}(z)],
\end{gather*}
which shows that a zero of $\phi_{n+1}$ corresponds to an eigenvalue of $\V_n=\tilde{\zeta}_{\A_{n}}(\G_n)$.
Thus the following theorem is still valid for a general $\seq{\gamma}$.
\begin{Theorem}\label{thmzerosphi}
For the general sequence $\seq{\gamma}$ let $\G_n$ be the truncation of size $n+1$ of the matrix~$\G$,
and define $\V_n=\tilde{\zeta}_{\A_n}(\G_n)$, then the eigenvalues of $\V_n$ will be the zeros of
$\phi_{n+1}$.
The $($left$)$ eigenvector corresponding to an eigenvalue $\xi$ is given by
$\Phi_n(\xi)=[\phi_0(\xi),\ldots,\phi_n(\xi)]$.
\end{Theorem}

For the unitary truncation, we f\/irst form a unitary truncation $\G_{n}^u$ of $\G_{n}$
by keeping all $G_k$ with $k<n+1$ and replacing $G_{n+1}$ by $G_{n+1}^u=\diag(1,\ldots,1,\tau)$
for some $\tau\in\TT$. We shall then get (recall the def\/inition of the PORF $Q_{n+1}(z,\tau)$
in (\ref{defPORF}))
\begin{gather*}
[\phi_0,\ldots,\phi_n]\big(\varpi_{\A_{n}}^*-\varpi_{\A_{n}}\G_n^u\big)=[0,\ldots,0,X_{n+1}],
\end{gather*}
with $X_{n+1}$ of the form $c_{n+1}Q_{n+1}^\alpha(z,{\tau})\dB_{n+1}^\beta(z)$
or $c_{n+1}Q_{n+1}^\alpha(z,\overline{\tau})\dB_{n+1}^\beta(z)$
depending on whether the truncation falls inside and $\alpha$ block or a $\beta$ block.
By Corollary~\ref{cor2knd} this is proportional to~$Q_{n+1}(z,\tau')$ in both cases
for some appropriate choice of $\tau'$.
Therefore eigenvalues of $\U_n=\tilde{\zeta}_{\A_n}(\G_n^u)$ will give the nodes of the
$(n+1)$-point rational Szeg\H{o} quadrature formula for some particular choice of the
$\tau$-parameter.

Again this is worked out in detail in \cite{MISCVel07} for the Hessenberg and CMV matrices.
As in the previous theorem, also for the unitary truncations we can use similar arguments for a
general sequence and we therefore have also
\begin{Theorem}\label{thmzerosPORF}
For the general sequence $\seq{\gamma}$ let $\G_n^u$ be the unitary truncation of size $n+1$
of the matrix~$\G$,
and define $\U_n=\tilde{\zeta}_{\A_n}(\G_n^u)$, then the eigenvalues of $\U_n$ will be the zeros of
the PORF~$Q_{n+1}$ for some value of the $\tau$-parameter.
The $($left$)$ eigenvector corresponding to an eigenvalue~$\xi$ is given by
$\Phi_n(\xi)=[\phi_0(\xi),\ldots,\phi_n(\xi)]$.
\end{Theorem}
\begin{Remark}\label{renunitary}
Since $\U_n$ is unitary, its right eigenvectors will be the conjugate transpose of the left
eigenvectors.
Thus if $F$ is the matrix whose entry at position $(i,j)$ is $\overline{\phi_j(\xi_i)}$,
and $\Lambda=\diag(\xi_0,\ldots,\xi_n)$ contains the eigenvalues of $\U_n$,
then $F\U_n = \Lambda F$.
A row of $F$ is $[\phi_0(\xi_i),\ldots,\phi_n(\xi_i)]$ which is not normalized in $\ell^2$ sense, but
we know that $\phi_0=1$.
In linear algebra it is more usual to write the eigenvalue decomposition of a unitary (or normal)
matrix $\U_n$ as $U^*\U_n=\Lambda U^*$ with~$U$ unitary, thus with orthonormal rows and columns.
Thus if $D$ is the diagonal matrix whose diagonal entries are the elements on
the f\/irst row of $U$ (which in our case are all nozero) then $F=D^{-1}U^*$.
So $FF^*=D^{-1}U^* U D^{-*}=D^{-1}D^{-*}$ is diagonal with
entries $\sum\limits_{k=0}^n |\phi_k(\xi_i)|^2$.
In other words, the weights $w_i$ of the quadrature formula are given by the modulus
square of the elements in the f\/irst row of $U$.
\end{Remark}

Note that the full matrices never need to be constructed explicitly. There exist ef\/f\/icient algorithms to
perform all the required operations when the matrix is stored as a product of elementary $2\times2$ factors
(for example \cite{MISCAMVW15} or \cite{MISCBD07}).
In fact these techniques will be able to handle an arbitrary order in the product with the same
complexity (in terms of f\/loating point operations) as for example the Hessenberg or CMV ordering.
It will however reduce the programming complexity slightly requiring less bookkeeping,
if we use for example a CMV structure.
In the case of the CMV matrix, we do not even need a product at all to explicitly write down
the matrices involved.
Indeed, since the odd and even factors $\C_{on}^u$ and $\C_{en}^u$ of the unitarily truncated CMV matrix $\C_n^u$
are just block diagonals with all blocks of size at most $2\times2$ with exception of $G_{n+1}$
which is diagonal.
Hence a pencil like $(\A_n\C_{en}^{u*}+\C_{on}^{u},\C_{en}^{u*}+\A_n^*\C_{en}^u)$ that will allow to compute
the quadrature is extremely simple for numerical computations \cite[Theorem~7.3]{ArtBCCB14}.

We mentioned already in Section~\ref{secQF} that an arbitrary sequence $\seq{\gamma}$ did not give rise to new
quadrature formulas than already given by the sequence $\seq{\alpha}$ or the alternating sequence $\seq{\varepsilon}$.
We now f\/ind that there is no computational advantage either
it is not very rewarding in a practical computation of rational Szeg\H{o}
quadrature to work out all the details in the general case.
The CMV matrix obviously has the smallest possible bandwidth; it is conceptually simpler and somewhat easier to implement.

There is however a remarkable linear algebra side ef\/fect that
became clear from the previous analysis.
As far as we know, this was previously unknown in the linear algebra community and we will discussed
some preliminary results in the next section.

\section{AMPD matrices}\label{secAMPD}

In the previous section, it was recalled from \cite{MISCVel07} that under the appropriate conditions, the spectrum
of the unitary shift $\T_\mu$ can be recovered from the matrix representation $\tilde{\zeta}_{\A}(\G^\alpha)$
 with respect to the basis $[\varphi_0^\alpha,\varphi_1^\alpha,\ldots]$, whatever the sequence
$\seq{\gamma}$ is. Now the matrix $\G^\alpha$ is a product of elementary matrices that dif\/fer from the identity
only by one $2\times2$ block of the general form~(\ref{eqG0}). The only dif\/ference between the Hessenberg and the
CMV form is the order in which these factors are multiplied.

Something of that form was known for a f\/inite unitary Hessenberg matrix.
It can be decomposed as a product of {\it unitary} matrices of the type (\ref{eqG0}) and
multiplying these factors in any order will not change the spectrum of the matrix
(see, e.g., \cite{MISCAMVW15}).

Since the zeros of the ORFs are given as the eigenvalues of the matrix $\tilde{\zeta}_{\A_n}(\G_n)$ where $\G_n$ is
either a f\/inite Hessenberg or a CMV matrix and hence they dif\/fer only by the order of the elementary factors,
we might expect that if $\G_n$ is any matrix written as the product of $n$ elementary factors of the form
(\ref{eqG0}), then the spectrum of $\tilde{\zeta}_{\A_n}(\G_n)$ will not change if we change the order of the
elementary factors in the product.
This is a consequence of the fact that the spectrum of $\G_n$ will not depend on the order of its elementary factors.

Since this section is partly detached from the ORF theory, the notation used is also local and should not
be confused with the same notation used in other sections, unless stated explicitly.

We def\/ine an AMPD matrix as a matrix of the form $AM+D$
where $A$ and $D$ are diagonal matrices and $M$ is the product of $G$-matrices
that are matrices that dif\/fer from the identity by one
$2\times2$ diagonal block like ($I_0$ is the empty matrix)
\begin{gather}\label{eqG1}
G_{k}:=\left[
\begin{matrix}
I_{k-1} &0 &0 & 0 \\ 0 & \alpha_k & \beta_k & 0 \\ 0 & \gamma_k & \delta_k & 0\\ 0 & 0 & 0 & I_{n-k}
\end{matrix}\right],
\qquad 1\leq k \le n,
\end{gather}
which is more general than the matrices def\/ined in (\ref{eqG0}).

The purpose is f\/irst to show that the spectrum of an AMPD matrix does not
depend on the order in which the $G$-matrices are multiplied.
The next proof was provided by Marc Van Barel in private discussion with the f\/irst author.

\begin{Lemma}\label{lemAMPD2}
Suppose $M'\in\CC^{n\times n}$ is an arbitrary matrix with $n\ge2$ and set
\begin{gather*}
M=\left[\begin{matrix} M'& 0 \\ 0 & 1\end{matrix}\right]
\qquad\text{and}\qquad
G=\left[\begin{matrix}
I_{n-1} & 0 & 0\\
0 & \alpha & \beta \\
0 & \gamma & \delta
\end{matrix}\right]
\end{gather*}
and let $A,D\in\CC^{(n+1)\times(n+1)}$ be two arbitrary diagonal matrices.
Then the determinant and the eigenvalues of $AGM+D$ and $AMG+D$
are the same.
\end{Lemma}
\begin{proof}
Let us write
\begin{gather*}
M'=
\left[\begin{array}{@{}c|c@{}}
 & \\
~~~M''~~~& \mathbf{c} \\
 & \\ \hline
 \mathbf{r} & m
\end{array}\right]
\qquad \text{and}\qquad
A=\diag(A'',a',a), \qquad
D=\diag(D'',d',d).
\end{gather*}
Then
\begin{gather*}
AG M +D = \left[\begin{array}{@{}c|cc@{}}
 & & \\
A''M''+D'' & A''\mathbf{c} & 0 \\
 & & \\ \hline
a'\alpha\mathbf{r} & a'\alpha m +d' & a'\beta \\
a\gamma\mathbf{r} & a\gamma m & a\delta +d
\end{array}\right]
\end{gather*}
and
\begin{gather*}
AMG+D = \left[\begin{array}{@{}c|cc@{}}
 & & \\
A''M''+D'' & \alpha A''\mathbf{c} & \beta A''\mathbf{c} \\
 & & \\\hline
a'\mathbf{r} & a'\alpha m + d' & a'\beta m \\
0 & a\gamma & a\delta+d
\end{array}\right].
\end{gather*}
Taking the determinant of $AGM+D$ gives (expand along the last column)
\begin{gather*}
\det(AGM+D) =
(a\delta+d)a'\alpha\det\tilde{M}
+ (a\delta+d) d' \det (A''M''+D'')
-a'\beta a\gamma\det\tilde{M}
\\
\hphantom{\det(AGM+D)}{}=
aa'(\alpha\delta-\beta\gamma)\det\tilde{M}
+a'd\alpha\det \tilde{M}
+ (a\delta+d) d' \det (A''M''+D'') \\
\hphantom{\det(AGM+D)}{}= [aa'\det(G)+a'd\alpha]\det\tilde{M} + (a\delta+d) d' \det (A''M''+D'')
\end{gather*}
with $\tilde{M}$ a matrix independent of $G$, given by
\begin{gather*}
\tilde{M} = \left[\begin{array}{@{}c|c@{}} & \\ A''M''+D'' & A''\mathbf{c} \\ & \\\hline \mathbf{r} & m \end{array}\right]
= \diag(A'',1)M'+\diag(D'',0).
\end{gather*}
Evaluating the determinant of $AMG+D$ in a similar way, expanding along the last row, gives exactly the same result. Hence the determinants are the same:
$\det(AGM+D)=\det(AMG+D)$.

Note that $D-\lambda I$ is still a diagonal if $D$ is, so that if we replace
$D$ by $D-\lambda I$ in the determinants, we obtain the characteristic
polynomials of the matrices $AGM+D$ and \smash{$AMG+D$} respectively.
Equality of these polynomials implies the
equality of their roots, i.e., of the eigenvalues.
\end{proof}

In fact the result is more general since we could replace $D$ as well as $A$ by any diagonal matrices
that are functions of $\lambda$ and the determinants would still be the same.

It is interesting to note that $A''M''+D''$ is the same as the matrix $\tilde{M}$
in which the last row and the last column are deleted.
This will be used in Theorem~\ref{thmAMPD3} below.

\begin{Example}\label{ex1}
If in the previous lemma we set
\begin{gather*}
M=G_1=\left[\begin{matrix}
\alpha_1 & \beta_1 & 0 \\
\gamma_1 & \delta_1 & 0 \\
0 & 0 & 1
\end{matrix}\right]
\qquad \text{and}\qquad
G=G_2=\left[\begin{matrix}
1 & 0 & 0 \\
0 & \alpha_2 & \beta_2 \\
0 & \gamma_2 & \delta_2
\end{matrix}\right],
\end{gather*}
then for arbitrary diagonal matrices $A=\diag(a_1,a_2,a_3)$
and $D=\diag(d_1,d_2,d_3)$ a direct evaluation of the matrices
$AG_1G_2$ and $AG_2G_1$ gives
\begin{gather*}
AG_1G_2=\left[\begin{matrix}
a_1\alpha_1 & a_1\beta_1\alpha_2 & a_1\beta_1\beta_2\\
a_2\gamma_1 & a_2\delta_1\alpha_2 & a_2\delta_1\beta_2 \\
0 & a_3\gamma_2 & a_3\delta_2
\end{matrix}\right]
\qquad \text{and}\qquad
AG_2G_1=\left[\begin{matrix}
a_1\alpha_1 & a_1\beta_1 & 0\\
a_2\alpha_2\gamma_1 & a_2\alpha_2\delta_1 & a_2\beta_2 \\
a_3\gamma_2\gamma_1 & a_3\gamma_2\delta_1 & a_3\delta_2
\end{matrix}\right].
\end{gather*}
Note that the diagonal elements are the same. If we add $D$ to this diagonal, then
with the Laplace rule it is immediately seen that $\det(AG_1G_2+D)=\det(AG_2G_1+D)$.
\end{Example}

If $\pi=(\pi_1,\pi_2,\ldots,\pi_k)$ is a permutation
of the numbers $(1,2,\ldots,k)$, then with a~notation like $\prod\limits_{i\in\pi} G_i$ where $G_i\in\CC^{(n+1)\times (n+1)}$ are square matrices, we mean the product from left
to right in the order given by $\pi$. Thus
\begin{gather}\label{eqprodpi}
\prod_{i\in\pi} G_i = G_{\pi_1} G_{\pi_2} \cdots G_{\pi_k}.
\end{gather}

Recall that $G_i$ and $G_j$ of the form (\ref{eqG1}) commute whenever $|i-j|>1$.
This means that of the $k!$ possible permutations of the numbers $(1,2,\ldots,k)$, only $2^{k-1}$ dif\/ferent matrices will result when we consider all possible
permutations in the product $\prod\limits_{i\in\pi}G_i$.
Indeed, once the order is f\/ixed in which the f\/irst $k-1$ matrices
$\{G_i\colon i=1,2,\ldots,k-1\}$ are multiplied,
$G_k$ will have to multiply this product either from the right or from
the left depending on whether it is to the left or to the right of $G_{k-1}$
in the product (\ref{eqprodpi}).
This is because $G_k$ commutes with all previous~$G_i$, with $i<k-1$.

Let $P_n^N=[I_n~~0]\in\mathbb{C}^{n\times N}$ be the matrix projecting onto $\mathbb{C}^n$.
We def\/ine for any matrix $M\in\mathbb{C}^{N\times N}$ its truncation to $\mathbb{C}^{n\times n}$ as $P_n^N MP_n^{N*}$.
The following main theorem states that the determinant and hence the eigenvalues of the truncation
of an AMPD matrix does not depend on the order in which
its $G_i$ factors are multiplied.

\begin{Theorem}\label{thmAMPD3}
Let $\pi=(\pi_1,\pi_2,\ldots,\pi_{n})$ be an arbitrary permutation of $(1,2,\ldots,n)$ and
define
$M_\pi=\prod\limits_{i\in\pi}G_i$ with $G_i$ of the form \eqref{eqG1},
let $A$ and $D$ both be arbitrary diagonal matrices of size $n+1$ and finally set
$P=P_n^{n+1}$.
Then $\det(P(AM_\pi+D)P^*)$ will not depend on $\pi$.
\end{Theorem}
\begin{proof}
The proof is by induction.
For $n=2$ we need to consider $P(AG_1G_2+D)P^*$ and $P(AG_2G_1+D)P^*$.
It can be checked with the expressions of $AG_1G_2$ and $AG_2G_1$ from
Example~\ref{ex1} that
\begin{gather*}
\det(P(AG_1G_2+D)P^*) =
\det\left[\begin{matrix}a_1\alpha_1+d_1 & a_1\alpha_2\beta_1\\
a_2\gamma_1& a_2\alpha_2\delta_1+d_2\end{matrix}\right]
\end{gather*}
and
\begin{gather*}
\det(P(AG_2G_1+D)P^*)
=\det\left[\begin{matrix}a_1\alpha_1+d_1 & a_1\beta_1\\
a_2\alpha_2\gamma_1& a_2\alpha_2\delta_1+d_2\end{matrix}\right].
\end{gather*}
Both expressions give the same determinant.

For the induction step we recall that $G_n$ commutes with all $G_k$ with $k<n-1$
thus $\prod\limits_{i\in\pi}G_i$
equals either $G_n\big(\prod\limits_{i\in\pi'}G_i\big)$ or $\big(\prod\limits_{i\in\pi'}G_i\big) G_n$
where $\pi'$ is the same as $\pi$ after removing the value $n$.

In the f\/irst case with $G_n$ at the beginning (see the proof of Lemma~\ref{lemAMPD2})
\begin{gather*}
P(M_\pi)P^*=
P\left[\begin{array}{@{}c|cc@{}}
&& \\ M'' & \mathbf{c}_{n-1} &0 \\ & & \\\hline \alpha_n\mathbf{r}_{n-1} & \alpha_n m_{n-1} & \beta_n \\
\gamma_n\mathbf{r}_{n-1} & \gamma_n m_{n-1} & \delta_n
\end{array}\right] P^*
=
P\left[\begin{array}{@{}c|cc@{}}
& & \\ M'' & \mathbf{c}_{n-1} & 0 \\ & & \\\hline \alpha_n\mathbf{r}_{n-1} & \alpha_n m_{n-1} & 0 \\
0&0 & 1
\end{array}\right]P^*.
\end{gather*}
Therefore $\det [P(AM_\pi +D)P^*]=\det(\hat{A}\alpha_n\hat{M}_{\pi'} +\hat{D})$
where we introduced for any matrix $F$ the hat-notation to mean $\hat{F}=PFP^*$.
Note that $\hat{A}\alpha_n\hat{M}_{\pi'} +\hat{D}$ is an AMPD matrix of size $n$ which has only $n-1$ factors
($\alpha_n$ can be assimilated in $\hat{A}$).
By induction it is independent of the permutation $\pi'$.

The second case, when $G_n$ is at the end of the product, is completely similar.
The $\alpha_n$ in front of $\mathbf{r}_{n-1}$ will move to $\mathbf{c}_{n-1}$
and we f\/ind the same expression. This implies that the determinant is not only independent of $\pi'$
but that we can also insert $n$ anywhere in $\pi'$ and still get the same result, which proves the induction step.
\end{proof}

We have as an immediate consequence that the truncation is not essential.

\begin{Corollary}\label{corAMPD4}
The determinant, and hence also the eigenvalues of an AMPD matrix $AM_\pi+D$ do not depend on the
permutation $\pi$.
\end{Corollary}
\begin{proof}
To prove this for an AMPD matrix of size $n$ with $n-1$ factors, consider the AMPD matrix of the previous
theorem and assume that $G_n$ is just the identity matrix. It is then clear that the truncated AMPD matrix
$\hat{A}\hat{M}_\pi+\hat{D}$ is equal to an AMPD matrix of size $n$ with $\hat{M}_\pi=P M_{\pi'} P^*$
with $\pi'$ equal to $\pi$ in which $n$ is removed.
By the previous theorem, its determinant is independent of the permutation $\pi'$.
\end{proof}
As we mentioned earlier, we could have replaced $A$ and $D$ by any two diagonal matrices
that are functions of $\lambda$ and the determinants would still be independent of the permutation $\pi$.

RAMPD matrices are rational AMPD matrices.
They are the product of an AMPD matrix and an inverse of an AMPD matrix.
For example
\begin{gather*}
R= (AM+C)(BM+D)^{-1}
\end{gather*}
with $A$, $B$, $C$, $D$ all diagonal and $M$ a product of $G$-factors.
Again it can be shown that the spectrum of $R$ is independent of
the product order in the matrix $M$.

\begin{Theorem}\label{thmRAMPD}
Let $R=(AM_\pi+C)(BM_\pi+D)^{-1}$ be a RAMPD matrix as defined above, with $M_\pi=\prod\limits_{k\in\pi}G_k$, then the eigenvalues of $R$ do
not depend on the permutation $\pi$.
If $BM_\pi+D$ is not invertible, then a slightly more general formulation is that the pencil
$(AM_\pi+C,BM_\pi+D)$ has eigenvalues that do not depend on $\pi$.
\end{Theorem}
\begin{proof}
This problem can be conveniently solved by reducing it to the AMPD case.
Indeed, the eigenvalues of the pencil can be found by f\/inding the solutions of the
characteristic polynomial which is the determinant of the matrix
\begin{gather*}
(AM_\pi+C)-(BM_\pi+D)\lambda.
\end{gather*}
It is obvious that this can be rewritten as
\begin{gather*}
(A-B\lambda)M_\pi+(C-D\lambda)={A'}M_\pi+{D'}
\end{gather*}
in which ${A'}=A-B\lambda$ and ${D'}=C-D\lambda$ are diagonal matrices.
By what has been proved above, the determinant of the AMPD matrix ${A'}M_\pi+{D'}$
does not depend on the permutation $\pi$.
\end{proof}

In the context of ORF, the $G_k$ matrices were unitary, thus of the general form (\ref{eqG0}).
Note that then $M_\pi=\prod\limits_{k\in\pi}G_k$ and its unitary truncation will be unitary.
Conversely any unitary matrix can be written as a unitary truncation of such a product because
by a sequence of Givens transformations it can be brought into a~unitary Hessenberg form, which can be written
as a~product of such factors. All these elementary transforms can be merged into a unitary truncation of a~product of $n+1$ factors $G_k$ if the unitary matrix has size $n+1$.

The results about the zeros of the para-orthogonal functions suggested that the eigenvalues of the RAMPD matrix
\begin{gather}\label{eqRAMPDU}
R_\pi=\eta_A^{-1}(A+M_\pi)(A^* M_\pi+I)^{-1}\eta_A,\qquad \eta_A=(I-A^* A)^{1/2}
\end{gather}
are independent of the permutation $\pi$, and moreover, if $U$ is the unitary matrix of eigenvectors,
then the squared absolute values of the entries on the f\/irst row of $U$ are also independent of $\pi$.
This is because these values give the weights of the quadrature whose nodes are
the corresponding eigenvalues
and since these quadratures are uniquely def\/ined, the weights computed must be the same.

With the ORF interpretation of the previous section we can be more precise and prove that the
following theorem holds.
\begin{Theorem}\label{thmRAMPDu}
Let $\pi$ be a permutation of $(1,\ldots,n)$ and
$M_\pi=\prod\limits_{k\in\pi}G_k$ be the product of unitary factors $G_k$ of the form \eqref{eqG0} with all $\delta_k\in\DD$.
Let $A$ be a diagonal matrix of size $n+1$ with all entries in $\DD$.
Let $R_\pi$ be the unitary matrix defined by \eqref{eqRAMPDU} with eigenvalue decomposition
$R_\pi=V_\pi\Lambda_\pi V_\pi^*$ with $V_\pi$ the unitary matrix whose columns are the eigenvectors.
Then $\Lambda_\pi$ and the absolute values of all the entries in $V_\pi$ are independent of $\pi$.
\end{Theorem}
\begin{proof}
First note that it is suf\/f\/icient to prove this when $M_\pi$ is a unitary truncation of a product of
$n+1$ unitary factors. As in Corollary~\ref{corAMPD4} it will then also hold for an untruncated
product of $n$ unitary factors by choosing factor $G_{n+1}$ to be the identity.

We shall work with only one type of $G_k$ factors (those of type $G_k^\alpha$).
Using the results of the previous section we can indeed replace all the $G_k^\beta$ by $G_k^\alpha$
if we use the basis $\{\varphi_k^\alpha=\dvarsigma_k^\beta\phi_k\colon k=0,1,2,\ldots\}$
where $\dvarsigma_k^\beta=\prod\limits_{j\in\mathbbm{b}_k}\sigma_j\colon k=0,1,2,\ldots$.
Furthermore we know that row $i$ of~$V_\pi$ is proportional to
$[\varphi_0^\alpha(\xi_i),\ldots,\varphi^\alpha_n(\xi_i)]$ where $\Lambda_\pi=\diag(\xi_0,\ldots,\xi_n)$ are the
eigenvalues of~$R_\pi$.
We also know that reordering the entries in~$\pi$ will not change the eigenvalues,
but it does change the eigenvectors.
Assume for example that $\pi=(1,2,\ldots,n)$, then all $\varphi^\alpha_k=\dvarsigma_k^\beta\phi_k^\alpha=\phi_k^\alpha$, $k=1,\ldots,n$.
If we change $\pi$, then some of the $\alpha_k\in\mathbbm{a}_{n}$ will be
replaced by $\beta_k=1/\oa_k$.
For the $k$ that remain in $\mathbbm{a}_{n}$,
we have $\varphi_k^\alpha=\dvarsigma_k^\beta\phi_n=\dvarsigma_k^\beta\phi_k^\alpha \dB_k^\beta$.
The entry $\varphi_k^\alpha(\xi_i)=\dvarsigma_n^\beta\phi^\alpha_k(\xi_i)$ will thus change into
$\varphi_k^\alpha(\xi_i) \dB_k^\beta(\xi_i)$, which means a multiplication with a unimodular constant.
For other indices $k$ that migrate to $\mathbbm{b}_{n}$, the $\varphi^\alpha_k$ becomes
$\varphi^\alpha_k=\dvarsigma_k^\beta\phi_n={\dvarsigma_k^\beta}\phi_k^{\alpha*} \dB_k^\beta$ (see Theorem~\ref{thm1}).
But when evaluated in $\xi_i\in\TT$, clearly $\dB_k^\beta(\xi_i)\in\TT$, while
also $|\phi_k^\alpha(\xi_i)|=|\phi_k^{\alpha*}(\xi_i)|$ because
$\phi_k^{\alpha*}(\xi_i)=\dB_k^\alpha(\xi_i)\dB_k^\beta(\xi_i)\phi_{k*}^\alpha(\xi_i)= \dB_k^\alpha(\xi_i)\dB_k^\beta(\xi_i)\overline{\phi_k^\alpha(\xi_i)}$.
This proves the theorem.
\end{proof}
\begin{Remark}
The condition $|\delta_k|<1$ in the previous theorem is important because $\delta_k\in\TT$ can result in a
matrix $M_\pi$ that is the direct sum of diagonal blocks. Suppose as an example $\pi=(1,2,\ldots,n)$
so that $M_\pi$ is Hessenberg, but if $\delta_n\in\TT$, then the last column of $M_\pi$ will be
$[0,0,\ldots,0,{\delta}_k]^*$ and this will give an eigenvector $[0,\ldots,0,1]^*$ and the f\/irst entry
is obviously zero, so that the link with $\varphi^\alpha_0=1$ cannot be made.
In general, diagonal $G_k$ factors result in a reducible $\M_\pi$, which means that it is a direct
sum of submatrices. This corresponds to a~breakdown of the recurrence relation of the ORFs.
\end{Remark}

Even in the case of orthogonal Laurent polynomials, that is when the diagonal matrix $A=0$,
the ref\/inement of Theorem~\ref{thmRAMPDu} was not observed in \cite{MISCBD07}.
Without our ORF interpretation, a~proof purely on linear algebra arguments seems to be dif\/f\/icult.

\section{Relation with direct and inverse eigenvalue problems}\label{secAF}
We are convinced that the generalization of the ORFs that was developed in the previous sections
can be ef\/f\/iciently used to analyse algorithms that are developed and used to solve
direct and inverse eigenvalue problems.
A direct problem is: given a matrix, f\/ind its eigenvalues and the eigenvectors, which allows for example,
as we have
explained to compute the quadrature formula, while in an inverse problem, we are given the
eigenvalues and the weights of the quadrature and we have to f\/ind the matrix.
Of course the latter makes only sense if some structure is imposed on the matrix.
The problem related to what has been discussed in this paper is the inverse problem
for unitary Hessenberg matrices (see \cite[Section 8]{MISCCG02} or much earlier papers like \cite{LDRAGR91,IRAH95,LDRGra82}).
The methods described there are restricted to orthogonal polynomials. The more general recursion for orthogonal Laurent polynomials of \cite{MISCBD07}
can be used to construct more general snake shaped matrices, including Hessenberg and CMV matrices.
The algorithms to solve these problems apply the recurrence relation for the orthogonal (Laurent) polynomials in reverse order, to descend from degree $n$ to degree 0.
These inverse problems relate to discrete least squares approximation.
This is described in a linearized form in \cite{ArtVBB92d}, but proper rational approximation with
preassigned poles could be
applied when not working with vector polynomials but with ORFs on the unit circle.
A proposal is made in \cite{LDRVBFGM02,LDRVBFGM04}, but the ORFs are dif\/ferent from ours.
We explain the dif\/ference in Section~\ref{ssec12b} below.

The direct methods apply the recursion in the forward direction.
The idea is simple: the power method for a matrix $A$ computes vectors $v_k=A^kv$ for some vector $v=v_0$ and $k=0,1,2,\ldots,q$. This will with proper normalization converge to an eigenvector corresponding to the
dominant eigenvalue, while the inverse power method computes $v_{-k}=A^{-k}v$ for $k=0,1,2,\ldots,p$ and
this will approximate the dominant eigenvalue of $A^{-1}$ which is the smallest eigenvalue of $A$.
For the spaces spanned by $\{v_k\}_{k=-p}^q$ one wants to f\/ind an orthogonal basis to project the
problem on a~space that is usually much smaller than the size of the matrix $A$.
In polynomial terms, this means f\/inding a set of orthogonal (Laurent) polynomials spanning the space generated by~$\{z^k\}_{k=-p}^q$.
In each step of the algorithm either $p$ or $q$ is increased by 1. We are thus in the case considered
in \cite{MISCBD07}, i.e., the sequence $\seq{\gamma}$ has only entries 0 and $\infty$.
For a unitary matrix, the eigenvalues are all on the unit circle, and then the poles of the Laurent polynomials 0 and $\infty$ are the
points inside and outside the circle that are as far as possible away from the boundary.
So one may not expect the best possible approximation of the spectrum when considering a~unitary truncation
of the Hessenberg or general snake shaped matrix. Therefore it is essential to select poles at strategic places much closer to the circle to put more emphasis on regions where the eigenvalues are more
important. The ef\/fect of the location of the poles can be observed for example in the numerical experiments reported in \cite{ArtBC08,ArtBCCB14,ArtBGHN09b}.

When the matrix is not unitary, then the underlying theory is for polynomials and ORFs that are orthogonal
with respect to a measure supported on the real line or a real interval and
convergence has been studied based on the asymptotic behaviour of the ORFs and the approximation error
(see for example \cite{LANBGV10,MISCCDVB10}).
For the unitary case such a theory is much less developed.
One of the reasons may be that the link with the ORFs as explained in this paper is not clear.
Part of the reason is that the onset in the literature is slightly dif\/ferent
since the Hessenberg matrix is replaced
by its inverse and an extra diagonal was introduced. Since a~unitary Hessenberg can be written as
the product of the elementary $G_k$ factors, its inverse is also the product of (inverse) $G_k$ factors
and therefore computationally as simple as a Hessenberg matrix.
The study of the inverse of a Hessenberg has a long history.
See the many references in~\cite{BKVBVB07,BKVBVBM08}.
Such an inverse is known to be related to semi-separable matrices.
That are matrices in which every rectangular block that can be taken from its lower (or upper) triangular part have a low rank. Several more detailed variants of the def\/inition exist in the literature
but a~simple variant goes as follows. A~matrix is lower (upper-) semi-separable of semi-separability rank $r$
if all submatrices which can be chosen out of the lower (upper) triangular part of the matrix
have a rank at most $r$ and rank $r$ is reached for at least one of them.
They can be characterized with few (say $O(n)$) parameters. For a simple example take two vectors
$v,w\in\mathbb{C}^n$ and form a~matrix $A$ whose upper triangular part (including the diagonal)
is equal to the upper triangular part of $vw^*$ and whose strictly lower triangular part is equal
to the strictly lower triangular part of $wv^*$, then $A$ is a semi-separable matrix of rank 1
(Hermitian, except for the diagonal which may not be real).
The vectors $v$ and $w$ are called its generators. The generators and the ranks for upper and lower
triangular parts may dif\/fer in general and several other more general def\/initions exist.
The inverse of an invertible lower semi-separable matrix of rank 1 is an upper Hessenberg matrix.
Any Hermitian matrix is unitary similar to a tridiagonal matrix, and the inverse of a tridiagonal
is again a semi-separable matrix.
That explains why the literature for unitary and Hermitian eigenvalue problems and inverse eigenvalue problems are involved with
semi-separable (plus-diagonal) matrices.
The extra diagonal appears because these papers do not rely on the M\"obius transform
$\tilde{\zeta}(\G)$ that we used above but they use a less symmetric framework.
On the other hand, when the symmetry with respect to the unit circle is lost,
the poles can be anywhere in the complex plane except at inf\/inity which means that the important
case of (Laurent) polynomials is excluded.

What we have developed in this paper is however f\/lexible enough to give up the symmetry with respect to the
unit circle too but without excluding the point at inf\/inity.
This is illustrated in the next section. In the subsequent one we shall explain why in the literature
the semi-separable-plus-diagonal algorithms are used.

\subsection{Using the general recursion while dealing with inf\/inity}\label{ssec12b}
Suppose $\seq{\gamma}=\seq{\alpha}$, then our derivation given in Sections~\ref{secHB} and \ref{secMF}
shows that
\begin{gather*}
\Phi(\I-\A^*z)\hat{\cH}_\alpha=\Phi(z\I-\A)
\end{gather*}
with $\hat{\cH}_\alpha$ an upper Hessenberg matrix.
It is irreducible because the subdiagonal elements are proportional to $\sqrt{1-|\lambda|^2}$.
If we go through the derivation, then the same arguments used for the sequence $\seq{\alpha}$
also apply when we do exactly the same steps using the recursion of Theorem~\ref{thmrec} for the general sequence $\seq{\gamma}$, where we
assume for the moment that $\gamma_i=\infty$ does not appear. Then we shall again arrive at the
above relation, except that all $\alpha$'s are replaced by $\gamma$'s. Thus (assume for simplicity that
the whole sequence is regular) and recall that all $|\gamma_k|\ne1$ then
\begin{gather}\label{eqHPgam}
\Phi(\I-\varGamma^*z)\hat{\cH}_\gamma=\Phi(z\I-\varGamma)
\qquad\text{and}\qquad
z\I=(\hat{\cH}_\gamma+\varGamma)\big(\I+\varGamma^*\hat{\cH}_\gamma\big)^{-1}
\end{gather}
where $\varGamma=\diag(\gamma_0,\gamma_1,\ldots)$ and $\eta_\gamma=(\I-\varGamma^*\varGamma)^{1/2}$ in $\hat{\cH}_\gamma=\eta_\gamma^{-1}\cH_\gamma\eta_\gamma$ which is again upper Hessenberg.
An important note is in order here. The previous notation is purely formal. The expressions
will involve quantities $1-|\lambda_k|^2$ and $1-|\gamma_k|^2$ that can be negative so that their
 square roots, for example in the def\/inition of
$\eta_\gamma$ is problematic. However, remember the relation (\ref{eqe2}), which was the consequence
of the fact that $1-|\lambda_n|^2$ can only be negative if $(1-|\gamma_n|^2)(1-|\gamma_{n-1}|^2)<0$.
Therefore bringing the factor $\eta_\gamma$ and $\eta_\gamma^{-1}$ inside the $G$-factors will
give square roots of positive elements. Just note that
\begin{gather*}
\left[\begin{matrix} 1/\sqrt{1-|\gamma_{n-1}|^2} & \\ & 1/\sqrt{1-|\gamma_{n}|^2} \end{matrix}\right]
\left[\begin{matrix} -\lambda_n\eta_{n1} & \sqrt{1-|\lambda_n|^2}\\\sqrt{1-|\lambda_n|^2} & \overline{\lambda}_n\overline{\eta}_{n1} \end{matrix}\right]
\left[\begin{matrix} \sqrt{1-|\gamma_{n-1}|^2} & \\ & \sqrt{1-|\gamma_{n}|^2} \end{matrix}\right]\\[2mm]
\qquad{}=
\left[\begin{matrix} -\lambda_n\eta_{n1} & \sqrt{\dfrac{(1-|\lambda_n|^2)(1-|\gamma_{n}|^2)}{(1-|\gamma_{n-1}|^2)}}\\
\sqrt{\dfrac{(1-|\lambda_n|^2)(1-|\gamma_{n-1}|^2)}{(1-|\gamma_n|^2)}} & \overline{\lambda}_n\overline{\eta}_{n1} \end{matrix}\right]
\end{gather*}
and all square roots are taken from positive numbers.
To keep the link with what was done for the $\seq{\alpha}$ sequence,
we used the previous notation and we shall continue doing that in the sequel.
In any case the representation $\hat{\cH}_\gamma$ in this setting is a Hessenberg
matrix that can be computed in a proper way and we do not need a more general snake shaped matrix.

By doing this we loose the possibility to include $\gamma_k=\infty$.
In our previous setting this was not a problem because we reduced everything to the $\alpha$ parameters
so that $\gamma_k=\infty$ was easily dealt with by mapping it to
$\alpha_k=0$ and this f\/itted seamlessly in the
factors $\I-\A^*z$ and $z\I-\A$.
In the current general setting however a choice $\gamma_k=\infty$ requires the diagonal element
$\varpi_k(z)=1-\overline{\gamma}_kz$ to be replaced by $-z$ and
similarly $\varpi_k^*(z)=z-\gamma_k$ will become $-1$ which means that some of the entries in the
$\I$ matrix of (\ref{eqHPgam}) have to be replaced by zero and the second relation does not hold anymore.
To avoid this, we can use the strategy that was used before when switching between the formalism
for the $\seq{\alpha}$ and for the $\seq{\beta}$ sequence, except that we now use this strategy
to switch between a f\/inite and an inf\/inite entry in the $\seq{\gamma}$ sequence.

Without going through all the details, the duality between $\seq{\alpha}$ and $\seq{\beta}$
can be transferred to a duality between $\set{\gamma}=\{\gamma_1,\gamma_2,\ldots\}\subset\hat{\CC}\setminus\TT$
and the reciprocal sequence $\seq{\check\gamma}=(\check\gamma_1,\check\gamma_2,\ldots)$
with $\check\gamma_k=1/\overline{\gamma}_k$, $k=1,2,\ldots$.
The ORF for the sequence $\seq{\gamma}$ we denote as $\phi_n$ and the ORF for the sequence $\seq{\check\gamma}$ we denote as $\check{\phi}_n$.
If $\seq{\nu}=(\nu_1,\nu_2,\ldots)$ is a sequence that picks $\nu_k=\gamma_k$ if $\gamma_k$ is f\/inite
and $\nu_k=\check{\gamma}_k=0$ if $\gamma_k=\infty$,
then, in analogy with Theorem~\ref{thm1}, with proper normalization, the ORFs for this $\seq{\nu}$ sequence, which are denoted as $\phi_k^\nu$ will be given by
\begin{gather*}
\phi_n^\nu=\phi_n \check{B}_n\qquad\text{if}\quad \nu_n=\gamma_n
\qquad\text{and}\qquad
\phi_n^\nu=(\phi_n)^*\check{B}_n\qquad\text{if}\quad \nu_n=\check\gamma_n,
\end{gather*}
where
\begin{gather*}
\check{B}_n=\prod_{\substack{k=1\\\nu_k=\check\gamma_k}}^n \zeta_k^{\nu}.
\end{gather*}

First note that as long as we are dealing with f\/inite $\gamma_k$'s we can use the recurrence relation
and write the analog of the relation given in Theorem~\ref{thmSA1}, just replacing the $\alpha$'s
by $\gamma$'s (which by our convention means to remove the superscripts) and write
\begin{gather}\label{eq71b}
\left[\begin{matrix}\varpi_{n-1}^*\phi_{n-1} &\varpi_{n}\phi_{n}^{*}\end{matrix}\right] =
\left[\begin{matrix}\varpi_{n-1}\phi_{n-1}^{*} &\varpi_{n}\phi_{n}\end{matrix}\right] \hat{G}_n^\gamma
\end{gather}
with
\begin{gather*}
\hat{G}_n^\gamma= (N_n^\gamma)^{-1}\tilde{G}_n^\gamma (N_n^\gamma),\qquad
N_n^\gamma=
\left[\begin{matrix}\big(1-|\gamma_{n-1}|^2\big)^{-1/2} & 0\\
0 & \big(1-|\gamma_{n}|^2\big)^{-1/2}\end{matrix}\right],
\\
\tilde{G}_n^\gamma=
\overline{\sigma}_{n-1}\overline{\eta}_{n1}
\left[\begin{matrix}
-\lambda_n\eta_{n1} & \sqrt{1-|\lambda_n|^2}\\
\sqrt{1-|\lambda_n|^2} & \overline{\lambda_n}\overline{\eta}_{n1}
\end{matrix}\right]
\left[\begin{matrix}
1 & 0\\0 & \sigma_n
\end{matrix}\right].
\end{gather*}

Now, if some $\gamma_k=\infty$ is involved, we shall switch to the $\seq{\check{\gamma}}$ sequence,
which means that we avoid $\gamma_k=\infty$ and replace it by $\check{\gamma}_k=0$.
The factor $\check{B}_n^\gamma$ will thus only pick up a $\zeta_k^\nu=1/z$ each time a $\gamma_k=\infty$,
hence when $\nu_k=0$.
\begin{gather*}
\check{B}_n(z)=\prod_{\substack{j=1\\\gamma_j=\infty}}^n z=\prod_{j\in\mathbbm{i}_n} z = z^{|\mathbbm{i}_n|},
\end{gather*}
where $\mathbbm{i}_n=\{j\colon \gamma_j=\infty, 1\le j\le n\}$.
Since in the previous relations, as long as $\gamma_n\ne\infty$, we have $\check{B}_{n-1}=\check{B}_n$ so that it is no problem to write
\begin{gather*}
\left[\begin{matrix}\varpi_{n-1}\check{B}_{n-1}\phi_{n-1} &\varpi_{n}\check{B}_n\phi_{n}^{*}\end{matrix}\right] =
\left[\begin{matrix}\varpi_{n-1}\check{B}_{n-1}\phi_{n-1}^{*} &\varpi_{n}\check{B}_n\phi_{n}\end{matrix}\right] \hat{G}_n^\nu
\end{gather*}
with $\hat{G}_n^\nu=\hat{G}_n^\gamma$.
If however $\gamma_n=\infty$, then $\varpi_n=-z$ and this $z$-factor we absorb in $\check{B}_n$.
This is to say that $\varpi_n\check{B}_{n-1}=\varpi_n^*\check{B}_n$. Hence
\begin{gather}
 \check{B}_{n-1}\left[\begin{matrix}\varpi_{n-1}\phi_{n-1}^{*} &\varpi_{n}\phi_{n}\end{matrix}\right] \hat{G}_n^\nu
= \check{B}_{n-1}\left[\begin{matrix}\varpi_{n-1}\phi_{n-1} &\varpi_{n}\phi_{n}^{*}\end{matrix}\right]\nonumber\\
\hphantom{\check{B}_{n-1}\left[\begin{matrix}\varpi_{n-1}\phi_{n-1}^{*} &\varpi_{n}\phi_{n}\end{matrix}\right] \hat{G}_n^\nu}{}
= \left[\begin{matrix}\varpi_{n-1}\check{B}_{n-1}\phi_{n-1} &\varpi_{n}^*\check{B}_n\phi_{n}^*\end{matrix}\right].\label{eq72c}
\end{gather}
Now $\hat{G}_n^\nu$ is like $\hat{G}_n^\gamma$ but with $\gamma_n=\infty$ replaced by 0.
We can proceed exactly as in the case of the $\alpha$-$\beta$ duality, but now apply it to the
f\/inite-inf\/inite duality.
Thus using $\nu_k=\gamma_k$ if $\gamma_k$ is f\/inite and $\nu_k=0$ when $\gamma_k=\infty$ and setting
\begin{gather}\label{eqcalGP}
\hat{\G}=\eta_\N^{-1}\G\eta_\N, \qquad \eta_\N=(\I-\N^*\N)^{1/2},\qquad \N=\diag(\nu_0,\nu_1,\nu_2,\ldots),
\end{gather}
then the analog of (\ref{eqGblocks}) becomes
(all the $\gamma_i$'s denoted explicitly as such are f\/inite values and denoted as $\infty$ when not f\/inite)
\begin{gather*}
\seq{\gamma} = (\gamma_0,\gamma_1,\ldots,\gamma_{n-1}~\|\;\infty,\ldots,\;\infty|\;\gamma_k\;\|\;\gamma_{k+1},\ldots,\;\gamma_{m-1}\|\;\infty\;,\ldots,\;\infty|\;\gamma_\ell\|\gamma_{\ell+1},\ldots ), \\
\seq{\nu} = (\gamma_0,\gamma_1,\ldots,\gamma_{n-1}~\|\;~0~,\ldots,\;~0|\;\gamma_k\;\|\;\gamma_{k+1},\ldots,\;\gamma_{m-1}\|\;~0~\;,\ldots,\;~0|\;\gamma_\ell\|\gamma_{\ell+1},\ldots ), \\
\hat{\G}= \underbrace{\big(\hat{G}_1^\gamma\hat{G}_2^\gamma\cdots \hat{G}_{n-1}^\gamma\big)}_{\hat{G}_\gamma^1}
\underbrace{\big(\hat{G}_{k}^\gamma\hat{G}_{k-1}^\nu\cdots \hat{G}_n^\nu\big)}_{\hat{G}_\infty^1}
\underbrace{\big(\hat{G}_{k+1}^\gamma\cdots \hat{G}_{m-1}^\gamma\big)}_{\hat{G}_\gamma^2}
\underbrace{\big(\hat{G}_{\ell}^\gamma\hat{G}_{\ell-1}^\nu\cdots \hat{G}_m^\nu\big)}_{\hat{G}_\infty^2}
\big(\hat{G}_{\ell+1}^\gamma\cdots\big)\\
\hphantom{\hat{\G}}{}=
\eta_\N^{-1}
\underbrace{\big({G}_1^\gamma{G}_2^\gamma\cdots {G}_{n-1}^\gamma\big)}_{{G}_\gamma^1}\!
\underbrace{\big({G}_{k}^\gamma{G}_{k-1}^\nu\cdots {G}_n^\nu\big)}_{{G}_\infty^1}\!
\underbrace{\big({G}_{k+1}^\gamma\cdots {G}_{m-1}^\gamma\big)}_{{G}_\gamma^2}\!
\underbrace{\big({G}_{\ell}^\gamma{G}_{\ell-1}^\nu\cdots {G}_m^\nu\big)}_{{G}_\infty^2}\!
\big({G}_{\ell+1}^\gamma\cdots\big)
\eta_\N.
\end{gather*}
Thus in the $\gamma$ blocks we multiply the successive factors $G_i^\gamma$ to the right while in
an $\infty$ block we multiply in reversed order like we did in the $\alpha$-$\beta$ case.

Use (\ref{eq71b}) and multiply $\Phi^\nu(\I-\N^*z)$ from the right with $\hat{G}_\gamma^1$ to get
(note $\phi_0^\nu=\phi_0^{\nu*}=1$ and $\check{B}_{n-1}=1$)
\begin{gather*}
\big[
\varpi_0^{*}\phi_0,\ldots,\varpi_{n-2}^{*}\phi_{n-2},
\varpi_{n-1}\phi_{n-1}^{*}\check{B}_{n-1}\big|\varpi_n\phi_n^*\check{B}_n,\ldots
\big].
\end{gather*}
While multiplying this result with $\hat{G}_k^\gamma \hat{G}_{k-1}^\nu\cdots \hat{G}_{n+1}^\nu$, we make use of (\ref{eq71b}) and (\ref{eq72c}) to obtain
\begin{gather*}
\big[
\varpi_0^{*}\phi_0,\ldots,\varpi_{n-2}^{*}\phi_{n-2},
\varpi_{n-1}\phi_{n-1}^{*}\check{B}_{n-1}\big|\\
\qquad{}\varpi_{n}^{*}\phi_{n}\check{B}_n,\varpi_{n+1}^{*}\phi_{n+1},\ldots,\varpi_{k-1}^{*}\phi_{k-1}\big|
\varpi_k\phi_k^*,\varpi_{k+1}\phi_{k+1},\ldots\big]
\end{gather*}
and the remaining multiplication $\hat{G}_n^\nu$ links the $\gamma$ block to the $\infty$ block,
so that after reintroducing the $\phi_i^\nu$ notation:
\begin{gather*}
\big[
\varpi_0^{*}\phi_0^\nu,\ldots,\varpi_{n-1}^{*}\phi_{n-1}^\nu\big|
\varpi_{n}^{*}\phi_{n}^\nu,\ldots,\varpi_{k-1}^{*}\phi_{k-1}^\nu\big|
\varpi_k\phi_k^{\nu*},\varpi_{k+1}\phi_{k+1}^\nu,\ldots\big].
\end{gather*}
The next block is again an $\gamma$ block treated by the product $\hat{G}_\gamma^2$,
and one may continue like this to f\/inally get
\begin{gather}\label{eq73b}
\Phi^\nu(\I-\N^*z)\hat{\G}=\Phi^\nu(z\I-\N).
\end{gather}
Note that to include $\gamma_k=\infty$ in the sequence, we had to give up the Hessenberg structure of~$\cH_\gamma$ in~(\ref{eqHPgam}) because the $G_k$ factors corresponding to inf\/inite $\gamma$'s are multiplied in reverse order. Thus we again end up with a snake shaped matrix~$\G$.
It like in Fig.~\ref{fig1} with upper Hessenberg blocks for the blocks with
f\/inite $\gamma$'s and lower Hessenberg blocks for the sequences corresponding to the blocks of
inf\/inite $\gamma$'s.

This might seem to be a relatively complicated way of dealing with a pole at the origin, which was included in
a more elegant way in the original setting, but this is the way how it is dealt with in the papers such as~\cite{MISCMVBVB13} and others. The matrix obtained there is an upper Hessenberg matrix with blocks that
bulge below the subdiagonal at those places where in our structure of Fig.~\ref{fig1} we also have a lower Hessenberg block.

Note that whatever the shape of the matrix is, in the rational case, one should analyse the
spectrum of its
M\"obius transform as in (\ref{eqHPgam}) or analyse it as a pencil.
This can be avoided at the expense of a somewhat more complicated matrix problem
in the form of a semi-separable-plus-diagonal matrix.

\subsection{Alternative: semi-separable-plus-diagonal matrices}

If we start from the f\/irst relation in (\ref{eqHPgam}) in which $\hat{\cH}_\gamma$ is upper Hessenberg we can see that
the $n$th element on the right, i.e., $\varpi_n^{*}(z)\phi_n(z)$
is written as a linear combination of the elements
\begin{gather*}
[\phi_0,\varpi_1\phi_1,\ldots,\varpi_{n+1}\phi_{n+1}]=
\left[1,p_1,\frac{p_2}{\pi_1},\frac{p_3}{\pi_2},\ldots,\frac{p_{n+1}}{\pi_n}\right],\\
\phi_k=\frac{p_k}{\pi_k},\qquad p_k\in\P_k.
\end{gather*}
These elements will generate all elements of the form $\frac{q_{n+1}}{\pi_n}$ with $q_{n+1}$
a polynomial of degree at most $n+1$.
It is clear that also $\phi_n$ belongs to that space.
Hence there must exist an upper Hessenberg matrix $\cH$ such that
\begin{gather*}
\Phi(\I-\varGamma^*z)\cH=\Phi.
\end{gather*}
A f\/inite truncation then is
\begin{gather}\label{eq122a}
\Phi_n(\I_n-\Gamma_n^*z)\cH_n=\Phi_n+c_{n+1}[0,\ldots,0,\varpi_{n+1}\phi_{n+1}]
\end{gather}
with $c_{n+1}$ a constant, $\I_n$ the identity matrix of size $n+1$ and $\Gamma_n^*=\diag(\overline{\gamma}_0,\ldots,\overline{\gamma}_n)$.
If $z$ is a~zero of $\phi_{n+1}$, then $\Phi_n(\I_n-\Gamma_n^*z)\cH_n=\Phi_n$ and this can be rearranged as
\begin{gather*}
z\Phi_n=\Phi_n\big(\I_n-\cH_n^{-1}\big)\check{\Gamma}_n
\end{gather*}
with $\check{\Gamma}_n=(\Gamma_n^{*})^{-1}=\diag(\check{\gamma}_0,\ldots,\check{\gamma}_n)$
where $\check{\gamma}_k=\gamma_{k*}=1/\overline{\gamma}_k$.
This clearly requires the restrictive conditions that $\cH_n$ and $\Gamma_n$ are invertible,
thus $\gamma_k=0$ is excluded. It shows that $z$ is and eigenvalue of the matrix $(\I_n-\cH_n^{-1})\check{\Gamma}_n$ and that the left eigenvector is $[\phi_0(z),\ldots,\phi_n(z)]$.
By a~similar argument if $\cH_n$ is replaced by $\cH_n^u$, a unitary truncation of $\cH$
then the eigenvalue~$z$ will be on~$\TT$.

This is the kind of arguments used in for example \cite{LDRVBFGM02,LDRVBFGM04,MISCMVBVB14} and others.
They work with a~slightly simplif\/ied formula because the ORF have denominators $\pi_n^*$ instead of our
$\pi_n$. If poles at~0 and~$\infty$ are excluded, this is simple to obtain.
Suppose $\check{\phi}_n=\frac{p_n}{\pi_n^*}$, $n=0,1,\ldots$, then (\ref{eq122a}) becomes
\begin{gather*}
\check{\Phi}_n\big(z\I_n-\check{\Gamma}_n\big)\check{\cH}_n=\check{\Phi}_n+\check{c}_{n+1}\big[0,\ldots,0,\varpi_{n+1}\check\phi_{n+1}\big].
\end{gather*}
Thus if $\check{z}$ is a zero of $\check\phi_{n+1}$, then
\begin{gather*}
\check{z}\check{\Phi}_n=\check\Phi_n\big(\check{\cH}_n^{-1}+\check{\Gamma}_n\big).
\end{gather*}
This shows that $\check{z}$ is an eigenvalue of $\check{\cH}_n^{-1}+\check{\Gamma}_n$ and the corresponding
left eigenvector is $\check{\Phi}_n(\check{z})=[\check{\phi}_0(\check{z}),\ldots,\check{\phi}_n(\check{z})]$.
This $\check{\cH}_n^{-1}+\check{\Gamma}_n$ is the semi-separable-plus-diagonal matrix that we
mentioned before.
This is the relation given in \cite[Theorem~2.7]{LDRVBFGM04} where it was obtained in
the context of an inverse eigenvalue problem. See also \cite{LDRFG03,LDRVBFGM02} which also rely
on this structure.

The matrix $\check{\cH}_n^{-1}$ is a semi-separable matrix, as the inverse of an unreducible Hessenberg matrix,
which is as structured as the Hessenberg matrix itself.
Indeed, if $\check{\cH}_n$ can be written as a product of unitary $G_k$ factors, except for the last one
which is truncated, then the semi-separable $\cH_n^{-1}$ is the product
of $G_k^{-1}$ in reverse order. Obviously, this has computationally the same complexity.

Since this approach allows only for f\/inite poles, this theory was later extended in
\cite{MISCMVBVB14} which is basically using the technique explained in the previous section to include
$\infty$ in the algorithm. The Hessenberg matrix with bulges below the diagonal is then called an
extended Hessenberg matrix and the associated Krylov algorithms are called extended Krylov methods.

A selection of poles in an arbitrary order is implemented in this way
to solve inverse and direct eigenvalue problems, QR factorization of unitary (Hessenberg) matrices,
rational Krylov type methods, rational Szeg\H{o} quadrature, and rational approximation problems
like in \cite{MISCAMVW15,LDRFG07,LDRGem05,MISCMVBVB13,MISCMVBVB14,MISCMer16,MISCMV14}.
We hope that future publications will illustrate that with the approach from the f\/irst
10 sections of this paper we have provided a more elegant framework to solve all these problems.

\section{Conclusion}\label{secconclude}
We have systematically developed the basics about orthogonal rational functions
on the unit circle whose poles can be freely chosen anywhere in the complex plane
as long as they are not on the unit circle $\TT=\{z\in\CC\colon |z|=1\}$.
The traditional cases of all poles outside the closed disk and the so-called
balanced situations where the poles alternate between inside and outside the disk
are included as special instances. The case where all the poles are inside the disk
is somewhat less manageable because it is not so easy to deal with the
poles at inf\/inity in this formalism, but it has been included anyway.
The link between the ORF with all poles outside, all poles inside, and the
general case with arbitrary location of the poles is clearly outlined.
It is important that previous attempts to consider the general situation
were assuming that once some pole $1/\overline{\gamma}$ is chosen, then
$\gamma$ may not appear anywhere in the sequence of remaining poles, which
would for example exclude the balanced situation.
We have shown how the classical Christof\/fel--Darboux formula,
the recurrence relation, and the relation to rational Szeg\H{o} formulas
can be generalized.
Finally we analyzed the matrix representation of the shift operator
with respect to the general basis as a matrix M\"obius transform of a
generalization of a snake-shape matrix,
and how the latter can be factorized as a
product of elementary unitary $2\times2$ blocks.
It is shown that there is no theoretical or computational advantage to compute
rational Szeg\H{o} quadrature formulas for a general sequence of poles.
It is however conceptually slightly simpler to consider the balanced
situation, i.e., the generalization of the CMV representation.
In the last two sections we have made a link with the linear algebra literature.
In Section~\ref{secAMPD} we have given an illustration by proving some properties of AMPD matrices that are
directly inspired by the ORF recursion.
A~relation with direct and inverse eigenvalue problems for unitary matrices
and the related rational Krylov methods is brief\/ly discussed in Section~\ref{secAF} illustrating
that the approach given in this paper is more general and hence might therefore be
a more appropriate choice.

As a f\/inal note we remark that a very similar approach can be taken for ORF on the real line
or a real interval
in which case the unitary Hessenberg matrices will be replaced by tridiagonal Jacobi matrices
and their generalizations in case poles are introduced in arbitrary order taken from anywhere in the
complex plane excluding the real line or outside the interval are in order.

\subsection*{Acknowledgements}
We thank the anonymous referees for their careful reading of the manuscript and their suggestions for improvement.

\pdfbookmark[1]{References}{ref}
\LastPageEnding

\end{document}